\documentclass[11pt,leqno]{article}
\usepackage{times}
\usepackage{tgtermes}
\usepackage[T1]{fontenc}
\usepackage[utf8]{inputenc}
\usepackage[english]{babel}
\usepackage{fancyhdr}
\usepackage{setspace}
\usepackage{titlesec}
\usepackage{cite}
\titleformat{\section}
{\normalfont\fontsize{16}{20}\bfseries}{\thesection}{1em}{}
\titleformat{\subsection}
{\normalfont\fontsize{16}{20}\bfseries}{\thesubsection}{1em}{}
\titlespacing*{\section}
{0pt}{2.ex plus 1ex minus .2ex}{.0ex}
\titlespacing*{\subsection}
{0pt}{0.ex}{.0ex}
\usepackage{cancel}
\newcommand{\ud}{\underline}
\newcommand{\R}{\mathbb R} 
\newcommand{\dis}{\displaystyle}
\usepackage{amssymb}
\usepackage{amsfonts}
\usepackage{amsmath}
\numberwithin{equation}{section}
\usepackage{amsthm}
\usepackage{mathtools}
\usepackage{multirow}
\usepackage{cancel}
\usepackage[
left=25mm,
right=25mm,
top=25mm,
bottom=25mm,
]{geometry}

\usepackage{float}
\numberwithin{figure}{section} 
\numberwithin{table}{section} 

\newtheorem{theo}{Theorem}[section]
\newtheorem{prop}{Proposition}[section]

\newtheorem{rem}{Remark}[section]

\usepackage{stmaryrd}

\usepackage{xcolor}
\usepackage{tikz}
\usepackage{indentfirst} 

\usepackage{hyperref}
\usepackage{xcolor} % For colors
\usepackage{tcolorbox}
\usepackage{caption}
\usepackage{float}
\usepackage{booktabs}
\usepackage{authblk}

\title{On some 1D nonlocal models with coefficients changing sign}
\author{Maha DAOUD}
\affil{ Département Génie Mathématique et Modélisation, INSA Toulouse, 31400 Toulouse, France\\\vspace*{.15cm} mahaadaoud@gmail.com}

\begin{document}

	\maketitle
\begin{abstract}
In this work,	we study one-dimensional nonlocal elliptic transmission problems with piecewise constant coefficients that may change sign across an interface. In the local setting, we recall the T-coercive structure of the problem and characterize the critical contrast case. In the nonlocal setting, we focus on a simplified configuration in which the cross-interaction coefficient vanishes. Under this assumption, we prove a weak T-coercivity result for the global fractional problem and introduce a reconstructed formulation based on an explicit interface lifting. Then, we consider a simplified finite element discretization of the reconstructed model and prove its convergence toward the classical local transmission problem as the fractional parameter $s\to 1^-$ and the mesh size $h\to 0^+$. Numerical simulations in 1D illustrate the stability and consistency of the method, and a preliminary two-dimensional extension is presented as an exploratory perspective.
	
		\smallskip
	\noindent \textbf{Keywords:} Nonlocal transmission problem, Sign-changing coefficients, T-coercivity, Fractional Laplacian, Interface reconstruction, Finite element method, Local limit, Asymptotic convergence.
	
	\smallskip
	\noindent\textbf{2020 Mathematics Subject Classification:} 35R11, 65N30, 65N12, 35B40.
\end{abstract}
\section{Introduction}

Nonlocal diffusion models have attracted considerable attention in
	recent years, both because of their mathematical richness and because of their
	ability to describe long-range interactions and anomalous diffusion. They arise
	in a wide range of applications and have given rise to a vast literature, from
	the functional framework of fractional Sobolev spaces
	\cite{DiNezzaPalatucciValdinoci,Leoni,BourgainBrezisMironescu,MazyaShaposhnikova}
	to well-posedness and regularity issues for fractional problems
	\cite{ServadeiValdinoci,RosOtonSerra,DaouLaamBaal2024,BucVal2016,BorLiNo},
	as well as to the numerical approximation of fractional models
	\cite{AcostaBor,AcostaBersetcheBorthagaray,BonitoLeiPasciak,BonBorNochAl,DeliaDuGlusaGunzburgerTianZhou,BicHer}.
	We also mention the recent work \cite{BorCia3} on coupled local and nonlocal
	diffusion models.%, which is closely related to the present perspective.

\smallbreak 

One of the motivations of the present work comes from interface problems arising
in electromagnetics, and more specifically from transmission problems between a
standard dielectric material and a metamaterial. In such situations, the
effective coefficients may change sign, reflecting for instance negative
permittivity or permeability in part of the domain. In the local setting, these
sign changes may lead to major analytical and numerical difficulties: depending
on the contrast and on the geometry of the interface, the associated scalar problem may fail to be coercive,  and may exhibit singularities. For further details, see
\cite{BonCheCia,BonCheCla,BonCheCiaMaxwell,BonCheCiaMaxwell2D} and their bibliographies.
Indeed, the T-coercivity approach has proved to be a powerful tool
to recover a well-posedness theory away from critical contrasts
(see \cite{BonCheCia}).

Motivated by the nonlocal reformulation proposed in \cite{BorCia1}, we consider here a
one-dimensional nonlocal elliptic model involving a piecewise constant
interaction coefficient, allowed to change sign according to the
location of the interacting points.  To the best of our knowledge, a general
well-posedness theory for such nonlocal models is still unavailable.

\smallbreak
\noindent Our purpose is twofold. First, we clarify the structure of the
local transmission problem, which serves as a reference model.
Second, we analyze a simplified nonlocal formulation and derive from it a new
reconstructed model that is better suited both for the analysis and for the
numerical approximation.

\smallbreak \noindent
More precisely, for the local  problem, we establish Hadamard
well-posedness away from a critical contrast ratio. In the critical case, we
identify a nontrivial kernel generated by a function $\varphi$. This local
analysis motivates the form of
the lifting used in the nonlocal setting.

\smallbreak 
For the nonlocal model, we focus on a simplified configuration in which the
cross-interaction coefficient is removed. Under this
assumption, we prove that the associated global bilinear form is weakly
T-coercive, which yields well-posedness in the Fredholm sense. Then, we introduce a reconstructed decoupled
representation of the solution of the form
$
	u^s = u_0^s + u^s(b)\varphi^s,
$
where
$
u_0^s \in \widetilde H^s(I_1)\oplus \widetilde H^s(I_2)$
\text{and} $
\varphi^s\in \widetilde H^s(I)
$
is a lifting function satisfying suitable boundary and normalization conditions.
Once the normalization $\varphi^s(b)=1$ is imposed, the interface coefficient
$u^s(b)$ is uniquely determined.

\smallbreak 
Furthermore, this decomposition provides the basis for an efficient numerical strategy.
We consider finite element discretizations of the old model, the reconstructed
model, and a simplified reconstructed model.  We then study the convergence of
the latter toward the classical local solution as $s\to1^-$. The resulting
formulation has a natural block structure: the subdomain contributions can be
treated independently, while the global interaction is recovered through a small
number of interface unknowns. The one-dimensional numerical experiments
illustrate the stability of the method, its consistency with the local limit,
and the interest of the reconstructed formulation in the presence of
sign-changing coefficients.

\smallbreak
Finally, motivated by the encouraging simulations reported in \cite{BorCia1}
and by the structure of the reconstructed discrete system, we include a
preliminary extension of the method to a simple two-dimensional setting. This
part is exploratory in nature and is intended to illustrate the potential of
the approach beyond the one-dimensional framework.

%\smallbreak 
%
%For clarity, we distinguish three levels in the analysis. First, we study the
%local transmission problem, which serves as a reference model and identifies the
%relevant interface structure. Second, we analyze a simplified global nonlocal
%formulation, obtained by suppressing the cross-interactions between the two
%subdomains, and establish a Fredholm well-posedness result for it. Third, we
%introduce a reconstructed nonlocal formulation based on the decomposition
%\eqref{decomp}
%which isolates the interface contribution through the scalar quantity $u^s(b)$.
%This reconstructed point of view is then used as the basis for the numerical
%approximation and for the convergence analysis toward the local solution as
%$s \to 1^-$.

\medskip
\noindent
\textbf{Outline of the paper.}
The remainder of the paper is organized as follows. In Section~\ref{Section2}, we introduce
the local and nonlocal models and establish the analytical framework. In
Section~\ref{Section3}, we present the finite element discretizations of the old, new, and
simplified new models. In Section~\ref{Section4}, we study the convergence of the simplified
nonlocal model toward the local one. In Section~\ref{Section5}, we report one-dimensional
numerical experiments. Finally, in Section~\ref{Section6}, we discuss a preliminary
two-dimensional extension of the method.

\section{From local to nonlocal models with sign-changing coefficients} \label{Section2}
In this section, we study three 1D problems with
piecewise constant coefficients that may change sign across an interface.
First, we recall the well-posedness theory for the corresponding local problem, which serves as a reference framework.
In this setting, the problem admits a natural interface decomposition and
the associated bilinear form is T-coercive under a sharp algebraic
condition on the coefficients.

Then, we address the nonlocal counterpart involving the fractional
Laplacian. In contrast with the local case, the presence of long-range
interactions introduces additional difficulties, in particular when
coefficients change sign. We begin by discussing an intuitive fractional formulation, and then restrict ourselves to a simplified case. In that setting, we prove weak T-coercivity for the corresponding bilinear form. Finally, we introduce a reconstructed interface formulation, which provides a convenient framework to relate the nonlocal model to the local theory as $s\to1^-$.
	\subsection{Local problem}
For a fixed $b$, let $I=(0,1)$ be decomposed into the subintervals $I_1=(0,b)$ and $I_2=(b,1)$. Given $f\in H^{-1}(I)$, let us consider the one-dimensional elliptic problem 
\begin{equation} \label{ClassicalProblem}
	\left\{   
	\begin{array}{rcl}
		-\mathrm{div}\bigl(\sigma(x)\nabla u\bigr) & = f & \text{in } I,\\
		u(0)=u(1) &= 0,
	\end{array}     
	\right.
\end{equation}
where $\sigma(x)=\sigma_1>0$ in $I_1$ and $\sigma(x)=\sigma_2<0$ in $I_2$. This local sign-changing transmission problem is classical in the framework of
T-coercivity; see, for instance, \cite{BonCheCia,BonCheCla,BonCheCiaMaxwell}.

The associated variational formulation reads
\begin{equation} \label{ClassicalVar}
	\left\{
	\begin{array}{lll}
		\text{find } \;u\in H^1_0(I)\; \text{ such that:}
		\\\dis
		a(u,v)=\int_I f\,v\,\mathrm{d}x,
		\qquad \forall v\in H^1_0(I),
	\end{array}
	\right.
\end{equation}
where
\[
a(u,v):=\int_I \sigma(x)\,u'(x)v'(x)\,\mathrm{d}x,
\]
and the corresponding operator $A\in\mathcal L(H^1_0(I),H^{-1}(I))$ is defined by
\[
\langle Au,v\rangle_{H^{-1},H^1_0} := a(u,v), \qquad \forall u,v\in H^1_0(I).
\]
Let us recall that $H^1_0(I)$ is continuously embedded into
$C^0(\overline I)$, so that pointwise evaluation at $x=b$ is well-defined.
As a consequence, any $v\in H^1_0(I)$ admits the continuous decomposition
\begin{equation}
	\label{decc}
	v = v_0 + v(b)\,\varphi,
\end{equation}
where $v_0\in H^1_0(I)$ with $v_0(b)=0$, and $\varphi\in H^1_0(I)$ is a
piecewise affine function such that $\varphi(b)=1$.
This choice uniquely determines $\varphi$, which is given by
\[
\varphi(x)=
\begin{cases}
	\dfrac{x}{b}, & x\in I_1,\\[4pt]
	\dfrac{1-x}{1-b}, & x\in I_2.
\end{cases}
\]
\begin{figure}[H]
	\centering
	\includegraphics[scale=0.225]{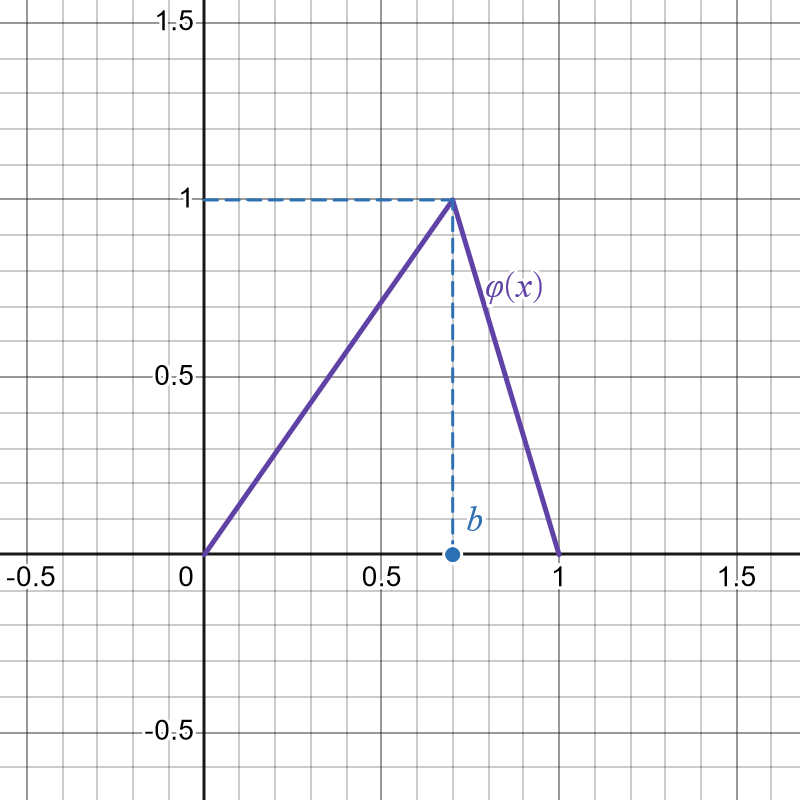}
	\caption{The function $\varphi$ for $b=0.7$}
	\label{fig:phi}
\end{figure}

The well-posedness of Problem~\eqref{ClassicalVar} in the Hadamard sense
is ensured by the T-coercivity of the bilinear form $a$.
More precisely, it suffices to find an isomorphism
$T \in \mathcal{L}(H^1_0(I))$  and $\alpha>0$ such that
\[
|a(v,Tv)| \ge \alpha \|v\|_{H^1_0(I)}^2 \qquad \forall v \in H^1_0(I).
\]

\begin{theo}[T-coercivity and characterization of the kernel]\label{thm:local-Tcoercive}
	Let $f\in H^{-1}(I)$.
	Then the following assertions hold:
	\begin{enumerate}
		\item[\em (i)] If
		$
		\dfrac{|\sigma_2|}{\sigma_1}\neq\dfrac{1-b}{b},
		$
		 the operator
	$
		A $
		is an isomorphism. In particular, Problem~\eqref{ClassicalVar} is well-posed in the
		Hadamard sense.
		\item[\em (ii)] In the critical case
	$
		\dfrac{|\sigma_2|}{\sigma_1}=\dfrac{1-b}{b},
		$
		the operator $A$ is not injective and
	$
		\ker(A)=\operatorname{span}\{\varphi\}.
	$
	\end{enumerate}
\end{theo}
\begin{proof}
The bilinear form associated with Problem~\eqref{ClassicalVar} is
	\[
	a(u,v)=\int_I \sigma(x)\,u'(x)v'(x)\,\mathrm{d}x,
	\qquad u,v\in H^1_0(I).
	\]

	\noindent {(i)} 
First, $v_0$ is $\sigma$-orthogonal to $\varphi$. Indeed, since $\varphi$ is
piecewise affine on $I_1$ and $I_2$, we have $\varphi''=0$ on each subinterval and
\[
\varphi'(x)=\frac1b \ \text{for }x\in I_1,
\qquad
\varphi'(x)=-\frac1{1-b} \ \text{for }x\in I_2.
\]
Therefore,
\[
\int_I \sigma(x)\,v_0'(x)\,\varphi'(x)\,\mathrm{d}x
=
\sigma_1\int_{0}^{b} v_0'(x)\,\frac1b\,\mathrm{d}x
-
|\sigma_2|\int_{b}^{1} v_0'(x)\,\Bigl(-\frac1{1-b}\Bigr)\,\mathrm{d}x.
\]
Computing the integrals gives
\[
\sigma_1\int_{0}^{b} v_0'(x)\,\frac1b\,\mathrm{d}x
=
\frac{\sigma_1}{b}\bigl(v_0(b)-v_0(0)\bigr),
\qquad
|\sigma_2|\int_{b}^{1} v_0'(x)\,\Bigl(-\frac1{1-b}\Bigr)\,\mathrm{d}x
=
-\frac{|\sigma_2|}{1-b}\bigl(v_0(1)-v_0(b)\bigr).
\]
Since $v_0\in H^1_0(I)$ we have $v_0(0)=v_0(1)=0$, and by construction
$v_0(b)=0$ (see \eqref{decc}). Hence both terms vanish and we obtain
\[
\int_I \sigma(x)\,v_0'(x)\,\varphi'(x)\,\mathrm{d}x = 0.
\]

	For $\varepsilon\in\{-1,1\}$, let us define the linear operator $T:H^1_0(I)\to H^1_0(I)$ by
	\[
	Tv :=
	\begin{cases}
		\phantom{-}v_0 + \varepsilon\,v(b)\varphi, & \text{in } I_1,\\
		-\,v_0 + \varepsilon\,v(b)\varphi, & \text{in } I_2.
	\end{cases}
	\]
	
By construction, $T$ belongs to $\mathcal{L}(H^1_0(I))$, and we can easily check that $T$ is bijective.

Now, let us compute $a(v,Tv)$. We have
\[
a(v,Tv)=\int_{I}\sigma(x)\,v'(x)(Tv)'(x)\,\mathrm{d}x=\sigma_1\int_{I_1}v'(x)(Tv)'(x)\,\mathrm{d}x-|\sigma_2|\int_{I_2}v'(x)(Tv)'(x)\,\mathrm{d}x,
\]
with \[
v' = v_0' + v(b)\varphi', \qquad
(Tv)'=
\begin{cases}
\phantom{-}	v_0' + \varepsilon v(b)\varphi', & \text{in } I_1,\\
	-\,v_0' + \varepsilon v(b)\varphi', & \text{in } I_2.
\end{cases}
\]
On $I_1$, we expand the product:
\[
v'(Tv)'=\bigl(v_0'+v(b)\varphi'\bigr)\bigl(v_0'+\varepsilon v(b)\varphi'\bigr)
=(v_0')^2+(\varepsilon+1)v(b)\,v_0'\varphi'
+\varepsilon v(b)^2(\varphi')^2.
\]
Thus
\[
\sigma_1\int_{I_1}v'(x)(Tv)'(x)\,\mathrm{d}x
=\sigma_1\int_{I_1}|v_0'(x)|^2\,\mathrm{d}x
+(\varepsilon+1)\sigma_1 v(b)\int_{I_1}v_0'(x)\varphi'(x)\,\mathrm{d}x
+\varepsilon\sigma_1 |v(b)|^2\int_{I_1}|\varphi'(x)|^2\,\mathrm{d}x.
\]

On $I_2$, we similarly expand:
\[
v'(Tv)'=\bigl(v_0'+v(b)\varphi'\bigr)\bigl(-v_0'+\varepsilon v(b)\varphi'\bigr)
=-(v_0')^2+(\varepsilon-1)v(b)\,v_0'\varphi'
+\varepsilon v(b)^2(\varphi')^2,
\]
and therefore
\[
|\sigma_2|\int_{I_2}v'(x)(Tv)'(x)\,\mathrm{d}x
=-|\sigma_2|\int_{I_2}|v_0'(x)|^2\,\mathrm{d}x
+(\varepsilon-1)|\sigma_2| v(b)\int_{I_2}v_0'(x)\varphi'(x)\,\mathrm{d}x
+\varepsilon|\sigma_2| |v(b)|^2\int_{I_2}|\varphi'(x)|^2\,\mathrm{d}x.
\]

We now use the $\sigma$--orthogonality of $v_0$ and $\varphi$,  so that
\[
(\varepsilon+1)\sigma_1\int_{I_1}v_0'(x)\varphi'(x)\,\mathrm{d}x
+(\varepsilon-1)|\sigma_2|\int_{I_2}v_0'(x)\varphi'(x)\,\mathrm{d}x
=0.
\]
Then, we obtain
\[
a(v,Tv)
=\sigma_1\int_{I_1} |v_0'(x)|^2\,\mathrm{d}x+|\sigma_2|\int_{I_2} |v_0'(x)|^2\,\mathrm{d}x
+\varepsilon |v(b)|^2\left(
\sigma_1\int_{I_1}|\varphi'(x)|^2\,\mathrm{d}x-|\sigma_2|\int_{I_2}|\varphi'(x)|^2\,\mathrm{d}x
\right).
\]

It remains to compute the last integrals. Since
\[
\int_{I_1}|\varphi'(x)|^2\,\mathrm{d}x=\int_0^b \frac1{b^2}\,\mathrm{d}x=\frac1b
\qquad \text{and}\qquad
\int_{I_2}|\varphi'(x)|^2\,\mathrm{d}x=\int_b^1 \frac1{(1-b)^2}\,\mathrm{d}x=\frac1{1-b},
\]
we get
\[
a(v,Tv)
\dis=\sigma_1\int_{I_1} |v_0'(x)|^2\,\mathrm{d}x+|\sigma_2|\int_{I_2} |v_0'(x)|^2\,\mathrm{d}x
+\varepsilon |v(b)|^2\left(\frac{\sigma_1}{b}-\frac{|\sigma_2|}{1-b}\right).
\]
In addition, we have
\begin{equation}
	\sigma_1\int_{I_1} |v_0'(x)|^2\,\mathrm{d}x+|\sigma_2|\int_{I_2} |v_0'(x)|^2\,\mathrm{d}x\geq \min\{	\sigma_1,|\sigma_2|\} \int_{I} |v_0'(x)|^2\,\mathrm{d}x = \min\{	\sigma_1,|\sigma_2|\} \| v_0'\|_{L^2(I)}^2.
\end{equation}

	At this stage, we choose $\varepsilon\in\{-1,1\}$ such that 
	$$\varepsilon \left(\frac{\sigma_1}{b}-\frac{|\sigma_2|}{1-b}\right)>0\Leftrightarrow \varepsilon  \left(1 - \dfrac{|\sigma_2|}{\sigma_1} \dfrac{b}{1-b} \right)>0\Leftrightarrow \varepsilon=\operatorname{sign}\left( 1 - \dfrac{|\sigma_2|}{\sigma_1} \dfrac{b}{1-b}\right).$$

Let us assume that $
\dfrac{|\sigma_2|}{\sigma_1}\neq\dfrac{1-b}{b}.
$ Then, there exists $c>0$ such that
	\[
	|a(v,Tv)|\ge
c\bigl(\|v_0'\|_{L^2(I)}^2+|v(b)|^2\bigr).
	\]
	Since \(v = v_0 + v(b)\varphi\) with \(v_0(b)=0\), the quantity
$
	\|v_0'\|_{L^2(I)}^2 + |v(b)|^2
$
	defines a norm equivalent to
$
	\|v\|_{H_0^1(I)}^2.
$
	Hence, there exists \(\alpha>0\) such that
	\[
	\|v_0'\|_{L^2(I)}^2 + |v(b)|^2 \ge \alpha\,\|v\|_{H_0^1(I)}^2,
	\]
	which implies that $a$ is T-coercive.
	% In particular, the operator $A$ is
%	injective, and Problem~\eqref{ClassicalVar} is well-posed in the Hadamard sense.
	Since the bilinear form $a$ is continuous and $T$-coercive on $H_0^1(I)$, the
	associated operator $A : H_0^1(I) \to H^{-1}(I)$ is an isomorphism. Therefore,
	Problem~\eqref{ClassicalVar} is well-posed in the Hadamard sense.
	\medskip
	
	\noindent {(ii)}  
Now, let us assume that
$
\dfrac{|\sigma_2|}{\sigma_1}=\dfrac{1-b}{b}.
$
 Let $u\in H^1_0(I)$ satisfy $Au=0$. Then $u$ is affine on each
subinterval $I_i$, and we may write
\[
u(x)=
\begin{cases}
	a_1x+b_1, & x\in I_1,\\
	a_2x+b_2, & x\in I_2.
\end{cases}
\]
Using the boundary conditions $u(0)=u(1)=0$, we obtain
\[
u(0)=b_1=0,
\qquad
u(1)=a_2+b_2=0 \ \Longleftrightarrow\ b_2=-a_2.
\]
Hence,
\[
u(x)=
\begin{cases}
	a_1x, & x\in I_1,\\
	a_2x-a_2=a_2(x-1), & x\in I_2.
\end{cases}
\]

Since $u\in H^1_0(I)\subset C^0(\overline I)$, the function is continuous at
$x=b$, and thus
\[
u(b^-)=u(b^+)
\quad\Longleftrightarrow\quad
a_1 b = a_2(b-1)=-a_2(1-b).
\]
Therefore,
\begin{equation}\label{eq:rel-a2-a1}
	a_2=-\frac{b}{1-b}\,a_1.
\end{equation}

It remains to use the equation $Au=0$ to obtain the transmission condition on
the flux. For any $v\in H^1_0(I)$, we have
\[
0=a(u,v)
=\sigma_1\int_0^b u'(x)v'(x)\,\mathrm{d}x+\sigma_2\int_b^1 u'(x)v'(x)\,\mathrm{d}x
=\sigma_1u'(b^-)\,v(b)-\sigma_2u'(b^+)\,v(b),
\]
where we used that $u''=0$ on each subinterval and that $v(0)=v(1)=0$.
Then, we deduce the transmission condition
\[
\sigma_1u'(b^-)=\sigma_2u'(b^+).
\]
As $u'(x)=a_1$ on $I_1$ and $u'(x)=a_2$ on $I_2$, hence
\[
\sigma_1 a_1=\sigma_2 a_2.
\]
Substituting \eqref{eq:rel-a2-a1} gives
\[
\sigma_1 a_1=\sigma_2\Bigl(-\frac{b}{1-b}a_1\Bigr)
\quad\Longleftrightarrow\quad
a_1\left(\sigma_1-|\sigma_2|\frac{b}{1-b}\right)=0.
\]
Under the present assumption $\dfrac{|\sigma_2|}{\sigma_1}=\dfrac{1-b}{b}$, we
have $\sigma_1-|\sigma_2|\dfrac{b}{1-b}=0$, so $a_1$ is free. Therefore, all
solutions are parametrized by $a_1\in\mathbb R$, with $a_2$ given by
\eqref{eq:rel-a2-a1}. Writing $c:=u(b)=a_1 b$, we obtain
\[
u(x)=
\begin{cases}
	c\,\dfrac{x}{b}, & x\in I_1,\\[6pt]
	c\,\dfrac{1-x}{1-b}, & x\in I_2,
\end{cases}
= c\,\varphi(x).
\]
%where $\varphi$ is the piecewise affine function introduced above. 
Hence
$
\ker(A)=\operatorname{span}\{\varphi\}.
$

Finally, if
$\dfrac{|\sigma_2|}{\sigma_1}\neq\dfrac{1-b}{b}$, then
$\sigma_1-|\sigma_2|\dfrac{b}{1-b}\neq0$ and the previous identity forces
$a_1=0$, hence $a_2=0$ and $u\equiv0$. This shows that $A$ is injective away from
the critical ratio.
\end{proof}

\begin{rem}
 This classical result highlights two important facts: first, the local problem
naturally involves a single interface profile $\varphi$; second, the coercivity
properties are governed by a sharp algebraic condition on the coefficients.
These features stand in contrast with the nonlocal case, where the operator
creates interactions between the two subdomains that are not confined to the
interface. This leads to a weaker notion of
\text{T-coercivity}.
\end{rem}

\subsection{Nonlocal counterpart  of Problem \eqref{ClassicalProblem}}
%The Fractional Analogue and a New Decoupled Nonlocal Model
%Two Nonlocal Models: Intuitive and Decoupled
%Fractional Extension and New Nonlocal Construction

In this subsection, we study Problem \eqref{ClassicalProblem} in a nonlocal setting. 

Unlike the classical case, where only $\sigma_1$
and $\sigma_2$ suffice, the nonlocal nature of the fractional operator necessitates an additional parameter $\sigma_3$ to account for interactions across $I_1$
and $I_2$.

 The intuitive fractional counterpart of the classical variational formulation writes
\begin{equation} \label{FractionalVF}
    \left\{ 
    \begin{array}{lll}
    	\text{Find}\;\; u\in\widetilde{H}^s(I)\;\;\text{such that:} \\
    		a^s_{\underline{\sigma}}(u,v)=\dis\int_I f v, \quad \forall v\in\widetilde{H}^s(I),
    \end{array}
    \right.
\end{equation}
where \begin{equation}\label{FractionalBF}
	a^s_{\underline{\sigma}}(u,v):=\frac{C(s)}{2} \int_\R \int_\R \underline{\sigma}(x,y) \frac{(\widetilde{u}(x)-\widetilde{u}(y))(\widetilde{v}(x)-\widetilde{v}(y))}{|x-y|^{1+2s}} \, \mathrm{d}y \, \mathrm{d}x,
\end{equation}
$C(s):=\dfrac{2^{2 s} s \Gamma\left(s+\frac{1}{2}\right)}{\sqrt{\pi} \Gamma(1-s)}$,
 $\underline{\sigma}(x,y)=\left\{\begin{array}{lll}
	\sigma_1 & \text{if}\; (x,y)\in I_1\times [I_1\cup (\R\setminus I)],\\
	\sigma_2 & \text{if}\; (x,y)\in I_2\times [I_2\cup (\R\setminus I)],\\
	\sigma_3 & \text{if}\; (x,y)\in I_1\times I_2,
\end{array}  \right. \;\text{ and } \; \underline{\sigma}(x,y)=\underline{\sigma}(y,x) $ .

\smallbreak
The variational formulation \eqref{FractionalVF} is associated with the integral fractional
Laplacian under the homogeneous exterior condition; see, for instance,
\cite{DaouLaam,LisAl,MolRadRaf} and references therein.

\smallbreak

$\widetilde{H}^s(I)$ is the classical fractional Sobolev space given by $\widetilde{H}^s(I):=\{w\in{H}^s(I)\;;\; \widetilde{w}\in {H}^s(\R)\}$, where $\widetilde{w}$ is the continuation of $w$ by 0 to $\R\setminus I$. The space is equipped with the norm
$$
\|w\|^2_{\widetilde{H}^s(I)}=\frac{C(s)}{2}\int_\R \int_\R  \frac{|\widetilde{w}(x)-\widetilde{w}(y)|^2}{|x-y|^{1+2s}} \mathrm{d}y \, \mathrm{d}x.
$$
%Let us go back to the fractional variational formulation \eqref{FractionalVF},

We recall that $\widetilde H^s(I)$ is the natural energy space associated with
the Dirichlet problem for the integral fractional Laplacian.
{Furthermore, for $i=1,2$, we denote
	$$
	|w_0|^2_{{H}^s(I_i)}:=\frac{C(s)}{2}\int_{I_i} \int_{I_i} \frac{|{w_0}(x)-{w_0}(y)|^2}{|x-y|^{1+2s}} \mathrm{d}y \, \mathrm{d}x.
	$$
	} For more details about fractional Sobolev spaces, we refer, e.g., to
	\cite{DiNezzaPalatucciValdinoci,BucVal2016,AcostaBor} and references cited therein.

\begin{comment}
	Since $s>\frac12$, the traces of $w_0\in H_e^s(I)$ at $0$ and $1$ are well defined and vanish, and by assumption $w_0(b)=0$. Hence
	\[
	w_0=0 \quad \text{on } \partial I_1=\{0,b\},
	\qquad
	w_0=0 \quad \text{on } \partial I_2=\{b,1\}.
	\]
	Therefore, by the fractional Poincar\'e inequality on each interval $I_i$, there exists $C>0$ such that
	\[
	\|w_0\|_{L^2(I_i)}^2 \le C\, |w_0|_{H^s(I_i)}^2,
	\qquad i=1,2.
	\]
	Consequently,
	\[
	\|w_0\|_{H^s(I_i)}^2
	=
	\|w_0\|_{L^2(I_i)}^2+|w_0|_{H^s(I_i)}^2
	\le C\,|w_0|_{H^s(I_i)}^2,
	\qquad i=1,2.
	\]
	Combining this with \eqref{2.14}--\eqref{2.15}, we obtain
	\[
	\int_{I_1}\int_{I_2}\frac{|w_0(x)-w_0(y)|^2}{|x-y|^{1+2s}}\,dy\,dx
	\le
	C\Big(|w_0|_{H^s(I_1)}^2+|w_0|_{H^s(I_2)}^2\Big),
	\]
	which, together with \eqref{2.11} and \eqref{2.12}, proves \eqref{2.9}.
\end{comment}

\medbreak
Problem \eqref{FractionalVF} is difficult to treat directly due to the nonlocal nature of the interactions, and even more challenging in the presence of a sign-changing coefficient \(\underline{\sigma}(x,y)\).
In particular, no complete theoretical analysis is yet available for the general case. To overcome this, we adopt a simplified model by removing the interface interaction term, i.e., we set \( \sigma_3 = 0 \). Under this assumption, the bilinear form \(a^s_{\underline{\sigma}}\) becomes weakly \(T\)-coercive, {\it that is}, there exist $ K \in \mathcal{K}(\widetilde{H}^s(I))$, $T \in \mathcal{L}(\widetilde{H}^s(I))$ bijective,  $\exists \alpha,\beta>0$, $\forall v \in \widetilde{H}^s(I)$,
$$|a^s_{\underline{\sigma}}(v, \mathrm{~T} v)| \geq \alpha\|v\|_{\widetilde{H}^s(I)}^2-\beta\|\mathrm{K} v\|_{\widetilde{H}^s(I)}^2.$$
The weak T-coercivity property implies that Problem~\eqref{FractionalVF}
is well-posed in the Fredholm sense.

\medbreak
{Before stating the main theorem, we introduce a few preliminary results that
	will be used in its proof.}

\begin{prop}\label{ClassicalDecomp}
	Let $s\in(\tfrac12,1)$ so that pointwise evaluation at $x=b$ is well-defined for 
	functions in $\widetilde H^s(I)$.  
	Let $\psi^s \in \widetilde H^s(I)$ be any function satisfying 
$
	\psi^s(b)=1.
$
	Then, any $w\in \widetilde H^s(I)$ admits a unique decomposition
	\begin{equation}\label{decomp-vs}
		w = w_0 + w(b)\, \psi^s,
	\end{equation}
	where $w_0\in\widetilde H^s(I)$ satisfies $w_0(b)=0$.
\end{prop}

\begin{proof}
	Since $s>\tfrac12$, the trace $w(b)$ is well-defined.  
	Let us define
$
	w_0 := w - w(b)\, \psi^s.
$
As $\psi^s \in \widetilde H^s(I)$, we have $w_0\in \widetilde H^s(I)$, and using $\psi^s(b)=1$ we obtain
$
	w_0(b) =w(b) - w(b) \psi^s(b) = 0,
$
	which proves existence of the decomposition \eqref{decomp-vs}.	For uniqueness, let us suppose that
	$
w = w_0 + c\, \psi^s = \widetilde w_0 + d\, \psi^s,
$
	with $w_0(b)=\widetilde w_0(b)=0$.  
	Evaluating at $x=b$ gives
$
	w(b) = c \psi^s(b) = c$, $
	w(b) = d \psi^s(b) = d,
$
	hence $c=d=w(b)$ and therefore $w_0 = \widetilde w_0$.  
	Thus the decomposition is unique.
\end{proof}

 \begin{prop}\label{NormWithoutIntermediate}
	Let $s\in (\tfrac12,1)$.	Let $w_0\in \widetilde{H}^s(I) $ such that $w_0(b)=0$. Then, there exists $C>0$ such that
	\begin{equation} \label{NormControl-eq1-init}
		\|w_0\|^2_{\widetilde{H}^s(I)}\leq C\left( 	|w_0|^2_{{H}^s(I_1)}+|w_0|^2_{{H}^s(I_2)}+\int_{I}\int_{I^c}\frac{|w_0(x)|^2}{|x-y|^{1+2s}} \, \mathrm{d}y \, \mathrm{d}x\right),
	\end{equation}
	which yields
	\begin{equation}\label{NormControl}
		\|w_0\|^2_{\widetilde{H}^s(I)}\leq C\left( 	\|w_0\|^2_{\widetilde{H}^s(I_1)}+	\|w_0\|^2_{\widetilde{H}^s(I_2)} \right).
	\end{equation}
\end{prop}
\begin{proof}
	Let us write
	\begin{equation} \label{NormControl-eq1}
		\|w_0\|^2_{\widetilde{H}^s(I)}=	|w_0|^2_{{H}^s(I_1)}+|w_0|^2_{{H}^s(I_2)}+C(s)\int_{I}\int_{I^c}\frac{|w_0(x)|^2}{|x-y|^{1+2s}} \, \mathrm{d}y \, \mathrm{d}x+C(s)\int_{I_1}\int_{I_2}\frac{|w_0(x)-w_0(y)|^2}{|x-y|^{1+2s}} \, \mathrm{d}y \, \mathrm{d}x.
	\end{equation}
	On one side, we have
	\begin{equation}\label{NormControl-eq2}
		\int_{I}\int_{I^c}\frac{|w_0(x)|^2}{|x-y|^{1+2s}} \, \mathrm{d}y \, \mathrm{d}x\leq 	\int_{I_1}\int_{I_1^c}\frac{|w_0(x)|^2}{|x-y|^{1+2s}} \, \mathrm{d}y \, \mathrm{d}x+	\int_{I_2}\int_{I_2^c}\frac{|w_0(x)|^2}{|x-y|^{1+2s}} \, \mathrm{d}y \, \mathrm{d}x.
	\end{equation}
%	Then, \eqref{NormControl-eq1} becomes
%	\begin{equation} \label{NormControl-eq3}
%		\|w_0\|^2_{\widetilde{H}^s(I)}\leq	\|w_0\|^2_{\widetilde{H}^s(I_1)}+\|w_0\|^2_{\widetilde{H}^s(I_2)}+C(s)\int_{I_1}\int_{I_2}\frac{|w_0(x)-w_0(y)|^2}{|x-y|^{1+2s}} \, \mathrm{d}y \, \mathrm{d}x.
%	\end{equation}
	On the other side, we have
	\begin{equation} \label{NormControl-eq4}
		\int_{I_1}\int_{I_2}\frac{|w_0(x)-w_0(y)|^2}{|x-y|^{1+2s}} \, \mathrm{d}y \, \mathrm{d}x\leq 2	\int_{I_1}\int_{I_2}\frac{|w_0(x)|^2}{|x-y|^{1+2s}} \, \mathrm{d}y \, \mathrm{d}x+2	\int_{I_2}\int_{I_1}\frac{|w_0(y)|^2}{|x-y|^{1+2s}} \, \mathrm{d}x \, \mathrm{d}y.
	\end{equation}
	For any $x\in I_1$,
	$$
	\int_{I_2}\frac{\mathrm{d}y}{|x-y|^{1+2s}} \leq 	\int_{I_1^c}\frac{\mathrm{d}y}{|x-y|^{1+2s}}=\frac{1}{2s}\left[\frac{1}{x^{2s}}+\frac{1}{(b-x)^{2s}}\right]\leq \frac{1}{s\operatorname{dist}(x,\partial I_1)^{2s}}.
	$$
	Furthermore, for any $y\in I_2$,
	$$
	\int_{I_1}\frac{\mathrm{d}x}{|x-y|^{1+2s}} \leq 	\int_{I_2^c}\frac{\mathrm{d}x}{|x-y|^{1+2s}}=\frac{1}{2s}\left[\frac{1}{(b-y)^{2s}}+\frac{1}{(1-y)^{2s}}\right]\leq \frac{1}{s\operatorname{dist}(y,\partial I_2)^{2s}}.
	$$
	As $w_0(0)=w_0(b)=w_0(1)=0$, Hardy's inequality (see, for instance, see \cite{Dyda} and \cite[Theorem 1.4.4.4]{Grisvard}) yields
	\begin{equation} \label{NormControl-eq5}
		\int_{I_1}\int_{I_2}\frac{|w_0(x)|^2}{|x-y|^{1+2s}} \, \mathrm{d}y \, \mathrm{d}x\leq 	\int_{I_1}\frac{|w_0(x)|^2}{s\operatorname{dist}(x,\partial I_1)^{2s}} \, \mathrm{d}x\leq C_1  |w_0|^2_{{H}^s(I_1)}
	\end{equation}
	and 
	\begin{equation} \label{NormControl-eq6}
		\int_{I_2}\int_{I_1}\frac{|w_0(y)|^2}{|x-y|^{1+2s}} \, \mathrm{d}x \, \mathrm{d}y\leq 	\int_{I_2}\frac{|w_0(y)|^2}{s\operatorname{dist}(y,\partial I_2)^{2s}} \, \mathrm{d}y\leq C_2  |w_0|^2_{{H}^s(I_2)}.
	\end{equation}
	This completes the proof.
\end{proof}

 \begin{prop}\label{NormEquiv}
For any $v^s \in \widetilde H^s(I)$ with $s\in(\tfrac{1}{2},1)$, we decompose
$
v^s = v_0^s + v^s(b)\,\psi^s,
$
where $v_0^s \in \widetilde H^s(I)$ and $v_0^s(b)=0$. Then, 
\begin{equation}
	\label{norm-equivalence}
	\|v^s\|_{\widetilde H^s(I)}^{2}
	\;\asymp\;
	\|v_0^s\|_{\widetilde H^s(I)}^{2}
	\;+\; |v^s(b)|^{2}\,|\lambda(s,I)|^{2},
\end{equation}
where 	$
0<\lambda(s,I):=\|\psi^{s}\|_{\widetilde{H}^{s}(I)}<+\infty.
$
%	see Remark \ref{Norm-phis}.
	%	By Proposition~\ref{cv-phi}, $\varphi^s-\varphi \to 0$ as $s\to 1^-$, so the choice of interface lift does not affect the validity of the argument, only the constants.
\end{prop}

Now, let us state and prove the main result of this paragraph.  
For later use, we introduce the notation
\[\Phi_i=
\Phi_i(\psi^s,s)
:= \frac{2}{C(s)}|\psi^s|_{H^s(I_i)}^{2}
+ 2\int_{I_i} |\psi^s(x)|^{2}\,\omega(x)\,\mathrm{d}x 
> 0,\qquad i=1,2,
\]
where
$
\omega(x)
:= \dis\int_{I^{c}} \frac{\mathrm{d}y}{|x-y|^{1+2s}}.
$

\begin{theo}\label{WeakTCoercivity}
	Let $s\in(\tfrac12,1)$ and $f\in H^{-s}(I)$.  
For a fixed $\psi^s\in\widetilde H^s(I)$ with $\psi^s(b)=1$, let
us define
	\[
T v^s=
\begin{cases}
\phantom{-}	v_0^s + \varepsilon\,v^s(b)\,\psi^s, &\text{ in } I_1, \\[3pt]
	-\,v_0^s + \varepsilon\, v^s(b)\,\psi^s, &\text{ in } I_2, %\\[3pt]
	%0, & \text{ in } \R\setminus I,
\end{cases}
\]
with $\varepsilon\in\{-1,1\}$. 
	  Assume that $\sigma_3=0$ and $\dfrac{|\sigma_2|}{\sigma_1}\neq\dfrac{\Phi_1}{\Phi_2}$, where $\Phi_1,\Phi_2$  are defined as above. 
	Then, the form $a^s_{\underline{\sigma}}$ is weakly T-coercive  and Problem \eqref{FractionalVF} is well-posed in the Fredholm sense.
\end{theo}
 \begin{rem}\label{Remark-lift}
	In contrast with the local case, where the interface profile is uniquely
	determined, the fractional setting admits many possible liftings. 	In fact, the quantity $\dfrac{\Phi_1}{\Phi_2}$ depends on the chosen function $\psi^s$ and
		thus on the specific construction of the operator $T$ used in the proof.		
		Different choices of $\psi^s$ yield different sufficient conditions for weak T-coercivity, since the proof depends on the associated operator $T$. This dependence concerns the criterion obtained by this argument, and not the definition of Problem~\eqref{FractionalVF} itself.

	Regarding the choice of  $\psi^s$, the piecewise-defined
	function
	\[
	\varphi^s(x):=
	\begin{cases}
		\dfrac{x^{s}}{b^{s}}, & x\in I_1,\\[4pt]
		\dfrac{(1-x)^{s}}{(1-b)^{s}}, & x\in I_2,\\[4pt]
		0, & \text{otherwise},
	\end{cases}
	\]
	is a convenient analytical choice due to its consistency
	with the local limit $s\to1^-$. In particular, $\varphi^s\in\widetilde H^s(I)$
	(see Remark~\ref{Norm-phis}). However, for this choice, no explicit expression
	for the ratio $\Phi_1(\varphi^s,s)/\Phi_2(\varphi^s,s)$ is available.
	
However, taking the classical lifting $\varphi$ instead,
 the quantities $\Phi_1(\varphi,s)$ and $\Phi_2(\varphi,s)$ can be
	computed explicitly. 	In this case, we obtain
	\[
	\frac{\Phi_1(\varphi,s)}{\Phi_2(\varphi,s)}
	=
	\frac{Q_1+R_1}{Q_2+R_2},
	\]
	where
	\[
	Q_1=\frac{b^{3-2s}}{b^2(1-s)(3-2s)}, 
	\qquad
	Q_2=\frac{(1-b)^{3-2s}}{(1-b)^2(1-s)(3-2s)},
	\]
	\[
	R_1=\frac{b^{1-2s}}{s(3-2s)}
	+\frac{(1-b)^{1-2s}\!\left(-2b^2s^2+3b^2s-b^2+2bs-b-1\right)+1}
	{b^2s(1-s)(1-2s)(3-2s)},
	\]
	and
	\[
	R_2=\frac{(1-b)^{1-2s}}{s(3-2s)}
	+\frac{b^{1-2s}\!\left(-2b^2s^2+4bs^2-2s^2-8bs+3b^2s+5s+3b-b^2-3\right)+1}
	{(1-b)^2s(1-s)(1-2s)(3-2s)}.
	\]
	Moreover, the following properties hold:
	\begin{itemize}
		\item $\displaystyle \lim_{s\to1^-} \frac{Q_1+R_1}{Q_2+R_2}=\frac{1-b}{b}$,
		which coincides with the classical transmission condition;
		\item if $b=\tfrac12$, then $\displaystyle \frac{Q_1+R_1}{Q_2+R_2}=1$ for any
		$s\in(\frac12,1)$.
	\end{itemize}
\end{rem}

\begin{proof}[Proof of Theorem \ref{WeakTCoercivity}] 	Recall that any $v^s\in \widetilde H^s(I)$ admits the unique decomposition
	\[
	v^s = v_0^s + v^s(b)\,\psi^s,
	\quad 
	v_0^s\in\widetilde H^s(I) \quad\text{and}\quad v_0^s(b)=0.
	\]
Let us choose
	\[
T v^s=
\begin{cases}
	\phantom{-}v_0^s + \varepsilon\,v^s(b)\,\psi^s, &\text{ in } I_1, \\[3pt]
	-\,v_0^s + \varepsilon\, v^s(b)\,\psi^s, &\text{ in } I_2, %\\[3pt]
	%0, & \text{ in } \R\setminus I,
\end{cases}
\]
where $\varepsilon\in\{-1,1\}$ to be chosen. It is straightforward to verify that $T\in\mathcal{L}(\widetilde H^s(I))$ and that $T$ is bijective.

We have
$$
Tv^s(x)-Tv^s(y)=
\left\{ 
\begin{array}{llll}
	[v_0^s(x)-v_0^s(y)]+\varepsilon v^s(b)[\psi^s(x)-\psi^s(y)],&x,y\in I_1,
	\\
	{[v_0^s(y)-v_0^s(x)]}+\varepsilon v^s(b)[\psi^s(x)-\psi^s(y)],&x,y\in I_2.
%		\\
%	{[v_0^s(x)+v_0^s(y)]}+\varepsilon v^s(b)[\psi^s(x)-\psi^s(y)],&x\in I_1,y\in I_2.
\end{array}
	\right.
$$
Since $\sigma_3=0$, interactions over $I_1\times I_2$
do not contribute, and
\begin{equation}
	\begin{array}{lll}
		a^s_{\underline{\sigma}}(v^s,Tv^s)&\dis=\frac{C(s)}{2} \int_\R \int_\R \underline{\sigma}(x,y) \frac{(v^s(x)-v^s(y))(Tv^s(x)-Tv^s(y))}{|x-y|^{1+2s}} \, \mathrm{d}y \, \mathrm{d}x	\\\\&\dis=\frac{C(s)}{2}\left(A_1+A_2+A_3+A_4+A_5\right),
	\end{array}
\end{equation}
where, %up to a positive multiplicative constant,
$$
A_1=\sigma_1	\int_{I_1}\int_{I_1}\frac{|v_0^s(x)-v_0^s(y)|^2}{|x-y|^{1+2s}} \, \mathrm{d}y \, \mathrm{d}x+2\sigma_1\int_{I_1}|v_0^s(x)|^2\omega(x) \, \mathrm{d}x,
$$
$$
A_2=|\sigma_2|	\int_{I_2}\int_{I_2}\frac{|v_0^s(x)-v_0^s(y)|^2}{|x-y|^{1+2s}} \, \mathrm{d}y \, \mathrm{d}x+2|\sigma_2|\int_{I_2}|v_0^s(x)|^2\omega(x) \, \mathrm{d}x,
$$
$$
A_3=\varepsilon |v^s(b)|^2\sigma_1	\int_{I_1}\int_{I_1}\frac{|\psi^s(x)-\psi^s(y)|^2}{|x-y|^{1+2s}} \, \mathrm{d}y \, \mathrm{d}x+2\varepsilon |v^s(b)|^2\sigma_1	\int_{I_1}|\psi^s(x)|^2\omega(x)\mathrm{d}x,
$$
$$
A_4=-\varepsilon |v^s(b)|^2|\sigma_2|	\int_{I_2}\int_{I_2}\frac{|\psi^s(x)-\psi^s(y)|^2}{|x-y|^{1+2s}} \, \mathrm{d}y \, \mathrm{d}x-2\varepsilon |v^s(b)|^2|\sigma_2|	\int_{I_2}|\psi^s(x)|^2\omega(x)\mathrm{d}x$$
and
$$
\begin{array}{lll}
	A_5&\dis=(1+\varepsilon)v^s(b)\sigma_1\left[\int_{I_1}\int_{I_1}\frac{(v_0^s(x)-v_0^s(y))(\psi^s(x)-\psi^s(y))}{|x-y|^{1+2s}} \, \mathrm{d}y \, \mathrm{d}x+2\int_{I_1}  v_0^s(x) \psi^s(x)\omega(x)\mathrm{d}x\right]
	\\\\&\dis+(1-\varepsilon)v^s(b)|\sigma_2|\left[\int_{I_2}\int_{I_2}\frac{(v_0^s(x)-v_0^s(y))(\psi^s(x)-\psi^s(y))}{|x-y|^{1+2s}} \, \mathrm{d}y \, \mathrm{d}x+2\int_{I_2}  v_0^s(x) \psi^s(x)\omega(x)\mathrm{d}x\right],
\end{array}
$$
with $\dis\omega(x)=\int_{I^c}\frac{\mathrm{d}y}{|x-y|^{1+2s}}$.

\medbreak
Thanks to Proposition \ref{NormWithoutIntermediate}, we have
$
A_1+A_2\geq C_1 	\|v^s_0\|^2_{\widetilde{H}^s(I)}.
$
Further, 
$
\begin{array}{lll}
A_3+A_4
%&=\varepsilon\dis v^s(b)^2\sigma_1\left[|\psi^s|^2_{{H}^s(I_1)}+2\int_{I_1}|\psi^s(x)|^2\omega(x)\mathrm{d}x\right]-\varepsilon v^s(b)^2|\sigma_2|\left[|\psi^s|^2_{{H}^s(I_2)}+2\int_{I_2}|\psi^s(x)|^2\omega(x)\mathrm{d}x\right]
%\\\\&
=
\varepsilon |v^s(b)|^2 \left( \sigma_1 \Phi_1 - |\sigma_2| \Phi_2\right).
\end{array}
$
At this stage, we choose $\varepsilon\in\{-1,1\}$ such that $A_3+A_4>0$, $that$ $is$
$$\varepsilon  \left( \sigma_1 \Phi_1 - |\sigma_2| \Phi_2\right)>0\Leftrightarrow \varepsilon  \left( 1 - \dfrac{|\sigma_2|}{\sigma_1} \dfrac{\Phi_2}{ \Phi_1} \right)>0\Leftrightarrow \varepsilon=\operatorname{sign}\left( 1 - \dfrac{|\sigma_2|}{\sigma_1} \dfrac{\Phi_2}{ \Phi_1}\right).$$
Now, let us assume $\dfrac{|\sigma_2|}{\sigma_1}\neq\dfrac{\Phi_1}{\Phi_2}$ so that $A_3+A_4=\mu |v^s(b)|^2$, with $\mu>0$.
It remains to control the mixed term $A_5$. Since $\varepsilon\in\{-1,1\}$, either
	$1+\varepsilon=0$ or $1-\varepsilon=0$, hence exactly one of the two contributions
	in $A_5$ vanishes. Denote by $i\in\{1,2\}$ the index of the remaining subdomain.
	Then $A_5$ can be written in the form
	\[
	A_5 = 2 v^s(b)\,|\sigma_i|\, \mathcal{B}_i(v_0^s,\psi^s),
	\]
	where 
	\[
	\mathcal{B}_i(v_0^s,\psi^s)
	:= \int_{I_i}\!\int_{I_i}\frac{(v_0^s(x)-v_0^s(y))(\psi^s(x)-\psi^s(y))}{|x-y|^{1+2s}}\,dy\,dx
	+2\int_{I_i} v_0^s(x)\psi^s(x)\,\omega(x)\,dx .
	\]
	By Cauchy--Schwarz in the Gagliardo term and in the weighted $L^2$-term, we obtain
	\begin{equation}\label{eq:Bi-est}
		\begin{array}{lll}
		|\mathcal{B}_i(v_0^s,\psi^s)|
	&\le\dis
		C_2\,\Bigl(|v_0^s|_{H^s(I_i)}^2 + \int_{I_i}|v_0^s(x)|^2\omega(x)\,dx\Bigr)^{1/2}
		\Bigl(|\psi^s|_{H^s(I_i)}^2 + \int_{I_i}|\psi^s(x)|^2\omega(x)\,dx\Bigr)^{1/2}
		\\\\
		&\leq\dis	C_2'\,\|v_0^s\|_{\widetilde H^s(I)}\|\psi^s\|_{\widetilde H^s(I)}.
	\end{array}
	\end{equation}
	%using the fact that $\dis\int_{I^c}\cdots\leq \dis\int_{I_i^c}\cdots$. 
	Then, we get
$
|A_5|\le C_3\,|v^s(b)|\,\|v_0^s\|_{\widetilde H^s(I)}\,\|\psi^s\|_{\widetilde H^s(I)}.
$

By Young's inequality, for any $\eta>0$,
\[
|A_5|
\le
\eta\,\|v_0^s\|_{\widetilde H^s(I)}^{2}
+ C_\eta\,|v^s(b)|^{2}\,\|\psi^s\|_{\widetilde H^s(I)}^{2}.
\]
Since $s>\tfrac12$, the trace map $\delta_b:\widetilde H^s(I)\to\mathbb R$,
$\delta_b(v^s)=v^s(b)$, is bounded. Define the rank-one operator
\[
K:\widetilde H^s(I)\to\widetilde H^s(I),\qquad Kv^s := v^s(b)\,\psi^s.
\]
Then $\mathrm{Ran}(K)=\mathrm{span}\{\psi^s\}$ is one-dimensional, hence
$K\in\mathcal K(\widetilde H^s(I))$. Moreover,
\[
\|Kv^s\|_{\widetilde H^s(I)}^{2}
= |v^s(b)|^{2}\,\|\psi^s\|_{\widetilde H^s(I)}^{2}.
\]
Therefore,
\[
|A_5|
\le
\eta\,\|v_0^s\|_{\widetilde H^s(I)}^{2}
+ C_\eta\,\|K v^s\|_{\widetilde H^s(I)}^{2}.
\]
Combining this estimate with the lower bound previously obtained for $A_1+A_2+A_3+A_4$,
and choosing $\eta>0$ small enough, we have
$$
|a^s_{\underline\sigma}(v^s,T v^s)|\geq   C_4	\|v^s_0\|^2_{\widetilde{H}^s(I)}+\mu |v^s(b)|^2- C_\eta\,\|K v^s\|_{\widetilde H^s(I)}^{2}
$$
Finally, thanks to Proposition \ref{NormEquiv}, there exist $\alpha,\beta>0$ such that
\[
|a^s_{\underline\sigma}(v^s,T v^s)|
\ge
\alpha \|v^s\|_{\widetilde H^s(I)}^{2}
-\beta \|K v^s\|_{\widetilde H^s(I)}^{2},
\qquad \forall v^s\in\widetilde H^s(I).
\]
This proves the weak T-coercivity of $a^s_{\underline\sigma}$ with respect to
$T$, and the Fredholm well-posedness of Problem~\eqref{FractionalVF} follows.
%conclude that t
	\end{proof}
%	\textcolor{gray}{Or should we choose something like		
%		$
%		K v:=\iota(v),
%		$	
%		where $\iota: \widetilde{H}^s(I) \hookrightarrow L^2(I)$ is the compact embedding?
%		I think it is unnecessary in our case.}

\subsection{New nonlocal model}

Under the assumption $\sigma_3=0$, Theorem~\ref{WeakTCoercivity} yields a weak
T-coercivity result for the original global nonlocal formulation, and hence
Fredholm well-posedness for Problem~\eqref{FractionalVF}. Now, we introduce a
different but related formulation, namely a reconstructed interface formulation
inspired by the decomposition of Proposition~\ref{ClassicalDecomp}. The purpose
of this formulation is not to replace the previous well-posedness result, but
rather to isolate the interface effect through the scalar unknown $u^s(b)$ and
to prepare the construction of an efficient numerical method. In particular,
the uniqueness of $u^s(b)$ in the reconstructed model should be distinguished
from the Fredholm well-posedness of the original global nonlocal problem. The
reconstructed formulation is intended to capture the interface structure of the
local problem and to provide a suitable framework for studying the limit
$s\to 1^-$.

\smallbreak
Rather than solving the global nonlocal problem~\eqref{FractionalVF} directly,
we propose a decoupled formulation on the two subdomains $I_1$ and $I_2$, which
allows us to isolate the interface effects while preserving the nonlocal
structure of the operator.
To this end, we consider the pair of independent subproblems
\begin{equation} \label{SubFractionalProblem1}
	\left\{   
	\begin{array}{rcl}
		\sigma_1(-\Delta)^s u_1 & = f & \text{in } I_1,\\
		u_1 &= 0 & \text{in }\mathbb{R}\setminus I_1 ,
	\end{array}     
	\right.
\end{equation}
and
\begin{equation} \label{SubFractionalProblem2}
	\left\{   
	\begin{array}{rcl}
		\sigma_2(-\Delta)^s u_2 & = f & \text{in } I_2,\\
		u_2 &= 0 & \text{in }\mathbb{R}\setminus I_2 .
	\end{array}     
	\right.
\end{equation}

The coupling between these two subproblems is recovered through an interface
lifting function $\psi^s\in\widetilde H^s(I)$ satisfying
\[
\psi^s(0)=\psi^s(1)=0,
\qquad 
\psi^s(b)=1,
\qquad 
\psi^s \longrightarrow \varphi \quad\text{as } s\to1^-,
\]
where $\varphi$ denotes the classical harmonic profile on~$I$.
For any $u^s\in\widetilde H^s(I)$, this choice yields the unique decomposition
\begin{equation}\label{us}
	u^s = u_0^s + u^s(b)\,\psi^s,
\end{equation}
where $u_0^s(b)=0$ and $u_0^s$ coincides with $u_1$ in $I_1$ and with $u_2$ in
$I_2$. In particular,
\[
u_0^s \in \widetilde H^s(I_1)\oplus \widetilde H^s(I_2).
\]

As mentioned in Remark \ref{Remark-lift}, our primary candidate for the interface function is the explicit family
$\{\varphi^s\}_{s\in(0,1)}$, defined by
\[
\varphi^s(x):=
\begin{cases}
	\dfrac{x^{s}}{b^{s}}, & x\in I_1,\\[4pt]
	\dfrac{(1-x)^{s}}{(1-b)^{s}}, & x\in I_2,\\[4pt]
	0, & \text{otherwise}.
\end{cases}
\]
\begin{figure}[H]
	\centering
	\includegraphics[scale=0.225]{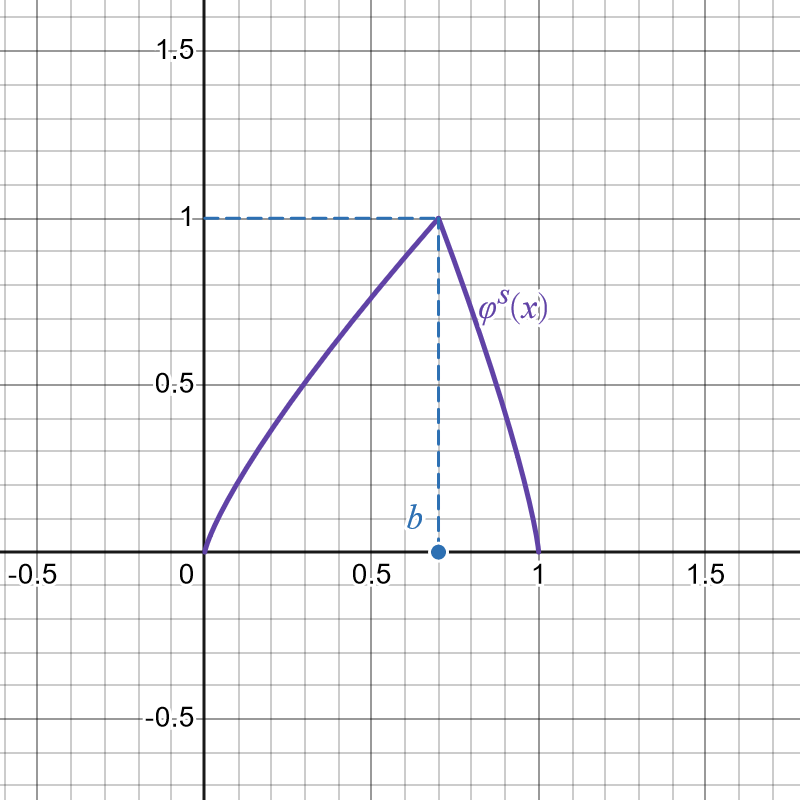}
	\caption{The function $\varphi^s$ for $b=0.7$ and $s=0.8$}
	\label{fig:phis}
\end{figure}

\begin{rem}
	A natural question is why one does not simply choose $\varphi$. Our viewpoint is that, to obtain a meaningful nonlocal-to-local limit, it is natural to start from a genuinely nonlocal lifting and recover $\varphi$ appearing only in the limit $s \to 1^-$.
\end{rem}
	
Let us recall that $\varphi^s\in\widetilde H^s(I)$
(see Remark~\ref{Norm-phis}). Although $\varphi^s$ is not $s$-harmonic, it satisfies the required interface
conditions and enjoys several advantages:  
(i) it converges to the classical harmonic function $\varphi$ as $s\to1^-$;  
(ii) its explicit dependence on $s$ facilitates the analysis of the limiting
behavior;  
(iii) numerical experiments indicate that alternative admissible choices of
$\psi^s$ lead to the same asymptotic behavior as \mbox{$s\to1^{-}$}.

By construction, the composite field $u^s=u_0^s+u^s(b)\varphi^s$ satisfies, in
each subdomain $I_k$,
\[
\sigma_k(-\Delta)^s u^s
= f + \sigma_k\,u^s(b)\,(-\Delta)^s\varphi^s,
\]
together with homogeneous exterior conditions on $\mathbb{R}\setminus I$.

We recall the variational formulations associated with
\eqref{SubFractionalProblem1} and \eqref{SubFractionalProblem2}. For $k=1,2$,
they read
\begin{equation} \label{FractionalVFi}
	\left\{ 
	\begin{array}{lll}
		\text{Find } u_k\in\widetilde H^s(I_k) \text{ such that:} \\
		\sigma_k\,a^s(u_k,v)=\displaystyle\int_{I_k} f v,
		\quad \forall v\in\widetilde H^s(I_k),
	\end{array}
	\right.
\end{equation}
where
\[
a^s(u_k,v)
:=\frac{C(s)}{2}\int_{\mathbb R}\!\int_{\mathbb R}
\frac{(\widetilde u_k(x)-\widetilde u_k(y))
	(\widetilde v(x)-\widetilde v(y))}
{|x-y|^{1+2s}}\,\mathrm dy\,\mathrm dx.
\]

The function $u_0^s$ is thus obtained by solving two decoupled subdomain problems.
Their well-posedness follows from the Lax--Milgram theorem, after the obvious rewriting
of the second problem in terms of $|\sigma_2|$; see, for instance,
\cite[Chapter~1]{BorthagarayThesis}.  The
scalar correction term $u^s(b)\varphi^s$ is then introduced to recover the
global behavior across the interface.

To determine the value of $u^s(b)$, we impose that the full function $u^s$
satisfies the global variational formulation \eqref{FractionalVF}.
By Theorem \ref{WeakTCoercivity}, this problem \eqref{FractionalVF} is well-posed in the Fredholm sense as
\begin{equation}\label{Hypo}
	\sigma_3=0\quad\text{and}\quad\dfrac{|\sigma_2|}{\sigma_1}\neq\dfrac{\Phi_1(\varphi^s,s)}{\Phi_2(\varphi^s,s)}.
\end{equation}
 Testing with
$v=\varphi^s$ and using the linearity of $a^s_{\underline\sigma}$ yields
\begin{equation}\label{Equ}
a^s_{\underline\sigma}(u^s,\varphi^s)
= a^s_{\underline\sigma}(u_0^s,\varphi^s)
+ u^s(b)\,a^s_{\underline\sigma}(\varphi^s,\varphi^s)
= \int_I f\varphi^s.
\end{equation}
By direct computation, one obtains
$
a_\sigma^s(\varphi^s,\varphi^s)
= \sigma_1 \Phi_1(\varphi^s,s)-|\sigma_2|\,\Phi_2(\varphi^s,s).
$
Thus, thanks to \eqref{Hypo}, we have
 \(a_\sigma^s(\varphi^s,\varphi^s)\neq 0\). Consequently, the quantity
\(u^s(b)\) is uniquely determined by \eqref{Equ}.

\noindent
Hence, the reconstructed solution
\[
u^s \in \widetilde H^s(I_1)\oplus \widetilde H^s(I_2)
\oplus \operatorname{span}\{\varphi^s\},
\]
is well-defined.

\begin{rem}\label{Norm-phis}
	Let $s\in [\frac12,1)$. Since $x^\alpha \in H^s(0,\lambda)$ for any $\lambda>0$ whenever
	$\alpha>s-\frac12$ (see  \cite[Lemma 4.1]{BiaBraZag}),
	we obtain $\varphi^s\in \widetilde H^s(I)$. In particular,
	\[
	0<\|\varphi^s\|_{\widetilde H^s(I)}<+\infty.
	\]
\end{rem}

\section{Finite element approximation}\label{Section3}
Now, let us introduce a standard conforming $\mathbb P_1$ finite element
discretization on the interval $I=(0,1)$; see, for instance,
\cite{CiarletFEM}. Let
\[
0=x_0<x_1<\cdots <x_M=b<\cdots <x_{N_h}<x_{N_h+1}=1
\]
be a uniform mesh of $I$, with mesh size $h$, chosen so that the interface
point $b$ coincides with the node $x_M$.

\begin{center}
	\begin{tikzpicture}[scale=1.5]
		% Draw the interval [0,1]
		\draw[thick, ->] (0,0) -- (6,0) node[right] {$x$};
		\foreach \x/\label in {0/0, 2/b, 5/1} {
			\draw (\x,0.1) -- (\x,-0.1) node[below] {$\label$};
		}
		% Mark the nodes
		\fill[blue] (0,0) circle (2pt) node[above] {$x_0$};
		\fill[blue] (2,0) circle (2pt) node[above] {$x_M$};
		\fill[blue] (5,0) circle (2pt) node[above] {$x_{N_h+1}$};
		% Discretization points
		\foreach \i in {0.5, 1, 1.5} {
			\fill[red] (\i,0) circle (2pt);
		}
		\foreach \i in {2.5, 3, 3.5, 4, 4.5} {
			\fill[red] (\i,0) circle (2pt);
		}
		\node at (3,-1) {Uniform discretization of $I$};
	\end{tikzpicture}
\end{center}

For $k=1,2$, we denote by $V_k^h\subset H^1(I_k)$ the finite element space of
continuous piecewise affine functions on the restriction of the mesh to $I_k$.
These spaces are generated by the standard nodal basis
$\{\phi_i\}_{i=1}^{N_h}$, where each basis function $\phi_i$ is the usual hat
function associated with the node $x_i$, namely
%The finite-dimensional subspaces 
%We denote by \( V_k^h \subset H^1(I_k) \) ($k=1,2$) the space of continuous piecewise linear functions on the mesh. They are spanned by the standard nodal basis \( \{\phi_i\}_{i=1}^{N_h} \), where each \( \phi_i \) is the ``tent'' function defined by:
\[
\phi_i(x) = 1 - \frac{|x - x_i|}{h}, \quad x \in I,
\]
with support \( \text{supp}(\phi_i) = [x_{i-1}, x_{i+1}] \). In $(0,b)$ (resp. $(b,1)$), we exclude the boundary basis functions \( \phi_0 \) and \( \phi_M \) (resp. \( \phi_M \) and \( \phi_{N_h+1} \)) to ensure compatibility with the continuous problems (homogeneous Dirichlet condition). Thus, the approximation is constructed solely over the interior nodes \( \{x_1, \dots, x_{N_h}\} \).
 
 \subsection{The old model} \label{Subsection3.1}
Here, we approximate \eqref{FractionalVF} with the discrete problem: 
 
\begin{equation} \label{FractionalDVFold1}
	\left\{
	\begin{array}{lll}
		\text{Find $\ud{u}_h \in V_h$ such that:}
		\\\dis
			a^s_{\underline{\sigma}}(\ud{u}^s_h,v_h)=
		\int_{0}^1f v_h, \quad\forall  v_h \in V_h,
	\end{array}
	\right.
\end{equation}
 where
 $$
	a^s_{\underline{\sigma}}(\ud{u}^s_h,v_h):= \frac{C(s)}{2} \int_\R \int_\R \underline{\sigma}(x,y) \frac{(\ud{u}^s_h(x)-\ud{u}^s_h(y))({v}_h(x)-{v}_h(y))}{|x-y|^{1+2s}} \, \mathrm{d}y \, \mathrm{d}x.
 $$
  Choosing $v_h=\phi_i$, we get
 
\begin{equation} \label{FractionalDVFold2}
 \frac{C(s)}{2} \int_{\mathbb{R}} \int_{\mathbb{R}}\underline{\sigma}(x,y) \frac{\left(\ud{u}^s_h(x)-\ud{u}^s_h(y)\right)\left(\phi_j(x)-\phi_j(y)\right)}{|x-y|^{1+2 s}} \, \mathrm{d} y \, \mathrm{d}x=\int_{0}^1 f(x) \phi_j (x)\, \mathrm{d}x, \quad j=1, \ldots, N_h.
\end{equation}

Since $u_h \in V_h$, we have $\dis \ud{u}^s_h(x)=\sum_{i=1}^{N_h}  \ud{u}^s_h(x_i) \phi_i(x) $. Therefore,  \eqref{FractionalDVFold2} is reduced to solve the linear system $\mathbb{A}_{\text{old}} \underline{U}_h^s=F$, where the elements of  $\mathbb{A}_{\text{old}} \in \mathbb{R}^{N_h \times N_h}$ are given by
 
\begin{equation} \label{ExpressionsAijOld}
 \mathbb{A}^{ij}_{\text{old}}=\frac{C(s)}{2} \int_{\mathbb{R}} \int_{\mathbb{R}}\underline{\sigma}(x,y) \frac{\left(\phi_i(x)-\phi_i(y)\right)\left(\phi_j(x)-\phi_j(y)\right)}{|x-y|^{1+2 s}} \, \mathrm{d} y \, \mathrm{d}x.
\end{equation}
The constant $C(s)$ is the same for all entries of the matrix. For clarity and to simplify the presentation, let us denote $\mathbb{B}:=\dfrac{2}{C(s)}\mathbb{A}_{\text{old}}$.
Moreover, the coefficients of $F \in \mathbb{R}^{N_h}$, denoted by $\left(F_1, \ldots, F_{N_h}\right)$, are defined as
\begin{equation}\label{ExpressionsFi}
F_j:=\int_{0}^1 f(x) \phi_j (x)\, \mathrm{d}x, \quad j=1, \ldots, N_h.
\end{equation}

It follows directly from  \eqref{ExpressionsAijOld} that the matrix \( \mathbb{B} \) is symmetric. Consequently, it suffices to compute only the entries \( \mathbb{B}_{ij} \) for \( j \geq i \), see Figure \ref{fig:matrixB}.
In the following, we present the explicit expressions of the matrix entries. These expressions are inspired by the computations carried out in \cite{BicHer}, where the authors derived them in the particular case \( \sigma_1 = \sigma_2 = \sigma_3 \). For the reader's convenience, the detailed computations are collected in the appendix; see Section~\ref{Appendix8}.
\vspace*{-3cm}
\begin{figure}[H]
	\centering
	\includegraphics[scale=1.0]{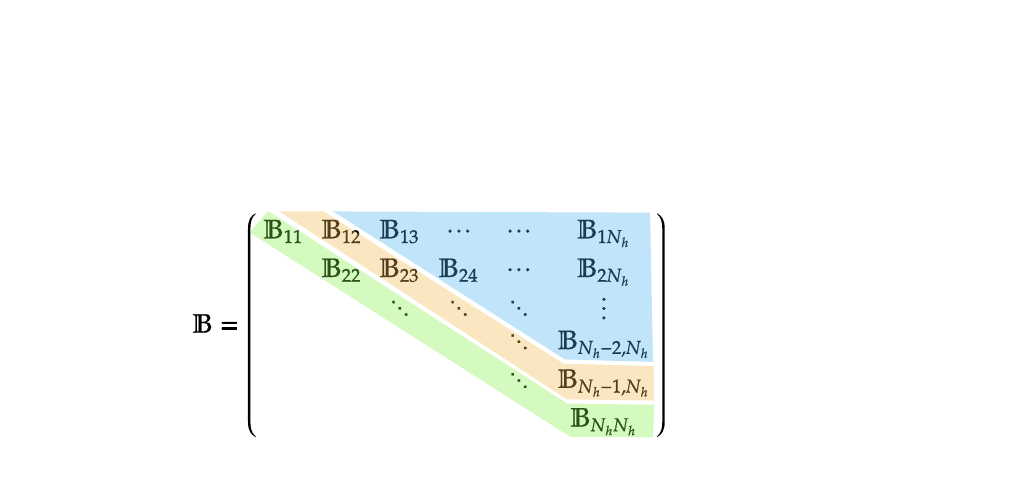}
	\caption{Structure of the stiffness matrix $\mathbb{B}$}
	\label{fig:matrixB}
\end{figure}

Let us denote,
\begin{itemize}
	\item [---] $H(h,s):=\dfrac{h^{1-2s}}{2s(1-s)(1-2s)(3-2s)}$ $(s\neq\dfrac12)$~;

	\item[---] $H_1(h,s):= \dfrac{h^{1-2s}}{s(3-2s)}$ ;

	\item[---] $H_2(h,s):=\left\{\begin{array}{lll}
		2H(h,s)[-2s^2+7s-7+2^{3-2s}] &\;\; \text{if} &s\neq\frac12,
		\\
		-5+8\log(2)&\;\; \text{if} &s=\frac12\;;
	\end{array} \right.$

	\item[---] $H_3(h,s):=\dfrac{h^{1-2s}}{(1-s)(3-2s)}$ ;
	
	\item[---] $H_4(h,s):=\left\{\begin{array}{lll}
		2H(h,s)[2s^2-5s+4-2^{2-2s}] &\;\; \text{if} &s\neq\frac12,
		\\
		3-4\log(2)&\;\; \text{if} &s=\frac12\;;
	\end{array} \right.$

	\item[---] $H_5(h,s):=H(h,s)[1-2s] $ ;

	\item[---] $H_6(h,s):= 2H(h,s)[1-2s][s-2+2^{1-2s}]
	$ ;
	
	\item[---] $H_7(h,s):=\left\{\begin{array}{lll}
	H(h,s)[4s^2-s(14-2^{4-2s})+13+3^{3-2s}-5\cdot2^{3-2s}] &\;\; \text{if} &s\neq\frac12,
		\\
		1-16\log(2)+9\log(3)&\;\; \text{if} &s=\frac12\;;
	\end{array} \right.$

	\item[---]  $\forall r>h$, $S_1(h,s,r):=\left\{\begin{array}{lll}
		\dfrac{2H(h,s)}{h^{3-2s}}[(h+r)^{3-2s}+2hr^{2-2s}[2s-3]-(r-h)^{3-2s}] &\;\; \text{if} &s\neq\frac12,
		\\
		\frac{1}{h^{2}}[-2 \log  \left(-h +r \right) \left(h -r \right)^{2}+2 \left(h +r \right)^{2} \log \left(h +r \right)-8 r \left(\log \left(r \right)+\frac{1}{2}\right) h]&\;\; \text{if} &s=\frac12\;;
	\end{array} \right.$	
	
	\item[---]  $\forall r>h$, $S_2(h,s,r):=\left\{\begin{array}{lll}
		\dfrac{H(h,s)}{h^{3-2s}}[r^{2-2s}[2hs-3h+2r]+(r-h)^{2-2s}[2hs-h-2r]] &\;\; \text{if} &s\neq\frac12,
		\\
		\frac{1}{h^{2}}[2 r \left(h -r \right) \log  \left(-h +r \right)+2 r\log  \left(r \right)  \left(-h +r \right)+h \left(h -2 r \right)]&\;\; \text{if} &s=\frac12\;;
	\end{array} \right.$

	\item[---] $\forall k \geq 2$,
	\[
L_1(h,s,k) :=
	\begin{cases}
	H(h,s)\Big[
		3k^{3-2s} - 2(k+1)^{3-2s} + (k+2)^{3-2s}
		- (k - 2s + 2)(k-1)^{2-2s} \\
		\hspace{4.5cm}  + (-4s + 6)k^{2 - 2s}
		- (k - 2s + 4)(k+1)^{2 - 2s}
		\Big] & \text{if } s \neq \frac{1}{2}, \\[1.5ex]
		(-k^2 + 1)\log(k - 1)
		 +[-3k^2 - 8k - 5]\log(k + 1)+ (k + 2)^2 \log(k + 2)
		\\
		\hspace{9cm}+ 3k[k + \tfrac{4}{3}] \log(k)
		& \text{if } s = \frac{1}{2}\;;
	\end{cases}
	\]

	\item[---] $\forall k \geq 2$,
	\[
L_2(h,s,k) :=
	\begin{cases}
	H(h,s)\Big[
		(k-2)^{3-2s}-2(k-1)^{3-2s}+3k^{3-2s}-(k+2s-4)(k-1)^{2-2s} \\
		\hspace{4.5cm} 	+(4s-6)k^{2-2s}-(k+2s-2)(k+1)^{2-2s}
		\Big] & \text{if } s \neq \frac{1}{2}, \\[1.5ex]
		\left(k -2\right)^{2} \log \left(k -2\right)+[-3 k^{2}+8 k -5] \log \left(k-1 \right)+\left(-k^{2}+1\right) \log \left(k +1\right)
		\\
		\hspace{9cm}+3  k [k -\frac{4}{3}]\log \left(k \right)
		& \text{if } s = \frac{1}{2} \text{ and } k\neq2,
		\\[1.5ex]
		4 \log (2)-3 \log (3) & \text{if } s = \frac{1}{2} \text{ and } k=2.
	\end{cases}
	\]

\end{itemize}

We have the following expressions for the elements of $\mathbb{B}$ :

(i) when $j=i$ :
$$
\mathbb{B}_{ii} =  \begin{cases} (\sigma_1 + \sigma_2)\big[H_1 + H_3\big] + 2\sigma_3\big[H_2 + H_4\big] & \quad\text{if } x_i=b, 
	\\
	(\sigma_1 + \sigma_3)\big[H_1 + H_2\big] + 2\sigma_1\big[H_3 + H_4\big] & \quad\text{if } x_{i+1}=b, 
	\\
	(\sigma_2 + \sigma_3)\big[H_1 + H_2\big] + 2\sigma_2\big[H_3 + H_4\big]  & \quad\text{if } x_{i-1}=b, 
	\\
	2\sigma_2 \big[H_1 + H_2+H_3 + H_4\big] +(\sigma_3-\sigma_2)S_1(h,s,x_i-b) & \quad\text{if } x_{i-1}>b, 
	\\
	2\sigma_1 \big[H_1 + H_2+H_3 + H_4\big] +(\sigma_3-\sigma_1)S_1(h,s,b-x_i) & \quad\text{if } x_{i+1}<b\;;
	 \end{cases}
$$

(ii) when $j=i+1$ :
$$
\mathbb{B}_{i,i+1} =  \begin{cases}    -\sigma_2H_3 +(\sigma_2 + \sigma_3)\big[H_5 + H_6\big] + \sigma_3H_7 & \quad\text{if } x_i=b, 
	\\     -\sigma_1H_3 +(\sigma_1 + \sigma_3)\big[H_5+ H_6\big] + \sigma_3H_7 & \quad\text{if } x_{i+1}=b,
	\\
	\sigma_2[-H_3 +2H_5 + 2H_6+H_7]-(\sigma_3-\sigma_2)S_2(h,s,x_{i+1}-b)  & \quad\text{if } x_i>b, 
\\
\sigma_1[-H_3+2H_5 + 2H_6+H_7]-(\sigma_3-\sigma_1)S_2(h,s,b-x_{i})  & \quad\text{if } x_{i+1}<b\;;
		 \end{cases}
	$$

(iii) when $k:=j-i\geq 2$ :

\bigbreak 
$
\mathbb{B}_{ij} =  \begin{cases}   \sigma_3L_1(h,s,k)+\sigma_2L_2(h,s,k) & \quad\text{if } x_i=b, 
	\\     \sigma_3L_1(h,s,k)+\sigma_1L_2(h,s,k) & \quad\text{if } x_{j}=b,
	\\
\widehat{\sigma}_{ij}[L_1(h,s,k)+L_2(h,s,k)]  & \quad\text{if } x_i,x_j\neq b,
\end{cases}
$ where $\widehat{\sigma}_{ij}:=\left\{\begin{array}{llll}
	\sigma_1&\text{if}& x_i,x_j<b,
	\\
	\sigma_2&\text{if}&x_i,x_j>b,
	\\
	\sigma_3&\text{if}& \text{otherwise}.
\end{array}   \right.$

\bigbreak

In the case \( \sigma_1 = \sigma_2 = \sigma_3 \), we recover the expressions given in \cite{BicHer}:  for any $s \neq \frac12$
$$
\mathbb{B}_{ij}= \begin{cases}
	H(h,s)[2^{4-2 s}-8] & \quad\text{if } j=i, 
	\\ 
		H(h,s)[3^{3-2 s}+7-2^{5-2 s}], & \quad\text{if } j=i+1,
	\\ H(h,s)[6 k^{3-2 s}+(k+2)^{3-2 s}+(k-2)^{3-2 s}-4(k+1)^{3-2 s}-4(k-1)^{3-2 s}], & \quad\text{if } k=j-i\geq 2 .\end{cases}
$$

\noindent Further, when $s=\frac12$ we have
$$
\mathbb{B}_{ij}= \begin{cases}
	8 \log(2) &  \quad\text{if } j=i, \\
 9 \log(3)-16 \log (2) &\quad\text{if } j=i+1, \\
(k-2)^2 \log (k-2)-4(k-1)^2 \log (k-1)	+6k^2 \log (k) & \\ \qquad\qquad\qquad-4(k+1)^2 \log (k+1)+(k+2)^2 \log (k+2)& \quad\text{if } k=j-i>2 \\ 56 \log (2)-36 \log (3) &\quad\text{if } j-i=2 .  \end{cases}
$$

  \subsection{The new model}\label{Subsection3.2}
 The corresponding discrete variational formulation for \eqref{FractionalVFi} is as follows : for $k=1,2$

\begin{equation} \label{FractionalDVF}
    \left\{ 
   \begin{array}{lll}
   	\text{Find}\;\; u^h_{k}\in V_k^h\;\;\text{such that:} \\
   \sigma_k	a^s(u^h_{k},v^h)=\dis\int_{I_k}  f v^h, \quad \forall v^h\in V^h_k,
   \end{array}
   \right.
\end{equation}
%where $V_h:=\{v\in \widetilde{H}^s(I) ;\;; v|_{[x_i,x_{i+1}]} \in\mathbb{P}_1\}$ where  Hence, 
The discrete solutions are $u_1^h(x)=\dis\sum_{i=1}^{M-1} u_1^h(x_i)\phi_i(x)$ and $u_2^h(x)=\dis\sum_{i=M+1}^{N_h} u_2^h(x_i)\phi_i(x)$.
%
%and the linear system we solve numerically is $\mathbb{A}U=F$, where
%$$
%\mathbb{A}_{i,j}=\frac{C(s)}{2} \int_\R \int_\R \underline{\sigma}(x,y) \frac{(\phi_i(x)-\phi_i(y))(\phi_j(x)-\phi_j(y))}{|x-y|^{1+2s}} \, \mathrm{d}y \, \mathrm{d}x
%$$
%and
%$$
%F_i=\int_I f \phi_i.
%$$
%We have checked that  the numerical solutions tend to the classical ones as $s\to 1$.
%
%
%
%
 Then, the approximated solution of the new model is 
{ \begin{equation}\label{uhs}
		\overline{u}^s_{h}=u^s_{0,h}+\overline{u}^s_{h}(b)\varphi^s=\dis\sum_{\underset{i\neq M}{i=1}}^{N_h} u^{s,i}_{{0,h}}\phi_i+\overline{u}^s_{h}(b)\varphi^s,
\end{equation}}
  where $u_{0,h}^s=u_1^h$ in $I_1$ and  $u_{0,h}^s=u_2^h$  in $I_2$. 
  
  Now, we  impose that the full function $\overline{u}^s_{h}$
  satisfies the global discrete variational formulation \eqref{FractionalDVFold1} with $\sigma_3=0$.
   Consequently, the corresponding linear system to solve is $
\pmb{\mathbb{K}}\boldsymbol{U_{0,h}^s=G}$, where

%\[
%B =
%\begin{bmatrix}
%	&&&D_1^1&0&\cdots&0\\
%	&	\mathbb{A}_1& & \vdots & \vdots & \ddots&\vdots \\
%		&&&D_1^{N_1}&0&\ldots&0\\
%D_1^1 &\cdots&D_1^{N_1} & c & D_2^1 &\cdots&D_2^{N_2}\\
%0&\cdots	 & 0& D_2^1 & &&\\
%\vdots &\ddots	 &\vdots & \vdots & &\mathbb{A}_2&\\
%0&\cdots	 &0 & D_2^{N_2} & &&
%\end{bmatrix}
%\]
%$$
%\mathbb{K} =
%\begin{bmatrix}
%	\begin{bmatrix}
%	\hspace*{0.05cm}&&	& &&& 	\hspace*{0.05cm}\\
%	&&	& \mathbb{A}_1 &&& \\
%	&&	& &&&
%	\end{bmatrix}
%	&\vspace*{0.12cm}
%	\begin{bmatrix}
%		D_1 \\
%		\vdots \\
%		D_{M-1}
%	\end{bmatrix}
%	&
%	\begin{bmatrix}
%	\hspace*{0.2cm}	0 &	& \cdots && 0\hspace*{0.2cm} \\
%		\vdots 	&& \ddots &	& \vdots \\
%		0 & 	&\cdots 	&& 0	
%	\end{bmatrix}
%	\\\vspace*{0.12cm}
%	\begin{bmatrix}
%	D_1 & \cdots & D_{M-1}
%	\end{bmatrix}
%	&
%	\begin{bmatrix}
%	\hspace*{.05cm}&c^*&\hspace*{.05cm}
%    \end{bmatrix}
%	&
%	\begin{bmatrix}
%	\hspace*{-.05cm}	D_{M+1} & \cdots & D_{N_h}	\hspace*{-0.02cm}
%	\end{bmatrix}
%	\\
%	\begin{bmatrix}
%		\hspace*{0.2cm}	0 && \cdots && 0	\hspace*{0.2cm} \\
%		\vdots && \ddots && \vdots \\
%		0 && \cdots && 0
%	\end{bmatrix}
%	&
%	\begin{bmatrix}
%		D_{M+1} \\
%		\vdots \\
%		D_{N_h}
%	\end{bmatrix}
%	&
%	\begin{bmatrix}
%	\hspace*{0.1cm}&&	& &&& 	\hspace*{0.1cm}\\
%&&	& \mathbb{A}_2 &	&& \\
%	& &	&&	&&
%	\end{bmatrix}
%\end{bmatrix}
%$$

\begin{figure}[H]
	\centering
	\includegraphics[scale=.9]{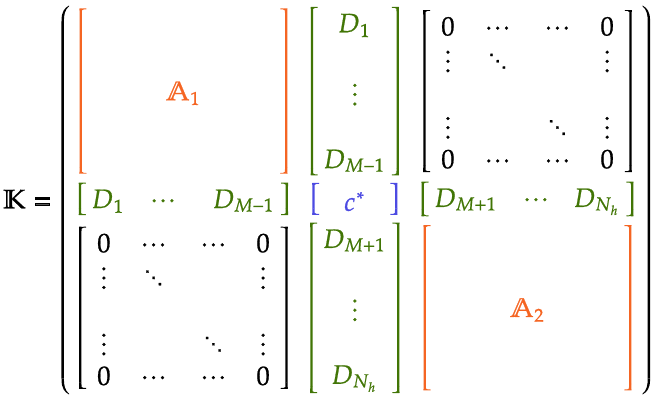}
	\caption{Structure of the stiffness matrix $\mathbb{K}$}
	\label{fig:matrixK}
\end{figure}

$\mathbb{A}_k$ ($k=1,2$) is the corresponding stiffness matrix to \eqref{FractionalDVF}, 
$$
\begin{array}{rclll}
 D_i&=&\dis \frac{C(s)}{2} \int_\R \int_\R \underline{\sigma}(x,y) \frac{(\varphi^s(x)-\varphi^s(y))(\phi_i(x)-\phi_i(y))}{|x-y|^{1+2s}} \, \mathrm{d}y \, \mathrm{d}x \quad \forall i\in\llbracket 1,N_h\rrbracket\setminus\{M\} ,\\\\
c^*&=&\dis\frac{C(s)}{2} \int_\R \int_\R \underline{\sigma}(x,y) \frac{|\varphi^s(x)-\varphi^s(y)|^2}{|x-y|^{1+2s}} \, \mathrm{d}y \, \mathrm{d}x,
\end{array}
$$
$G_i=\dis\int_I f \phi_i$ for any $i\in\llbracket 1,N_h\rrbracket\setminus\{M\} $  and $G_M=\dis\int_I f \varphi^s$. In the matrix $\mathbb{K}$, the positions of the entries $D$ and $c^*$ depend on the location of the interface. After solving the linear system 
$
\mathbb{K} U_{0,h}^s = G,
$
we obtain the full solution as
$
\overline{U}_h^s = U_{0,h}^s + W_h^s,
$
where 
$
W_h^s = \overline{U}_{h}^s(M) \cdot \big(\varphi^s(x_1), \dots, 0, \dots, \varphi^s(x_{N_h})\big)^T.
$
Here, the zero appears specifically at the index corresponding to the position of $b$. 

%This is because this was the concept of the new model where we proposed  $v(x)=v_s(x)+v(b)\varphi_s(x)$. 

%\bigbreak

\begin{rem}
	In order to implement $D_i$ and $c^*$ numerically, we interpolate $\varphi^s$ using a finite-dimensional approximation. Specifically, let  $\varphi_h^s$ denote the interpolant of $\varphi^s$ in the finite element space spanned by basis functions $\{\phi_j\}_{j=1}^{N_h}$, $i.e.$  $\varphi_h^s=\dis\sum_{j=1}^{N_h}\varphi^s(x_j)\phi_j$. Then, the discrete counterparts of $D_i$ and $c^*$ are obtained as follows:
	$$
	\begin{array}{rclll}
		D_{h,i}&=&\dis \frac{C(s)}{2} \dis\sum_{j=1}^{N_h}\varphi^s(x_j)\int_\R \int_\R \underline{\sigma}(x,y) \frac{(\phi_i(x)-\phi_i(y))(\phi_j(x)-\phi_j(y))}{|x-y|^{1+2s}} \, \mathrm{d}y \, \mathrm{d}x \quad \forall i\in\llbracket 1,N_h\rrbracket\setminus\{M\} ,\\\\
		c_h&=&\dis\frac{C(s)}{2} \dis\sum_{j=1}^{N_h}(\varphi^s(x_j))^2\int_\R \int_\R \underline{\sigma}(x,y) \frac{|\phi_j(x)-\phi_j(y)|^2}{|x-y|^{1+2s}} \, \mathrm{d}y \, \mathrm{d}x
		\\\\
		&& \hspace*{.5cm} +C(s)\dis \sum_{k=1}^{N_h-1} \sum_{j=k+1}^{N_h} \varphi^s(x_k) \varphi^s(x_j) \int_\R \int_\R \underline{\sigma}(x,y) \frac{(\phi_k(x)-\phi_k(y))(\phi_j(x)-\phi_j(y))}{|x-y|^{1+2s}} \, \mathrm{d}y \, \mathrm{d}x.
	\end{array}
	$$
However, we know that $ \|\varphi_h^s - \varphi^s\|_{\widetilde{H}^{s}(I)} \leq C h^{1-s}\| \varphi^s - \varphi \|_{H^1(I)}$ for some $C>0$ (see \cite[Subsubsection 3.3.1]{BorthagarayThesis}) and that $\| \varphi^s - \varphi \|_{H^1(I)} \lesssim 1-s$ (see Proposition \ref{cv-phi}).   	
		For sufficiently small $ h $ or $1-s$, $ \varphi_h^s $ converges to $ \varphi_s $ in the energy norm. 
%		In the following, we prove the convergence of $u_h^s$ to the exact solution of  \eqref{ClassicalProblem} as $ s \to 1^- $ and $h \to 0^+$, considering $ c^* $ and $ D_i $ (instead of $c_h$ and $D_{h,i}$).
%In the following, we focus on the simplified model introduced below, obtained by setting $D_i=0$ for all interior nodes, and study its convergence to the local solution as $s \to 1^-$ and $h \to 0^+$.
\end{rem}
 \subsection{The simplified new model}\label{Subsection3.3} 
 While the new model introduced in Subsection \ref{Subsection3.2} faithfully captures the nonlocal nature of the problem, the analysis of the simplified formulation and the numerical experiments reveal a striking observation: the coupling terms collected in the vector $D$ become asymptotically negligible in the regime $1-s=o(h)$, which motivates the reduced formulation studied below.
%  the full coupling via the vector $D$ may not be necessary to achieve accurate approximations.
% This insight motivates 
 Indeed, in the simplified version, the interface function $\varphi^s$ is only coupled with itself, i.e., we set all entries $D_i = 0$ and preserve only the scalar correction $c^* = a^s{\underline{\sigma}}(\varphi^s,\varphi^s)$.
 
 From a practical standpoint, this leads to a block-diagonal structure for the stiffness matrix $\mathbb{K}^*$, where the two local problems on $I_1$ and $I_2$ are completely independent, see Figure~\ref{fig:matrixKstar}. 
 \begin{figure}[H]
 	\centering
 	\includegraphics[scale=.9]{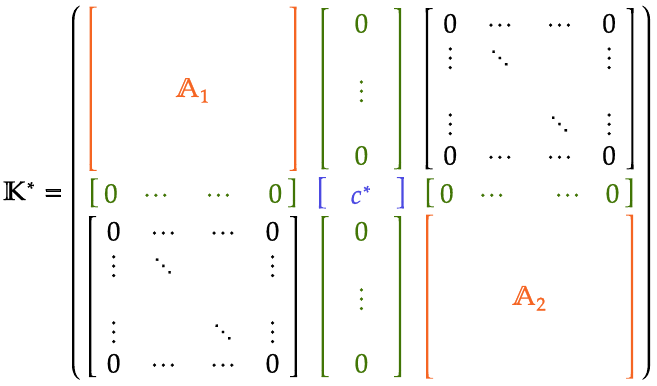}
 	\caption{Structure of the stiffness matrix $\mathbb{K}^*$}
 	\label{fig:matrixKstar}
 \end{figure}
 This greatly simplifies the assembly and resolution of the linear system while preserving the core mechanism of enrichment via $\varphi^s$.
 
 The simplified model therefore approximates the solution as
 
 \[
 u_h^s = u_{0,h}^s + u_h^s(b)\varphi^s,
 \]
 
 and $u_h^s(b)$ is now computed from the equation
 
 \[
 u_h^s(b) = \frac{\dis\int_I f(x) \varphi^s(x)\;\mathrm{d}x}{c^*}.
 \]
 
 Proposition \ref{PropLimitsDic} below shows that the neglected coupling coefficients satisfy
 $
\sup\limits_{i\in\llbracket 1,N_h\rrbracket\setminus\{M\}}|	D_i| \lesssim \dfrac{1-s}{h}.
 $
 Hence, under the asymptotic regime $1-s=o(h)$, the simplified model is asymptotically consistent with the reconstructed formulation. In the remainder of the paper, we study the convergence of this simplified model.

%\textcolor{blue}{\begin{rem}
%		We have conducted some numerical simulations to compare $U_h^s$ and $U$ (which is the vector representing the exact solution to \eqref{ClassicalProblem} at each node).  Surprisingly, using the block diagonal structure for $\mathbb{B}$, that is, in addition to the homogeneous subproblems solved by $u_1$ and $u_2$, $\varphi^s$ only interacts with itself (all $D_i$'s =0), seems to be enough to have convergence as $s\to 1^-$.
%\end{rem}}

\section{Convergence of the simplified nonlocal model to the classical one}\label{Section4}
In this section, we prove that the solution of the simplified model converges to the exact local solution as $s \to 1^-$ and $h \to 0^+$. More precisely, we consider
\begin{itemize}
	\item $u_h^s = u_{0,h}^s + u_h^s(b)\varphi^s$, the approximate solution of the simplified model introduced in Subsection \ref{Subsection3.3};
	\item $u = u_0 + u(b)\varphi$, the exact solution of the local problem.
\end{itemize}
Throughout this section, we assume that
$
s \in \left(\frac12,1\right),
$
so that pointwise evaluation at $x=b$ is well-defined for functions in
$\widetilde H^s(I)$. In particular, the quantities $u^s(b)$ and $u_h^s(b)$
are meaningful.   {Furthermore, we assume that
$\dis
	\frac{|\sigma_2|}{\sigma_1} \neq \frac{1-b}{b},
	$
	so that, by Theorem \ref{thm:local-Tcoercive}, the local problem \eqref{ClassicalProblem} is well posed in the Hadamard sense. Equivalently,
$\dis
	\frac{\sigma_1(1-b)+\sigma_2 b}{b(1-b)} \neq 0.
	$}

\smallbreak
We begin with a preliminary consistency result on the quantities defining the
simplified model. More precisely, we estimate the neglected coupling
coefficients $D_i$ and compare the scalar coefficient $c^*$ with its local
counterpart. 
%These asymptotic estimates justify the simplified formulation and
%will be used in the convergence analysis below.

\begin{prop}\label{PropLimitsDic}~\\
	Assume that $1-s$ is sufficiently small. Then, the following two properties hold :
	\\
{\em (i)}
	$
\sup\limits_{i\in\llbracket 1,N_h\rrbracket\setminus\{M\}}|	D_i| \lesssim \dfrac{1-s}{h}\;;
	$
	
	\noindent {\em (ii)} $|c^*-\widetilde{c}|\lesssim 1-s$, where $\widetilde{c}:=\frac{\sigma_1(1-b)+\sigma_2b}{b(1-b)}\neq0$.
\end{prop}
\begin{rem}
From the point~{\em (i)}, if $1-s = o(h)$, the contribution $D_i$ vanishes when $h \to 0^+$, which is consistent with the simplified new model.
\end{rem}

\begin{proof}[Proof of Proposition \ref{PropLimitsDic}]
	Assume that $1-s$ is sufficiently small.

\noindent(i) For any  $i\in\llbracket 1,N_h\rrbracket\setminus\{M\}$,

we have
$$
\begin{array}{lll}
 D_i&=&\dis \frac{C(s)}{2} \int_\R \int_\R \underline{\sigma}(x,y) \frac{(\varphi^s(x)-\varphi^s(y))(\phi_i(x)-\phi_i(y))}{|x-y|^{1+2s}} \, \mathrm{d}y \, \mathrm{d}x\\\\
 &=&\dis  C(s) \int_{I} \underbrace{\int_{\R\setminus I} \ldots \, \mathrm{d}y}_{J_{1,i}} \, \mathrm{d}x + \frac{C(s)}{2}\int_{I_1} \underbrace{\int_{I_1} \ldots \, \mathrm{d}y}_{J_{2,i}} \, \mathrm{d}x+\frac{C(s)}{2} \int_{I_2} \underbrace{\int_{I_2} \ldots \, \mathrm{d}y}_{J_{3,i}}\mathrm{d}x
% + \underbrace{C(s) \int_{I_1}\int_{I_2} \ldots \,\mathrm{d}y \, \mathrm{d}x}_{IT \textcolor{blue}{\textit{(Interaction term)}}}
 \\\\
  &=& \dis C(s)  \int_{I} J_{1,i}(x) \mathrm{d}x +\frac{C(s)}{2} \int_{I_1} J_{2,i}(x) \mathrm{d}x +\frac{C(s)}{2}  \int_{I_2} J_{3,i}(x) \mathrm{d}x.
\end{array}
$$
%\textcolor{blue}{In the new
%	model we set $\sigma_3=0$, hence $IT=0$; we keep $IT$ for the moment just to highlight
%	where cross-interactions enter ;  it does not affect the final estimate.
%}
--- \textbf{\em First,} $$
\begin{array}{lll}
	\dis C(s)  \int_{I} |J_{1,i}(x)| \mathrm{d}x&\leq&C(s)\max\{|\sigma_1|,|\sigma_2|\}\dis\int_I \varphi^s(x)\phi_i(x) \int_{\R\setminus I} \frac{dy}{|x-y|^{1+2s}}\mathrm{d}x
	\\\\ &=&\dfrac{C(s)}{2s}\max\{|\sigma_1|,|\sigma_2|\} \dis\int_I \left[ \frac{1}{x^{2s}}+ \frac{1}{(1-x)^{2s}}  \right]\varphi^s(x)\phi_i(x)\mathrm{d}x.
\end{array}
$$
\begin{itemize}
	\item[$\bullet$] Near 0, the dominant behavior comes from  $\dfrac{1}{x^{2s}}$, then
	$$\hspace*{-.5cm}
C(s)	 \dis\int_0^{2h} \left[ \frac{1}{x^{2s}}+ \frac{1}{(1-x)^{2s}}  \right]\varphi^s(x)\phi_i(x)\mathrm{d}x\sim
 \dfrac{C(s)}{hb^s} \dis\int_0^{2h} x^{1-s}\mathrm{d}x\lesssim (1-s)h^{1-s},
	$$	
	as $C(s)\sim 1-s$ as $s\to 1^-$, see \cite[Corollary 4.2]{DiNezzaPalatucciValdinoci}.
		\item[$\bullet$] Analogously, near 1 we get
	$$
	\dfrac{C(s)}{2s}\max\{|\sigma_1|,|\sigma_2|\} 	 \dis\int_{1-2h}^1 \left[ \frac{1}{x^{2s}}+ \frac{1}{(1-x)^{2s}}  \right]\varphi^s(x)\phi_i(x)\mathrm{d}x\lesssim (1-s)h^{1-s}.
	$$
		\item[$\bullet$] Inside $(2h,1-2h)$, we have $\dis \frac{1}{x^{2s}}+ \frac{1}{(1-x)^{2s}} \leq \frac{2}{h^{2s}}$. Therefore,
			$$
		\dfrac{C(s)}{2s}\max\{|\sigma_1|,|\sigma_2|\} 	 \dis\int_{(i-1)h}^{(i+1)h} \left[ \frac{1}{x^{2s}}+ \frac{1}{(1-x)^{2s}}  \right]\varphi^s(x)\phi_i(x)\mathrm{d}x\leq  \dfrac{2C(s)}{s}h^{1-2s}\lesssim (1-s)h^{1-2s}.
		$$
%		\textcolor{gray}{$h^{1-2s}=\dfrac{1}{h} e ^{2(1-s)\log(h)}\approx \dfrac{1}{h}[1+2(1-s)\log(h)]$.}
\end{itemize}
Hence, $\dis C(s)  \int_{I} |J_{1,i}(x)| \mathrm{d}x\lesssim (1-s)h^{1-2s}$.
\medbreak 
\noindent --- \textbf{\em Second,} let us treat $J_{2,i}$ and $J_{3,i}$. Near $x$, $\varphi^s(x)-\varphi^s(y)$ (resp. $\phi_i(x)-\phi_i(y)$) behaves at the first order like $(\varphi^s(x))' (x-y)$ (resp. $\phi'_i(x) (x-y)$). Then, we obtain

$$\dis\frac{C(s)}{2} \left|\int_{I_1} J_{2,i}(x) \mathrm{d}x\right|\approx \dis\frac{C(s)}{2} \left|\int_{I_1} \sigma_1 (\varphi^s(x))'\phi'_i(x)\frac{x^{2-2s}+(b-x)^{2-2s}}{2(1-s)}\mathrm{d}x\right|\leq  \dis\frac{C(s)}{2(1-s)} \left|\int_{I_1} \sigma_1 (\varphi^s(x))'\phi'_i(x)\mathrm{d}x\right|$$

and analogously,

$$\dis\frac{C(s)}{2} \left|\int_{I_2} J_{3,i}(x) \mathrm{d}x\right|\lesssim  \dis\frac{C(s)}{2(1-s)} \left|\int_{I_2} \sigma_2 (\varphi^s(x))'\phi'_i(x)\mathrm{d}x\right|.$$
%\approx  \dis\frac{C(s)}{2} \int_{I_2} |\sigma_2 (\varphi^s(x))'\phi'_i(x)|\frac{(x-b)^{2-2s}+(1-x)^{2-2s}}{2(1-s)}\mathrm{d}x
Thus,
$$\dis\frac{C(s)}{2}\left|\int_{I_1} J_{2,i}(x) \mathrm{d}x + \int_{I_2} J_{3,i}(x) \mathrm{d}x\right|\lesssim  \dis\frac{C(s)}{2(1-s)} \sum_{k=1,2}\left|\int_{I_k} (\varphi^s(x))'\phi'_i(x)\mathrm{d}x \right|
\lesssim {\dfrac{1-s}{h}}.$$

\bigbreak

\noindent(ii) Let us denote $\widetilde{c}=\dfrac{\sigma_1(1-b)+\sigma_2b}{b(1-b)}$. We have
$
|c^* -\widetilde{c}|\leq |\alpha_1(s)|+|\alpha_2(s) -\widetilde{c}|
$. In the subsequent, we precise the quantities $\alpha_k(s)$.
\\
	\noindent --- \textbf{\em First,} we treat the exterior contribution
\[
\alpha_1(s)
:= C(s)\int_I\int_{\mathbb R\setminus I}\underline\sigma(x,y)
\frac{|\varphi^s(x)-\varphi^s(y)|^2}{|x-y|^{1+2s}}\,\mathrm{d}x \,\mathrm{d}x.
\]
Since $\varphi^s\equiv0$ on $\mathbb R\setminus I$, we have $\varphi^s(y)=0$ for $y\notin I$ and thus
\[
|\alpha_1(s)|
\le C(s)\max\{|\sigma_1|,|\sigma_2|\}\int_I |\varphi^s(x)|^2
\left(\int_{\mathbb R\setminus I}\frac{dy}{|x-y|^{1+2s}}\right)\mathrm{d}x.
\]
Moreover, for $x\in I$,
\[
\int_{\mathbb R\setminus I}\frac{\mathrm{d}y}{|x-y|^{1+2s}}
=\frac1{2s}\left(\frac1{x^{2s}}+\frac1{(1-x)^{2s}}\right).
\]
Therefore
\[
|\alpha_1(s)|
\le \frac{C(s)}{2s}\max\{|\sigma_1|,|\sigma_2|\}
\int_0^1 |\varphi^s(x)|^2\left(\frac1{x^{2s}}+\frac1{(1-x)^{2s}}\right)\mathrm{d}x.
\]
The integrand is locally integrable for $s\in(\tfrac12,1)$, then the above integral is bounded uniformly for $s$ close to $1$ (it depends only on $b$).
Finally, using the asymptotic $C(s)\sim 2(1-s)$ as $s\to1^-$, we obtain
\[
|\alpha_1(s)|\lesssim 1-s.
\]

	\medskip
\noindent --- \textbf{\em Second,} we estimate the interior contribution
\[
\alpha_2(s)
:=\frac{C(s)}{2}\left[\sigma_1 \int_{I_1} \int_{I_1} 
\frac{|\varphi^s(x)-\varphi^s(y)|^2}{|x-y|^{1+2s}}\,\mathrm{d}y\,\mathrm{d}x
+\sigma_2 \int_{I_2} \int_{I_2} 
\frac{|\varphi^s(x)-\varphi^s(y)|^2}{|x-y|^{1+2s}}\,\mathrm{d}y\,\mathrm{d}x\right].
\]
A direct computation (based on the change of variables $x=bt$ on $I_1$ and
$x=b+(1-b)t$ on $I_2$, and the Beta--Gamma identity
$B(p,q)=\Gamma(p)\Gamma(q)/\Gamma(p+q)$) yields
\[
\alpha_2(s)
=\frac{C(s)s^2}{4(1-s)}\left[\frac{\sigma_1}{b^{2s-1}}+\frac{\sigma_2}{(1-b)^{2s-1}}\right]
\Bigl(1+\Gamma(2s-1)\Gamma(3-2s)\Bigr).
\]
	As $s\to1^-$ we have
	\[
	\Gamma(2s-1)\Gamma(3-2s)\to 1,
	\qquad
	\frac{C(s)s^2}{4(1-s)}\to \frac12,
	\]
	hence
	\[
	\frac{C(s)s^2}{4(1-s)}\Bigl(1+\Gamma(2s-1)\Gamma(3-2s)\Bigr)=1+o(1).
	\]
	Moreover,
	\[
	b^{1-2s}=\frac1b\,e^{2(1-s)\log b}
	=\frac1b\Bigl(1+2(1-s)\log b+o(1-s)\Bigr),
	\]
	and similarly
	\[
	(1-b)^{1-2s}
	=\frac1{1-b}\Bigl(1+2(1-s)\log(1-b)+o(1-s)\Bigr).
	\]
	Therefore
	\[
	\alpha_2(s)
	=\left(\frac{\sigma_1}{b}+\frac{\sigma_2}{1-b}\right)
	+2(1-s)\left(\frac{\sigma_1\log b}{b}+\frac{\sigma_2\log(1-b)}{1-b}\right)
	+o(1-s),
	\]
	and in particular
	\[
	|\alpha_2(s)-\widetilde c|\lesssim 1-s.
	\]
	Combining the bounds for $\alpha_1$ and $\alpha_2$ yields $|c^*-\widetilde c|\lesssim 1-s$.
\end{proof}

\medbreak
{The following proposition concerns the particular choice of $\varphi^s$.}

\begin{prop}\label{cv-phi}
	For $1-s$ sufficiently small, it holds that ;
	
		\noindent (i)
	$\|\varphi^s-\varphi\|_{\widetilde{H}^s(I)}\lesssim 1-s$;
	
	\noindent (ii)
	$\|\varphi^s-\varphi\|_{H^1(I)}\lesssim 1-s$. 
\end{prop}

\begin{proof}  For $x\in(0,b)$, let $z = x/b \in (0,1)$. Expanding $\varphi^s(x) = z^s/b^s$ near $s=1$, we get
 	\[
 	|\varphi^s(x)-\varphi(x)| = \left| z^s - z \right| \approx (1-s) \frac{x}{b} \left| \log \left( \frac{x}{b} \right) \right|,
 	\]
 	since $z^s \approx z + (s-1) z \log z$. The function $z |\log z|$ attains its maximum at $z=e^{-1}$, so
 	\[
 	\|\varphi^s-\varphi\|_{L^\infty(0,b)} \lesssim e^{-1} (1-s).
 	\]
 	Thus, $\|\varphi^s-\varphi\|_{L^2(0,b)} \lesssim (1-s)$. In addition, we have 
 	$\dis
 	(\varphi^s)'(x) = \frac{s}{b} z^{s-1}$ and $\dis(\varphi)'(x) = \frac{1}{b}.
 	$
 	Expanding $s z^{s-1} \approx [1-(1-s)](1-(1-s)\log z) \approx 1 - (1-s) (1 + \log z)$, we get
 	\[
 	|(\varphi^s)'(x) - (\varphi)'(x)| = \frac{1}{b} |s z^{s-1} - 1| \approx \frac{1-s}{b} |1 + \log z|.
 	\]
 {Since $\dis\int_0^b |1 + \log(x/b)|^2 dx = b \int_0^1 |1 + \log z|^2 dz < +\infty$}, we have
 	$
 	\|(\varphi^s)' - (\varphi)'\|_{L^2(0,b)}^2 \approx  C (1-s)^2,
 	$
 	so $\|(\varphi^s)' - (\varphi)'\|_{L^2(0,b)} \lesssim 1-s$. Identical analysis for $x \in (b,1)$ yields similar bounds. Thus, we obtain
 	\[
 	\|\varphi^s-\varphi\|_{H^1(I)}^2 = \|\varphi^s-\varphi\|_{L^2(I)}^2 + \|(\varphi^s)'-\varphi)'\|_{L^2(I)}^2 \lesssim (1-s)^2,
 	\]
 	implying (ii).
\\ 	
 For (i), we use the embedding $H^1(\mathbb{R}) \subseteq \widetilde{H}^s(I)$. More precisely, for any measurable function $w:\R\to \R$ belonging to $H^1(\R)$, 
 \[
 \|w\|_{\widetilde{H}^s(I)}^2 \lesssim \frac{C(s)}{ \ell^{2s}} \|w\|_{L^2(\mathbb{R})}^2 + \frac{C(s)}{ (1-s) \ell^{2-2s}} \|w'\|_{L^2(\mathbb{R})}^2, \;  0 < \ell < +\infty,
 \]
 see \cite[Proposition 1.25]{Leoni}.
 Let us set $w = \varphi^s - \varphi$, extended by zero outside $I$. Using $\|\varphi^s-\varphi\|_{L^2(\mathbb{R})} = \|\varphi^s-\varphi\|_{L^2(I)} \lesssim 1-s$ and $\|w'\|_{L^2(\mathbb{R})} \lesssim 1-s$. Therefore, we get
 \[
 \|w\|_{\widetilde{H}^s(I)}^2 \lesssim \frac{1-s}{\ell^{2s}} (1-s)^2 + \frac{1-s}{(1-s) \ell^{2-2s}} (1-s)^2 \approx (1-s)^3 + (1-s)^2 \approx (1-s)^2,
 \]
 since $\ell^{2s} \to \ell^2$ and $\ell^{2-2s} \to 1$ as $s \to 1^-$. Thus, (i) holds.
 \end{proof}

In the following proposition, we assume that \( D_i = 0 \) for all \( i \in \llbracket 1, N_h \rrbracket \setminus \{M\} \), under the natural condition that \( 1 - s = o(h) \).
\begin{prop}\label{cv-ub-gen}
	Let $s\in(\frac12,1)$. Assume that $f\in H^{-s}(I)$.  Let $u$ the exact solution to \eqref{ClassicalProblem} and $u_h^s$ the approximated solution given by \eqref{uhs}.
	Then, for $1-s$ sufficiently small, it holds that $$|u_h^s(b)-u(b)|\lesssim (1-s)\|f\|_{H^{-s}(I)}.$$ %In particular $|u_h^s-u(b)|\to 0 \;\;\text{as}\;\; s\to 1^-$.
\end{prop}

\begin{proof}

\noindent
--- \textbf{\em First,}  we have $$-\dis\int_0^1(\sigma(x) u'(x))'\mathrm{d}x=\int_0^1f(x)\mathrm{d}x=\sigma_1u'(0)-\sigma_2u'(1).$$ Furthermore, \begin{itemize}
	\item in $I_1$, $-(\sigma_1 u')'=f$ implies $u(b)=\dis\int_0^b u'(t)\mathrm{d}t=bu'(0)-\sigma_1^{-1} \dis\int_0^b \left[ \dis\int_0^t f(x) \mathrm{d}x\right] \mathrm{d}t $ ;
	\item in $I_2$, $-(\sigma_2 u')'=f$ implies $u(b)=-\dis\int_b^1 u'(t)\mathrm{d}t=-(1-b)u'(1)-\sigma_2^{-1} \dis\int_b^1 \left[ \dis\int_t^1 f(x) \mathrm{d}x\right] \mathrm{d}t $.
\end{itemize} 
By exchanging the order of integration (Fubini), we obtain $\dis\int_0^b \left[ \dis\int_0^t f(x) \mathrm{d}x\right] \mathrm{d}t=\int_0^b f(x) (b-x)\mathrm{d}x$ and $\dis\int_b^1 \left[ \dis\int_t^1 f(x) \mathrm{d}x\right] \mathrm{d}t=\int_b^1 f(x) (x-b)\mathrm{d}x$.

Then,
\begin{equation}
	\begin{array}{rlll}
	\underbrace{\left[ \dfrac{\sigma_1}{b}+\dfrac{\sigma_2}{1-b}\right]}_{\widetilde{c}} u(b)&=\dis \sigma_1u'(0)-\sigma_2u'(1)-\dfrac{1}{b}  \int_0^b f(x) (b-x)\mathrm{d}x -\dfrac{1}{1-b}\int_b^1 f(x) (x-b)\mathrm{d}x\\\\
		&=\dis \int_0^1f(x)\mathrm{d}x-\dfrac{1}{b}  \int_0^b f(x) (b-x)\mathrm{d}x -\dfrac{1}{1-b}\int_b^1 f(x) (x-b)\mathrm{d}x
		\\\\
	&= \dis \int_0^b f(x) \frac{x}{b}\mathrm{d}x +\int_b^1 f(x) \dfrac{1-x}{1-b}\mathrm{d}x
		\\\\
	&=\dis \int_0^1 f(x) \varphi(x)  \mathrm{d}x.
	\end{array}
\end{equation}
--- \textbf{\em Second,} we have 
\begin{equation}\label{c*uhsb-}
	c^*u_h^s(b)=\dis\int_0^1 f(x)\varphi^s(x)\mathrm{d}x.%=\dfrac{b^{\alpha+1}}{\alpha+s+1}+\dfrac{B(1-b;s+1,\alpha+1)}{(1-b)^s}.
\end{equation}

Therefore, we get 
\begin{equation}
	\begin{array}{rlll}
		|u_h^s(b)-u(b)|&=\left|\dfrac{1}{c^*} \dis\int_0^1 f(x)\varphi^s(x)\mathrm{d}x -\dfrac{1}{\widetilde{c}}\dis \int_0^1 f(x) \varphi(x)  \mathrm{d}x\right|
		\\\\
		& \leq \dfrac{1}{|c^*|}\dis\int_0^1 |f(x)(\varphi^s(x)-\varphi(x))|\mathrm{d}x +\left|\dfrac{1}{c^*}-\dfrac{1}{\widetilde{c}}\right|\dis\int_0^1 |f(x)\varphi(x)|\mathrm{d}x
		\\\\
			& \leq \dfrac{1}{|\widetilde{c}|}\dis\int_0^1 |f(x)(\varphi^s(x)-\varphi(x))|\mathrm{d}x +\left|\dfrac{1}{c^*}-\dfrac{1}{\widetilde{c}}\right|
		\left[	\dis\int_0^1 |f(x)(\varphi^s(x)-\varphi(x))|\mathrm{d}x+
		\dis\int_0^1 |f(x)\varphi(x)|\mathrm{d}x\right]
			\\\\
		& \lesssim \|f\|_{H^{-s}(I)} \|\varphi^s-\varphi\|_{\widetilde{H}^{s}(I)}+ (1-s) \|f\|_{H^{-s}(I)} \left[\|\varphi^s-\varphi\|_{\widetilde{H}^{s}(I)}+\|\varphi\|_{{H}^{1}(I)}\right],
	\end{array}
\end{equation}	
as 	$\left|\dfrac{1}{c^*}-\dfrac{1}{\widetilde{c}}\right|\lesssim 1-s$. Using Proposition \ref{cv-phi}, we conclude the proof.
\end{proof}

\bigbreak

Let us recall that $u_0^s=u_1$ in $I_1$ solution to  subproblem-1 \eqref{SubFractionalProblem1}, and $u_0^s=u_2$ in $I_2$ solution to subproblem-2 \eqref{SubFractionalProblem2}.
\begin{prop} \label{cv-us}
 Assume that $s\in (0,1)$ and {$f\in L^\infty(I)$}. Let $u_0^s\in \widetilde{H}^s(I_i)$ and $u_0\in H^1(I_i)$ be the weak solutions to nonlocal and local subproblems-$i$ respectively for $i=1,2$. Then, there exists a constant $\Lambda=\Lambda(I,u_0)\geq 0$ such that  %For $1-s$ be sufficiently small, it holds that
\begin{equation}
	\|u_0^s-u_0\|_{\widetilde{H}^s(I_i)}\leq \Lambda \sqrt{1-s}\|f\|_{L^\infty\left(I\right)}.
\end{equation} 

\end{prop}

\begin{proof} Let us denote $\widehat{C}(s):=4\dfrac{\Gamma(\frac32)}{\sqrt{\pi}}(1-s)$ and   $\|\cdot \|^2_{s}:=\dfrac{\widehat{C}(s)}{{C}(s)}\|\cdot\|^2_{\widetilde{H}^s(I_i)}$
Further, let us denote  $\widehat{f}:=\dfrac{\widehat{C}(s)}{{C}(s)}f$ and $\widehat{(-\Delta)^s}:=\dfrac{\widehat{C}(s)}{{C}(s)}(-\Delta)^s$.
\begin{figure}[H]
	\centering
	\includegraphics[scale=0.25]{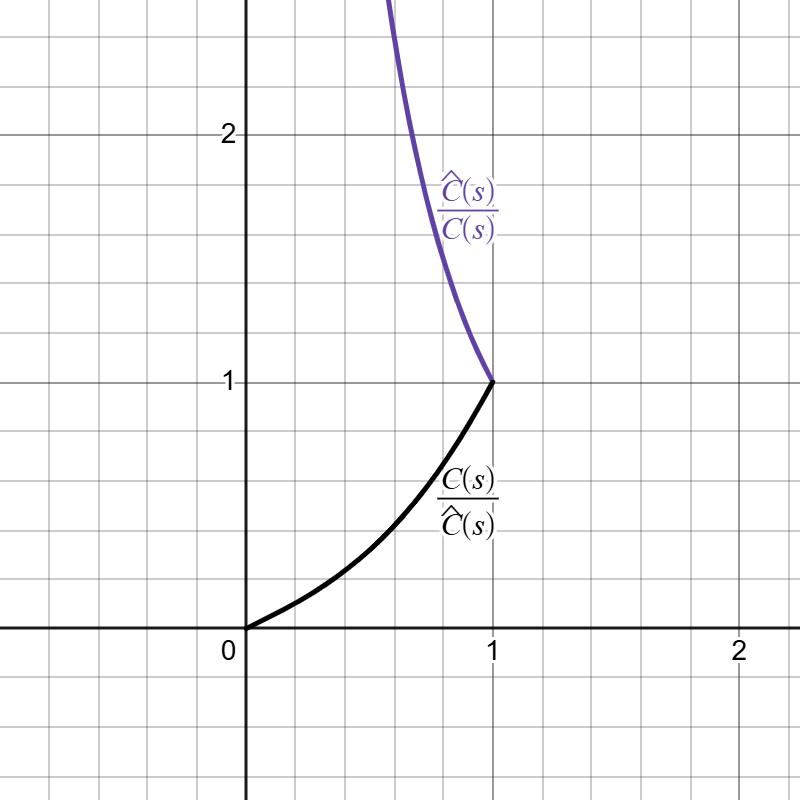}
		\caption{Behavior of $\dfrac{\widehat{C}(s)}{C(s)}$ and its reciprocal}
	\label{fig:Cnts}
\end{figure}
Then, solving $(-\Delta)^su_i=f$ in $I_i$ is equivalent to solve $\widehat{(-\Delta)^s} u_i=\widehat{f}$ in $I_i$. Thanks to \cite[Theorem 1]{BungdelTeso}, we have
\begin{equation}
\begin{array}{rclll}
	\|u_0^s-u\|^2_{\widetilde{H}^s(I_i)}&=& \dfrac{{C}(s)}{\widehat{C}(s)}\|u_0^s-u\|^2_s
	\\
	&\leq& \|u_0^s-u\|^2_s
	\\
	&\leq& K\Bigg[1-s+\Bigg\|\widehat{f}-\dfrac{C(s)}{\widehat{C}(s)}\widehat{f}\Bigg\|_{L^1(I_i)}\Bigg]
	\\
	&\leq& K\Bigg[1-s+\Bigg|1-\dfrac{C(s)}{\widehat{C}(s)}\Bigg|\Big\|\widehat{f}\Big\|_{L^1(I_i)}\Bigg]
		\\
	&\leq&  K\max\{1,\rho\big\|\widehat{f}\|_{L^1(I_i)}\big\} (1-s),
\end{array}
\end{equation}
where $K=K(I, u_0, f)\geq 0$ 
and $\rho>0$ (independent of $s$) arises from the inequality $\Bigg|1-\dfrac{C(s)}{\widehat{C}(s)}\Bigg|\leq \rho(1-s)$, see \cite[Appendix A]{BungdelTeso}.\end{proof}

\begin{prop}\label{cv-u0hs}
Let $s\in(\frac12,1)$. Assume that $f\in C^\beta(\overline{I})$ for some $\beta>0$.
	Then for $h$ sufficiently small, it holds that
\begin{itemize}
	\item[\em (i)]	\begin{equation}\label{ush-us-Hs}
		\left\|u^s_{0,h}-u^s_0\right\|_{\widetilde{H}^s(I_i)} \lesssim h^{1/2}|\log(h)|\,\|f\|_{C^{\beta}(\overline{I})}.
\end{equation}

	\item[\em (ii)]  If, moreover $u_0^s=u^s-u^s(b)\varphi^s\in H^1(I)$, then	\begin{equation}\label{ush-us-H1}
	\left\|u^s_{0,h}-u^s_0\right\|_{H^1(I_i)} \lesssim h^{s-(1/2)}|\log(h)|\,\|f\|_{C^{\beta}(\overline{I})}.
	\end{equation}
\end{itemize}
\end{prop}

\noindent {\em Proof.} %Let $1-s=h^k$ for a fixed $k> 1$.

\noindent
 (i) The estimate \eqref{ush-us-Hs} follows directly from \cite[Theorem 4.7]{AcostaBor}.
 % have

%\begin{equation}
%	\left\|u^s_{0,h}-u^s_0\right\|_{\widetilde{H}^s(I_i)} \lesssim h^{1/2}|\log(h)|\,\|f\|_{C^{\beta}(\overline{I_i})}.
%\end{equation}
%Therefore, the estimate \eqref{ush-us-Hs} follows directly.

\noindent
(ii) According to \cite[Proposition 3.2]{BorCia2}, we obtain \eqref{ush-us-H1}.
%we have
%
%\begin{equation}
%	\left\|u^s_{0,h}-u^s_0\right\|_{H^1(I_i)} \lesssim h^{s-(1/2)}|\log(h)|\,\|f\|_{C^{\beta}(\overline{I_i})}.
%\end{equation}
%Then, we get    
%   \begin{equation}
%   	\left\|u^s_{0,h}-u^s_0\right\|_{H^1(I_i)} \lesssim (1-s)^{(2s-1)/(2k)}|\log(1-s)|\,\|f\|_{C^{\beta}(\overline{I})} 
%   \end{equation}
%   Using Taylor series for small $1-s$, we have
%   $$
%   (1-s)^{(2s-1)/(2k)}\approx (1-s)^{1/(2k)}-(1-s)^{1/(2k)+1}\log(1-s).
%   $$
%The proof is concluded as
% $$
%% \left\{ 
%\begin{array}{rclll}
%	(1-s)^{(2s-1)/(2k)}|\log(1-s)|&\lesssim& (1-s)^{1/(2k)}|\log(1-s)|+(1-s)^{1/(2k)+1}|\log(1-s)|^2\\
%	&\leq& (1-s)^{1/(2k)}|\log(1-s)|.
%\end{array}
%% \right.
%$$
\hfill$\square$

Finally, the main convergence result reads :
\begin{prop} Let $s\in(\frac12,1)$. Assume that $f\in C^\beta(\overline{I})$ for some $\beta>0$. 	Let $u$ be the exact solution of the local problem \eqref{ClassicalProblem}, and let $u_h^s$ be the approximate solution of the simplified model introduced in Subsection \ref{Subsection3.3}, that is, the discrete reconstructed formulation obtained by setting $D_i=0$ for all $i \in \llbracket 1,N_h \rrbracket \setminus \{M\}$.
	\\		
	Assume moreover that $1-s=o(h)$, and denote by $u^s$ the solution given by \eqref{us}. 
	Then, for $1-s$ sufficiently small, the following estimates hold:	
	\noindent
{\em (i)} 
\begin{equation}\|u^s_h-u\|_{\widetilde{H}^s(I)}\lesssim h^{1/2}|\log(h)|\|f\|_{C^{\beta}(\overline{I})}.
\end{equation}

	\noindent	
{\em (ii)} If, moreover $u_0^s=u^s-u^s(b)\varphi^s\in H^1(I)$, then
	\begin{equation}
	\|u^s_h-u\|_{H^1(I)}\lesssim h^{1/2}|\log(h)|\|f\|_{C^{\beta}(\overline{I})}.
	\end{equation}

\end{prop}

\begin{proof}
 (i) We have

\noindent  $\|u^s_h-u\|_{\widetilde{H}^s(I)}=\|u^s_{0,h}-u_h^s(b)\varphi^s-u_0-u(b)\varphi\|_{\widetilde{H}^s(I)}\leq \underbrace{\|u^s_{0,h}-u_0\|_{\widetilde{H}^s(I)}}_{K_1}+\underbrace{\|u_h^s(b)\varphi^s-u(b)\varphi\|_{\widetilde{H}^s(I)}}_{K_2}$.

\noindent Let us first treat $K_1$ :

$$K_1\leq\|u^s_{0,h}-u_0^s\|_{\widetilde{H}^s(I)}+\|u^s_{0}-u_0\|_{\widetilde{H}^s(I)}.$$

\noindent
We have

 $
\begin{array}{lllll}
\|u^s_{0,h}-u_0^s\|_{\widetilde{H}^s(I)}&\leq \|u^s_{0,h}-u_0^s\|_{\widetilde{H}^s(I_1)}+\|u^s_{0,h}-u_0^s\|_{\widetilde{H}^s(I_2)}
\\
&  \lesssim h^{1/2}|\log(h)|\,\|f\|_{C^{\beta}(\overline{I})} ,
\end{array}
$

 thanks to Proposition \ref{cv-u0hs}. Furthermore, according to  Proposition \ref{cv-us}, we obtain
\begin{equation}\label{estimate-Hs}
\|u_0^s-u_0\|_{\widetilde{H}^s(I)}\lesssim  \sqrt{1-s}\|f\|_{C^\beta(\overline{I})}.
\end{equation}

On the other hand, we have

$$
\begin{array}{lllll}
K_2&\leq&|u_h^s(b)| \|\varphi^s-\varphi\|_{\widetilde{H}^s(I)}+\|\varphi\|_{\widetilde{H}^s(I)} |u_h^s(b)-u(b)|
\\
&\lesssim &(|u(b)|+|u_h^s(b)-u(b)|) \|\varphi^s-\varphi\|_{\widetilde{H}^s(I)}+\|\varphi\|_{\widetilde{H}^s(I)} |u_h^s(b)-u(b)|
\\
&\lesssim &\|u\|_{H^1(I)} \|\varphi^s-\varphi\|_{\widetilde{H}^s(I)}+(\|\varphi^s-\varphi\|_{\widetilde{H}^s(I)}+\|\varphi\|_{H^1(I)}) |u_h^s(b)-u(b)|
\\
&\lesssim & (1-s)\|f\|_{L^2(I)}+(1-s)^2\|f\|_{H^{-s}(I)}+(1-s)\|f\|_{H^{-s}(I)}
\\
&\lesssim& (1-s)\|f\|_{C^\beta(\overline{I})}.
\end{array}
$$
thanks to Proposition \ref{cv-phi}.

Finally, we obtain
$$
\begin{array}{lllll}
\|u^s_h-u\|_{\widetilde{H}^s(I)}&\lesssim  \left[\sqrt{1-s}+h^{1/2}|\log(h)|\right]\|f\|_{C^{\beta}(\overline{I})}.
%\\
%&\lesssim  (1-s)^{1/(2k)} |\log(1-s)|  \|f\|_{C^\beta(\overline{I})}.
\end{array}
$$
Under the hypothesis $1-s=o(h)$, we have $\sqrt{1-s}=o(h^{1/2})$. This yields the desired result.

\smallbreak \noindent (ii) Regarding $\|u^s_h-u\|_{H^1(I)}$, we follow the same reasoning as (i) to obtain 
\begin{equation}
	\|u^s_h-u\|_{H^1(I)} \lesssim	h^{1/2}|\log(h)|\|f\|_{C^{\beta}(\overline{I})} + \|u^s_{0}-u_0\|_{H^1(I)}.
\end{equation}
		As $u_0^s=u^s-u^s(b)\varphi^s\in H^1(I)$,
	we have
		\begin{equation}\label{Ineq-rem-H1}
		\begin{array}{rcllll}
			\|u^s_{0}-u_0\|_{H^1(I)}& \leq 	\|u^s_{0}-u^s_{0,h}\|_{H^1(I)}+	\|u_0-u_{0,h}\|_{H^1(I)}+	\|u^s_{0,h}-u_{0,h}\|_{H^1(I)}.
		\end{array}
	\end{equation}
	For $1-s$ and $h$ sufficiently small,
		\begin{itemize}
				\item[$\bullet$] $\|u^s_{0}-u^s_{0,h}\|_{H^1(I)}\lesssim h^{s-(1/2)}|\log(h)|    \|f\|_{C^{\beta}(\overline{I})}$ thanks to Proposition \ref{cv-u0hs} ;																																																																																																																																																																																																																																																																																																																																																																																																																																																																																																																																																																																																																																																																																																																																																																													
			\item[$\bullet$] it is well known that $\|u_0-u_{0,h}\|_{H^1(I)} \lesssim h\|f\|_{L^2(I)}$ ;
		\item[$\bullet$] $\|u^s_{0,h}-u_{0,h}\|_{H^1(I)}\leq C^{1-s} h^{s-1}\|u^s_{0,h}-u_{0,h}\|_{\widetilde{H}^s(I)}$ (the inverse inequality, see \cite[Proposition 3.1]{BorCia2}).
		\\ However,
			$$
		\begin{array}{rcllll}
	\|u^s_{0,h}-u_{0,h}\|_{\widetilde{H}^s(I)}& \leq 	\|u^s_{0,h}-u^s_{0}\|_{\widetilde{H}^s(I)}+	\|u_{0,h}-u_0\|_{\widetilde{H}^s(I)}+	\|u^s_{0}-u_{0}\|_{\widetilde{H}^s(I)}.
		\end{array}
		$$
		On the other hand, we have
		\begin{itemize}
			\item[$\bullet$] $\|u^s_{0,h}-u^s_{0}\|_{\widetilde{H}^s(I)}\lesssim h^{1/2}|\log(h)|  \|f\|_{C^{\beta}(\overline{I})}$ by Proposition \ref{cv-u0hs} ;
			\item[$\bullet$] $\|u_{0,h}-u_0\|_{\widetilde{H}^s(I)}\lesssim\|u_{0,h}-u_0\|_{H^1(I)} \lesssim h\|f\|_{L^2(I)}$ ;
			\item[$\bullet$] $\|u^s_{0}-u_{0}\|_{\widetilde{H}^s(I)}\lesssim \sqrt{1-s}\|f\|_{L^\infty\left(I\right)}$ according to Proposition \ref{cv-us}.
		\end{itemize}
		Therefore, we get
		$$
		\|u^s_{0,h}-u_{0,h}\|_{\widetilde{H}^s(I)}\lesssim \left[\sqrt{1-s}+h^{1/2}|\log(h)| \right]   \|f\|_{C^{\beta}(\overline{I})}.
		$$
	\end{itemize}
	Hence, the inequality \eqref{Ineq-rem-H1} becomes
	$$
		\begin{array}{llll}
    	\|u^s_{0}-u_0\|_{H^1(I)}&\lesssim \left[h^{s-1}\left(\sqrt{1-s}+h^{1/2}|\log(h)|\right) + h^{s-(1/2)}|\log(h)|    \right]   \|f\|_{C^{\beta}(\overline{I})}
    	\\\\
    	&\lesssim \left[h^{s-1}\sqrt{1-s} + h^{s-(1/2)}|\log(h)|    \right]   \|f\|_{C^{\beta}(\overline{I})}.
    		\end{array}
	$$
Under the hypothesis $1-s=o(h)$, we have $(s-1)\log (h)\to0$ then $h^{s-1}\approx 1+(1-s) |\log(h)|$. Therefore, $h^{s-1}\sqrt{1-s}\approx \sqrt{1-s}+(1-s)^{3/2}|\log(h)|\lesssim \sqrt{1-s}$ and $h^{s-(1/2)}|\log(h)|=h^{1/2}h^{s-1}|\log(h)|\approx h^{1/2}|\log(h)|+(1-s)h^{1/2}|\log(h)|^2\lesssim h^{1/2}|\log(h)|$. As $\dfrac{1-s}{h^{1/2}|\log(h)|}\to 0$, we get 
	$$
	\|u^s_{0}-u_0\|_{H^1(I)}\lesssim  h^{1/2}|\log(h)|\|f\|_{C^{\beta}(\overline{I})}.
	$$
\end{proof}

\begin{rem}
	Previously, we established the convergence of $u_h^s$ to $u$ as $s \to 1^-$ and $h \to 0^+$ simultaneously. However, it is also possible to study two other distinct types of convergence:
	\begin{itemize}
		\item[1)] For a fixed $s$, the convergence of $u_h^s$ to $u^s$ as $h \to 0^+$;
		\item[2)] For a fixed $h$, the convergence of $u_h^s$ to $u_h$ as $s \to 1^-$.
	\end{itemize}
	
	\medskip
	\noindent\textbf{1) Convergence as $h \to 0^+$ for fixed $s$.} In this case, we compare the exact and approximated solutions of the simplified new model. {In the case where $s \to 1^-$ and $h \to 0^+$ simultaneously, we prove the convergence by using directly $\varphi^s$ instead of its interpolate, as
	$ \|\varphi_h^s - \varphi^s\|_{\widetilde{H}^{s}(I)} \lesssim (1-s) h^{1-s}$.	
		 We then get the result $|u_h^s(b)-u(b)|\lesssim (1-s)\|f\|_{H^{-s}(I)}$ in Proposition \ref{cv-ub-gen}. As here we need a result for a fixed $s$ and $h \to 0^+$, we reconsider the interpolate.} \\
	For a fixed $s$, we have the estimate
	\[
	\|u_h^s - u^s\|_{\widetilde{H}^s(I)} \leq \|u_{0,h}^s - u_0^s\|_{\widetilde{H}^s(I)} + \|u_h^s(b)\varphi_h^s - u^s(b)\varphi^s\|_{\widetilde{H}^s(I)}.
	\]
	Thanks to \cite[Theorem 4.7]{AcostaBor}, it holds that
	\[
	\|u_{0,h}^s - u_0^s\|_{\widetilde{H}^s(I)} \lesssim h^{1/2} |\log h| \, \|f\|_{C^\beta(\overline{I})}.
	\]
	Moreover, assuming that all $D_i = 0$ for $h$ small enough, we have
	\[
	c^* u^s(b) = \int_I f(x) \varphi^s(x)\mathrm{d}x, \quad \text{and} \quad c_h u^s_h(b) = \int_I f(x) \varphi_h^s(x)\mathrm{d}x.
	\]
	Thus, by adding and subtracting $u_h^s(b)\varphi^s$, we get
	$$
	\begin{array}{rclll}
		\|u_h^s(b)\varphi^s_h-u^s(b)\varphi^s\|_{\widetilde{H}^s(I)}&\leq [|u_h^s(b)-u^s(b)|+|u^s(b)|]\|\varphi_h^s - \varphi^s\|_{\widetilde{H}^s(I)}+ |u_h^s(b)-u^s(b)|\| \varphi^s\|_{\widetilde{H}^s(I)},
	\end{array}
	$$
	where 
	\[
	u_h^s(b) - u^s(b) = \frac{1}{c_h c^*} \left[ c^* \int_I f(x) \varphi_h^s(x)\mathrm{d}x - c_h \int_I f(x) \varphi^s(x)\mathrm{d}x \right].
	\]
Thanks to \cite[Subsubsection 3.3.1]{BorthagarayThesis}, we know that   (We hide $1-s$ as $s$ is fixed)
	\[
	\|\varphi_h^s - \varphi^s\|_{\widetilde{H}^s(I)} \lesssim h^{1 - s}.
	\]
In addition, we estimate
$$
\begin{array}{lllll} 
	|u_h^s(b)-u^s(b)|&= \dis\dfrac{1}{|c_hc^*|}\left|c^*\int_I f(x)\varphi_h^s(x)\mathrm{d}x-c_h\int_I f(x)\varphi^s(x)\mathrm{d}x+c_h\int_I f(x)\varphi_h^s(x)\mathrm{d}x-c_h\int_I f(x)\varphi_h^s(x)\mathrm{d}x\right|
	\\\\
	& = \dis\dfrac{1}{|c_hc^*|}\left|c_h \int_I f(x)[\varphi_h^s(x)-\varphi^s(x)]\mathrm{d}x+(c^*-c_h)\int_I f(x)\varphi_h^s(x)\mathrm{d}x\right|
	\\\\
	& \leq\dfrac{1}{|c^*|}\|f\|_{H^{-s}(I)} \|\varphi_h^s - \varphi^s\|_{\widetilde{H}^s(I)} + \dfrac{|c_h-c^*|}{|c_hc^*|}  \|f\|_{H^{-s}(I)} \left[ \|\varphi_h^s - \varphi^s\|_{\widetilde{H}^s(I)}+\| \varphi^s\|_{\widetilde{H}^s(I)}\right] 
	\\\\
	&\lesssim h^{1-s} \|f\|_{H^{-s}(I)} + \dfrac{|c_h-c^*|}{|c_hc^*|}  \|f\|_{H^{-s}(I)} \left[h^{1-s}+\| \varphi^s\|_{\widetilde{H}^s(I)}\right].
\end{array}
$$

	It remains to show that $|c_h - c^*|=	|a^s_{\underline{\sigma}}(\varphi_h^s, \varphi_h^s) - a^s_{\underline{\sigma}}(\varphi^s, \varphi^s)| \lesssim h^\alpha$. To this end, consider
	\[
	\begin{aligned}
		|a^s_{\underline{\sigma}}(\varphi_h^s, \varphi_h^s) - a^s_{\underline{\sigma}}(\varphi^s, \varphi^s)|
		&= |a^s_{\underline{\sigma}}(\varphi_h^s + \varphi^s, \varphi_h^s - \varphi^s)| \\
		&\leq C \|\varphi_h^s + \varphi^s\|_{\widetilde{H}^s(I)} \|\varphi_h^s - \varphi^s\|_{\widetilde{H}^s(I)} \\
		&{\lesssim h^{1 - s}}.
	\end{aligned}
	\]
When $h\to 0^+$, $c_h\to c*>0$ and $\dfrac{1}{|c_hc^*|}\leq C$.

Finally, we obtain
	\[
	\|u_h^s - u^s\|_{\widetilde{H}^s(I)} \lesssim [h^{1 - s} + h^{1/2} |\log h|] \, \|f\|_{C^\beta(\overline{I})}.
	\]
	
	\medskip
	\noindent\textbf{2) Convergence as $s \to 1^-$ for fixed $h$.} \\
	For a fixed $h$, we have
	\begin{equation} \label{NormH1-Fixed-h}
		\|u_h^s - u_h\|_{H^1(I)} \leq \|u_{0,h}^s - u_{0,h}\|_{H^1(I)} + \|u_h^s(b) \varphi^s - u_h(b) \varphi\|_{H^1(I)}.
	\end{equation}
	We note that $a(u_{0,h}, v_h)-a^s(u_{0,h}^s, v_h)  =\dfrac{1}{\sigma_k}\dis\int_{I_k} f(x)v_h(x)\mathrm{d}x-\dfrac{1}{\sigma_k}\dis\int_{I_k} f(x)v_h(x)\mathrm{d}x=0$ for any $v_h \in V_h$, where $a(u, v) := \dis\int_I  u'(x) v' (x)\mathrm{d}x$. 
	
	Moreover, for any 
	 $ u,v \in H_0^1(I)$, we have
	\[
	a^s(u, v) - a(u, v) \lesssim (1 - s)	\|u\|_{H^1(I)}\|v\|_{H^1(I)},
	\]
see Remark \ref{Proof-as-a}.

	Let us denote $E_h :=  u_{0,h}-u_{0,h}^s$. Then,
	\[
	\begin{aligned}
		\|E_h\|^2_{H^1(I)} &= a(E_h, E_h) \\
		&=   \cancel{a(u_{0,h}, E_h)}-a(u_{0,h}^s, E_h)+a^s(u_{0,h}^s,E_h)-\cancel{a^s(u_{0,h}^s,E_h)} \\
		&= a^s(u_{0,h}^s, E_h) - a(u_{0,h}^s, E_h) \\
		&\lesssim (1 - s)\|u_{0,h}^s\|_{H^1(I)}\|E_h\|_{H^1(I)}.
	\end{aligned}
	\]		
	Therefore,
	$$
	\|E_h\|_{H^1(I)}\lesssim (1 - s)\|u_{0,h}^s\|_{H^1(I)}\leq (1 - s)\left[\|u_{0,h}\|_{H^1(I)}+\|E_h\|_{H^1(I)}\right].
	$$
	For $1-s$ sufficiently small
		$$
	\|E_h\|_{H^1(I)}\lesssim (1 - s)\|u_{0,h}^s\|_{H^1(I)}\leq (1 - s)\|u_{0,h}\|_{H^1(I)}.
	$$
%	\textcolor{red}{On the other hand, we have
%		$$
%	\begin{array}{lll}
%		 \|u_h^s(b) \varphi^s - u_h(b) \varphi\|_{H^1(I)}&\leq  |u_h^s(b) - u_h(b) | [ \|\varphi^s - \varphi\|_{H^1(I)}+ \|\varphi\|_{H^1(I)}]+| u_h(b) |  \|\varphi^s - \varphi\|_{H^1(I)}
%		 \\\\
%		 &\lesssim |u_h^s(b) - u_h(b) | [(1-s)+ \|\varphi\|_{H^1(I)}]+(1-s)
%	\end{array}
%		$$
%	Further, using the midpoint rule, we have
%	$$
%	u_h(b)=\frac{1}{\tilde{c}} \frac{1}{N} \sum_{i=1}^{N} f(x_i)\varphi(x_i)
%	$$
%	and
%		$$
%	u^s_h(b)=\frac{1}{c^*} \frac{1}{N} \sum_{i=1}^{N} f(x_i)\varphi^s(x_i),
%	$$ 	
%where $x_i =\frac{2i-1}{N}$ $(N=100)$. Then, 
%$$
%\begin{array}{llll}
%|	u_h^s(b) - u_h(b) |\lesssim 1-s.
%\end{array}
%$$
%	 }
On the other hand, we have
	 	$$
	 	\begin{array}{lll}
	 		\|u_h^s(b) \varphi^s - u_h(b) \varphi\|_{H^1(I)}&\leq  |u_h^s(b) - u_h(b) | [ \|\varphi^s - \varphi\|_{H^1(I)}+ \|\varphi\|_{H^1(I)}]+| u_h(b) |  \|\varphi^s - \varphi\|_{H^1(I)}
	 		\\\\
	 		&\lesssim |u_h^s(b) - u_h(b) | [(1-s)+ \|\varphi\|_{H^1(I)}]+(1-s)| u_h(b) |
	 	\end{array}
	 	$$
	 	Further, as
	 	$$
	 	u_h(b)=\frac{1}{\tilde{c}} \int_I f(x)\varphi(x)\,\mathrm{d}x
	 	$$
	 	and
	 	$$
	 	u^s_h(b)=\frac{1}{c^*} \int_If(x)\varphi^s(x)\,\mathrm{d}x,
	 	$$ 	
	 	then
	 	$$
	 	\begin{array}{llll}
	 		|	u_h^s(b) - u_h(b) |\lesssim (1-s)\|f\|_{H^{-s}(I)}.
	 	\end{array}
	 	$$	 
	Therefore, from \eqref{NormH1-Fixed-h}, it follows that
	\[
	\|u_h^s - u_h\|_{H^1(I)} \lesssim (1 - s)\|u_{0,h}\|_{H^1(I)}.
	\]
\end{rem}

\begin{rem}\label{Proof-as-a}
		We show that, as $s\to 1^-$,	$\forall u,v \in H_0^1(I)$,
		\[
		a^s(u, v) - a(u, v) \lesssim (1 - s)	\|u\|_{H^1(I)}\|v\|_{H^1(I)}.
		\]
		\bigbreak\bigbreak
		Let us decompose $\dis \int_{\mathbb R}\int_{\mathbb R} \ldots =  \iint_{|x-y|\geq \varepsilon}\ldots + \iint_{|x-y|< \varepsilon}\ldots+2\int_{I}\int_{\mathbb R \setminus I}\ldots$.
		\begin{itemize}
			\item First, on the far-diagonal region \(|x-y|\geq\varepsilon\) the kernel is bounded by \(\varepsilon^{-1-2s}\), hence
			\[
			\begin{array}{lll}
				\dis	\iint_{|x-y|\geq\varepsilon} \frac{(u(x)-u(y))(v(x)-v(y))}{|x-y|^{1+2s}}\,\mathrm{d}y\,\mathrm{d}x&\leq\dis
				2\varepsilon^{-1-2s}\|u\|_{L^2(I)}\|v\|_{L^2(I)}
				&\leq\dis
				4C_1\varepsilon^{-1-2s}\|u\|_{H^1(I)}\|v\|_{H^1(I)}.
				\\\\
			\end{array}
			\]
			\item Second, we deal with the near-diagonal region \(|x-y|<\varepsilon\). Near $x$, we use a Taylor expansion, That is,
			\[
			u(x)-u(y)\approx u'(x)(x-y), \quad 	v(x)-v(y)\approx v'(x)(x-y).
			\] 
			Thus, 
			\[
			\begin{array}{lll}
				\dis	\iint_{|x-y|<\varepsilon} \frac{(u(x)-u(y))(v(x)-v(y))}{|x-y|^{1+2s}}\,\mathrm{d}y\,\mathrm{d}x&\approx \dis
				\iint_{|x-y|<\varepsilon} u'(x)v'(x) \frac{|x-y|^2}{|x-y|^{1+2s}}\,\mathrm{d}y\,\mathrm{d}x
				\\\\
				&= \dis
				\frac{\varepsilon^{2-2s}}{1-s}
				\int_{|x-y|<\varepsilon}  u'(x)v'(x) \mathrm{d}x
				\\\\
				&\leq  \dis
				\frac{\varepsilon^{2-2s}}{1-s} a(u,v).
			\end{array}
			\]
			\item Third, we have
			\[
			\begin{array}{lll}
				\dis \int_{I}\int_{\mathbb R\setminus I}\frac{u(x)v(x)}{|x-y|^{1+2s}}\,\mathrm{d}y\,\mathrm{d}x&\dis =	\dis \int_{I}u(x)v(x)\int_{\mathbb R\setminus I}\frac{\mathrm{d}y}{|x-y|^{1+2s}}\,\mathrm{d}x
				\\\\
				&\dis \leq \frac{1}{2s}\dis \int_{I}\frac{u(x)v(x)}{\mathrm{dist}(x,\partial I)^{2s}}\,\mathrm{d}x
				\\\\
				&\dis \leq \frac{1}{2s}\dis \left(\int_{I}\frac{u(x)}{\mathrm{dist}(x,\partial I)^{2s}}\,\mathrm{d}x\right)^{\frac12}\left(\int_{I}\frac{v(x)}{\mathrm{dist}(x,\partial I)^{2s}}\,\mathrm{d}x\right)^{\frac12}
				\\\\
				&\dis \leq \frac{C_2}{2s} \|u\|_{H^1(I)}\|v\|_{H^1(I)}.
			\end{array}
			\]
		\end{itemize}
		Finally,
		$$
		\int_{\mathbb R}\int_{\mathbb R}\frac{(u(x)-u(y))(v(x)-v(y))}{|x-y|^{1+2s}}\,\mathrm{d}y\,\mathrm{d}x	\leq \dis
		\frac{\varepsilon^{2-2s}}{1-s} a(u,v)+\left[4C_1\varepsilon^{-1-2s}+ \frac{C_2}{2s}\right] \|u\|_{H^1(I)}\|v\|_{H^1(I)}.
		$$
		Therefore, we get, for $1-s$ sufficiently small,
		$$
		a^s(u,v)-a(u,v)\lesssim\left[	\frac{C(s)\varepsilon^{2-2s}}{2(1-s)} -1\right]a(u,v)+	C(s) \|u\|_{H^1(I)}\|v\|_{H^1(I)}.
		$$
		As $\dis\frac{C(s)\varepsilon^{2-2s}}{2(1-s)} \sim s\sim 1$, we obtain the desired result.
\end{rem}

\section{Numerical simulations}\label{Section5}

In this section, we present some finite element simulations. More precisely, we compare the four formulations introduced earlier: the \emph{local model} (LM), the \emph{old model} (OM), the \emph{new model} (NM), and the \emph{simplified new model} (SNM). To do so, let us recall the choice of the cross coefficient:
\[
\sigma_3 = \frac{\sigma_1+\sigma_2}{2}\quad\text{for the old model}, 
\qquad \sigma_3 = 0\text{ [with } \varphi^s]\quad\text{for the new and simplified new models}.
\]

\subsection{Setup}
Recall that
\[
I=(0,1), \qquad I_1=(0,b), \quad I_2=(b,1), 
\]
with homogeneous Dirichlet condition on $\mathbb{R}\setminus I$. 
The source term is chosen as
\[
f(x) = x^{\alpha}, \qquad \alpha>-\dfrac{1}{2},
\]
so that for the local problem (corresponding to $s\to 1^-$) the exact solution is explicitly known. More precisely,
The exact solution to %\eqref{ClassicalProblem}
\begin{equation} \label{ClassicalProblem-}
    \left\{   
\begin{array}{rclll}
	-\text{div} (\sigma(x) \nabla u) & =f & \text{in} \; I,
	\\
	u(0)=u(1)&=0, & %\text{on} \; \partial I,
\end{array}     
\right.
\end{equation}
 is given by:
\begin{equation}
	u_1(x)=-\frac{x^{\alpha+2}}{\sigma_1(\alpha+1)(\alpha+2)}+\lambda x, \quad 	u_2(x)=\frac{ 1-x^{\alpha+2}}{\sigma_2(\alpha+1)(\alpha+2)}+\lambda\frac{\sigma_1}{\sigma_2}( x-1 ),
\end{equation}
where $\lambda:=\dfrac{\sigma_1(1-b^{\alpha+2})+\sigma_2  b^{\alpha+2}}{\sigma_1(\alpha+1)(\alpha+2)[\sigma_1(1-b)+\sigma_2b]}$ if $\left|\dfrac{\sigma_2}{\sigma_1}\right|\neq \dfrac{1-b}{b}$.

%This allows us to compute directly the error of the nonlocal approximations with respect to the exact local solution. 
\medbreak
In the following numerical experiments, we consider both the classical local problem \eqref{ClassicalProblem-}
and its nonlocal counterparts for $s\in(\frac12,1)$. The local exact solution is used as a reference to study
the convergence of the nonlocal solutions as $s\to 1^{-}$, as well as the accuracy of the proposed numerical
schemes. 
	Indeed, we investigate the influence of the fractional order $s$, the mesh size $h$, and the contrast
between the coefficients $\sigma_1$ and $\sigma_2$ on the convergence behavior.

\subsection{Test A: Comparison of the four models}
We first compare the numerical solutions obtained from the old, new, and simplified new models for fixed parameters
\begin{samepage}
\[
(A1)\qquad s=0.75, \qquad \alpha=0, \qquad h=2^{-9}.
\]
	\vspace*{-1cm}
\begin{figure}[H]
	\centering
	\includegraphics[width=0.66\textwidth]{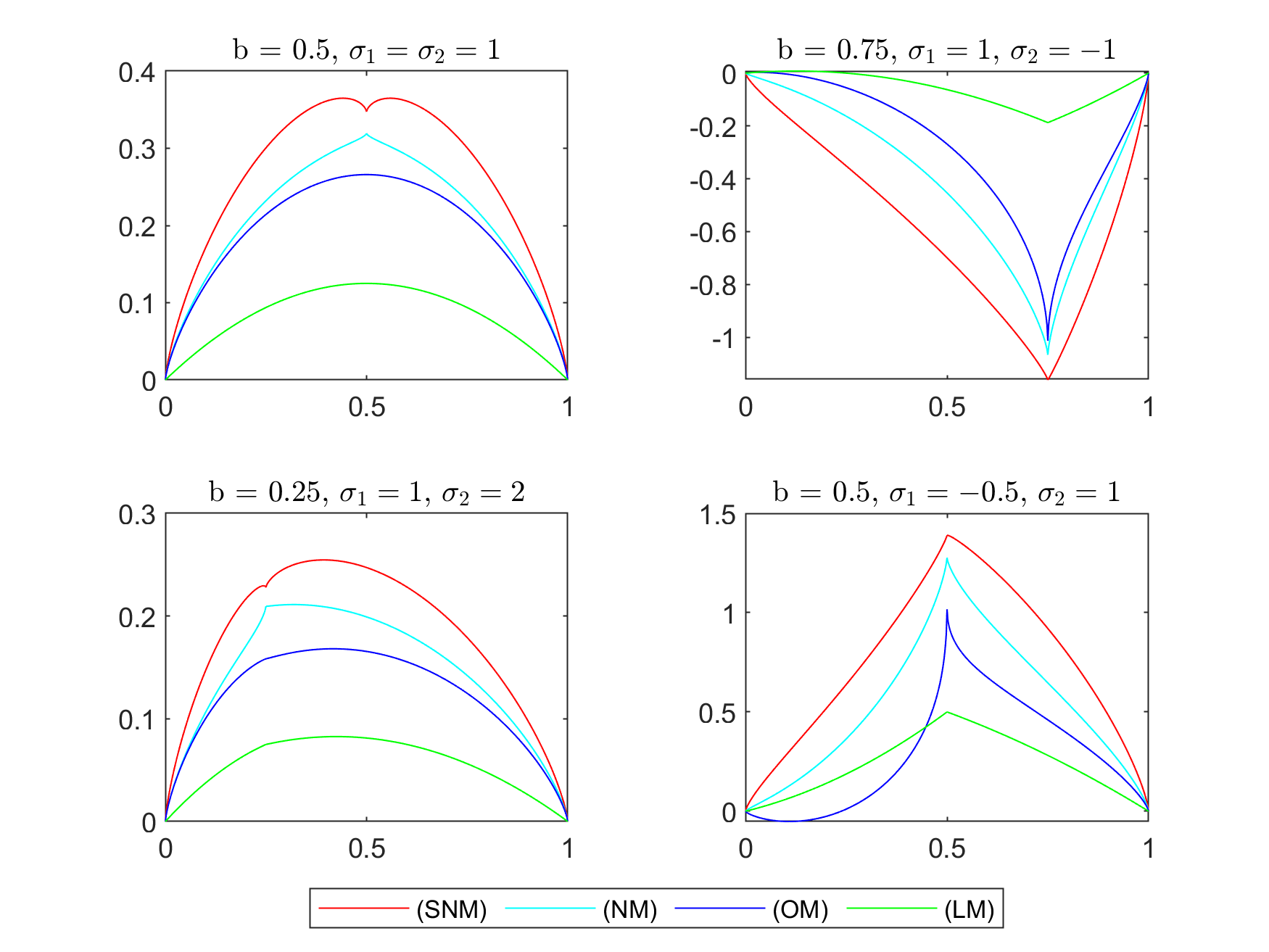}
	\caption{Comparison of solutions for the (exact) local, old, new, and simplified new models.% with  $s=0.75$, $h=2^{-9}$, and $f(x)=1$.
	}
	\label{fig:TestA1}
\end{figure}
\end{samepage}

\begin{samepage}
	\[
	(A2)\qquad s=0.9, \qquad  \alpha=1, \qquad h=2^{-9}.
	\]	
	\vspace*{-1cm}
	\begin{figure}[H]
		\centering
		\includegraphics[width=0.66\textwidth]{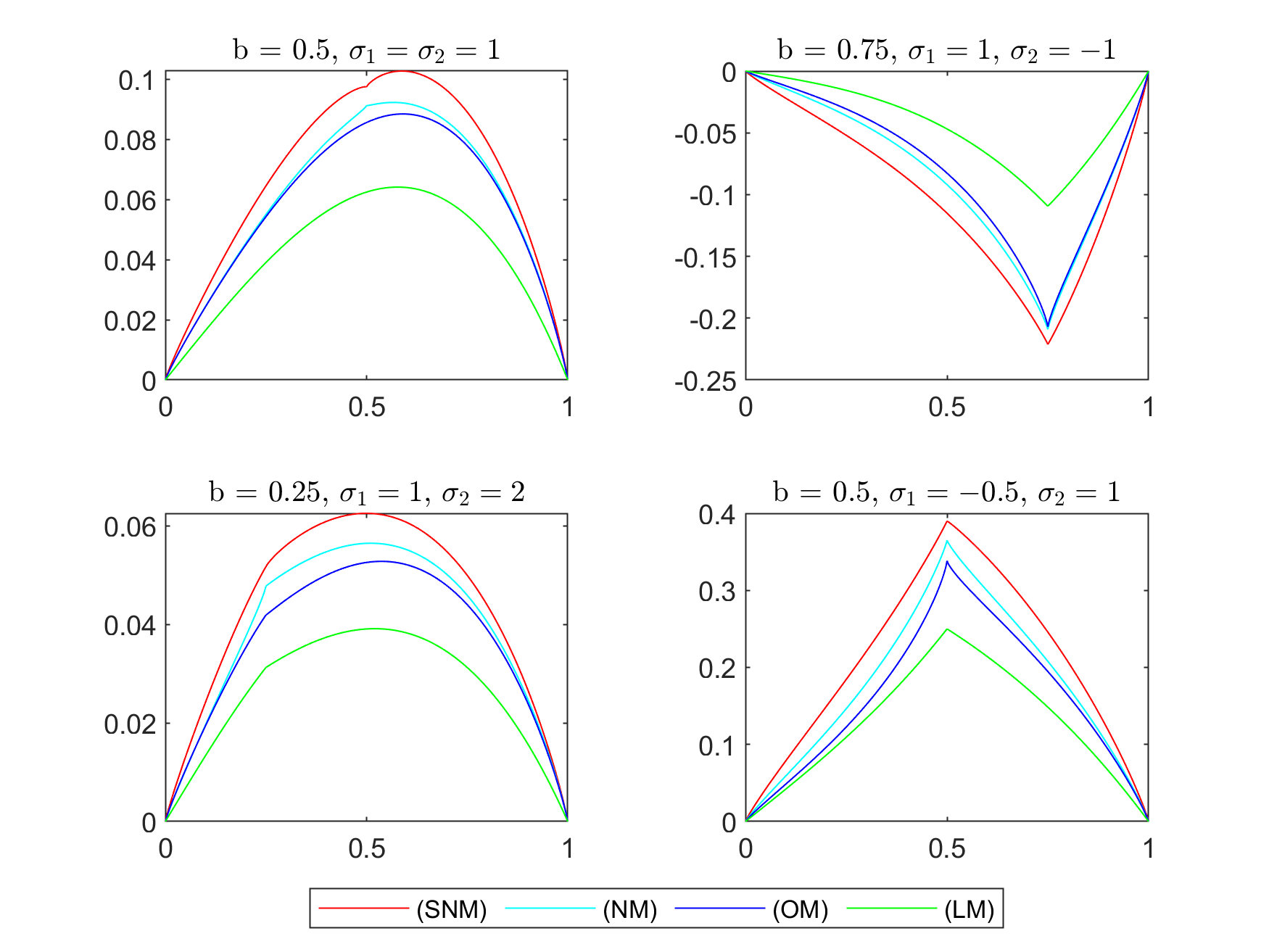}
		\caption{Comparison of solutions for the (exact) local, old, new, and simplified new models.% with  $s=0.9$, $h=2^{-9}$, and $f(x)=x$.
		}
		\label{fig:TestA2}
	\end{figure}
\end{samepage}

Figures~\ref{fig:TestA1}–\ref{fig:TestA2} display the solutions obtained with the four models for representative values of $s$, $\alpha$, and a fixed mesh size $h$.

Having established the consistency of (SNM) at the level of the solution profiles, we now investigate the behavior of the numerical solution of (SNM) as $s \to 1^-$. 
To this end, we denote by $e_h$ the difference between the numerical solution obtained with the simplified new model and the exact solution of the local model (LM).

In this test, the mesh size is fixed to $h = 2^{-9}$, while the parameter $(1-s)$ varies in the range $[5\times 10^{-4},\,4\times 10^{-2}]$. 
Figure~\ref{fig:SlopesA1} displays the evolution of  $\|e_h\|_{H^1}$ as a function of $(1-s)$ for four representative configurations of the interface location and coefficients.

\begin{figure}[H]
	\centering
	\includegraphics[width=0.75\textwidth]{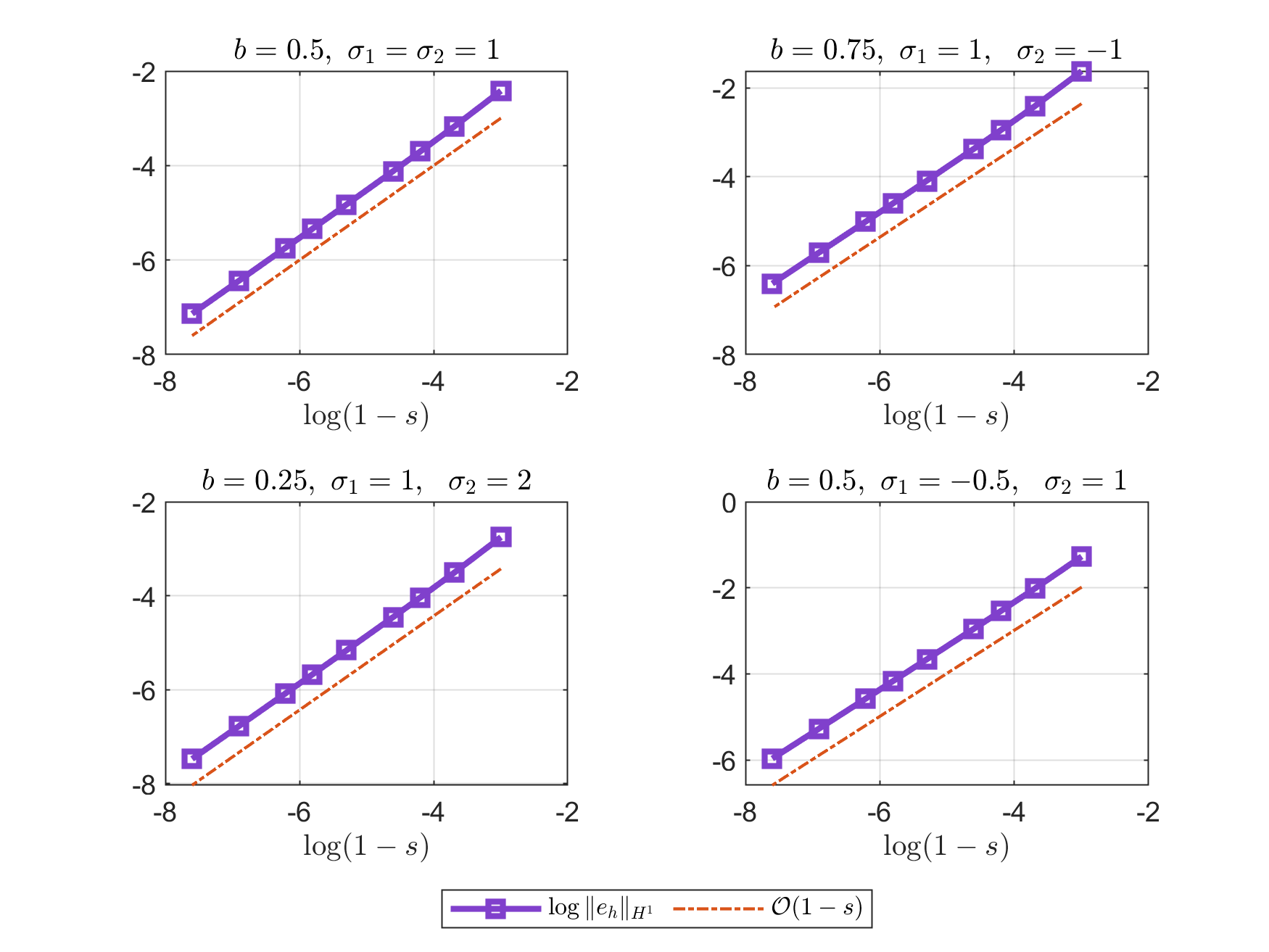}
	\caption{Convergence behavior of the $H^1$-error with respect to $(1-s)$ for the four configurations of (SNM), with $h = 2^{-9}$ and $\alpha=0$.
	}
	\label{fig:SlopesA1}
\end{figure}
In all cases, the numerical error exhibits a clear linear behavior with respect to $(1-s)$ in log-log scale. 
More precisely, the numerical curves are parallel to the reference line $\mathcal{O}(1-s)$ suggesting that
\[
\|e_h\|_{H^1} = \mathcal{O}(1-s) \quad \text{as } s \to 1^-,
\]
uniformly with respect to the tested configurations.
%This observation is fully consistent with the theoretical analysis of the nonlocal-to-local convergence of the simplified new model, which predicts that the dominant consistency error arises from the approximation of the local operator by its fractional counterpart and scales linearly with $(1-s)$.

%We emphasize that this convergence rate is observed for fixed spatial discretization, confirming that the error displayed in Figure~\ref{fig:SlopesA1} is genuinely driven by the nonlocal parameter and not by the mesh size.

\subsection{Test B: Convergence as $h\to0^+$ and $s\to1^-$}
In this test, we investigate the combined asymptotic regime in which both the mesh size $h$ tends to zero and the fractional parameter $s$ approaches the local limit $s \to 1^-$. 
This regime is particularly relevant from both theoretical and numerical viewpoints, since the total error results from the interplay between the discretization error associated with $h$ and the nonlocal-to-local consistency error driven by $(1-s)$.
To this end, we consider a sequence of coupled parameters $(h,s)$, where $h=2^{-k}$ for $k\in\{4,8,\dots,12\}$ and $(1-s)$ varies in the range $[5\times 10^{-4},\,4\times 10^{-2}]$.

Figures~\ref{fig:TestB1} and~\ref{fig:TestB2} illustrate the convergence of the numerical solutions obtained with the nonlocal models towards the exact local solution for two representative configurations.
In both cases, one clearly observes that, as $h$ decreases and $s$ approaches $1$, the solutions of the simplified new model (SNM) progressively align with the local solution.

%This behavior is consistent across all refinement levels and confirms that the combined limit $(h,s)\to(0^+,1^-)$ is correctly captured by the proposed formulation.

\begin{samepage}
	\[
	(B1) \qquad b=0.5,  \qquad (\sigma_1,\sigma_2)=(1,1), \qquad \alpha=0.
	\]
	\begin{figure}[H]
		\centering
		\includegraphics[width=0.66\textwidth]{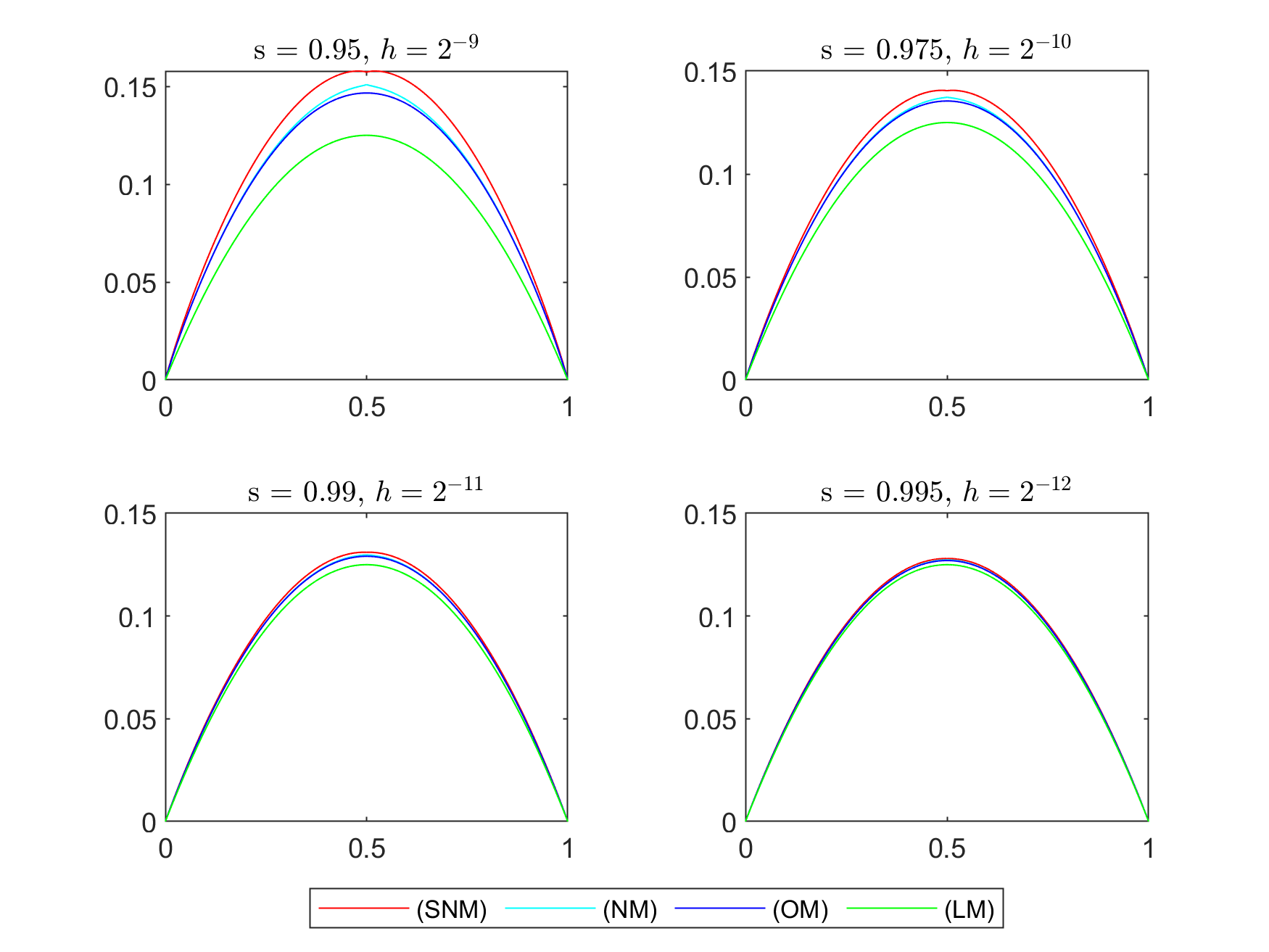}
		\caption{Convergence of nonlocal models towards the local one as $h\to0^+$ and $s\to1^-$.}
	\label{fig:TestB1}
	\end{figure}
\end{samepage}

A quantitative assessment of this convergence is provided in Figure~\ref{fig:SlopesB1}, where the same $H^1$-error is reported as a function of $h$ and $(1-s)$ along the considered sequence of parameter pairs.
The results show that the error exhibits a clear decay with respect to both variables, and that the observed trends are well approximated by reference slopes of order $\mathcal{O}(1-s)$ and $\mathcal{O}(h^\beta)$, with $\beta\approx0.85$.

Overall, these results demonstrate that the simplified new model remains stable and accurate when both parameters are refined simultaneously, and that it converges towards the local model in the relevant joint limit $h\to0^+$ and $s\to1^-$.

\begin{samepage}
	\[
	(B2) \qquad b=0.75,  \qquad (\sigma_1,\sigma_2)=(1,-1), \qquad \alpha=1.
	\]
	\begin{figure}[H]
		\centering
		\includegraphics[width=0.73\textwidth]{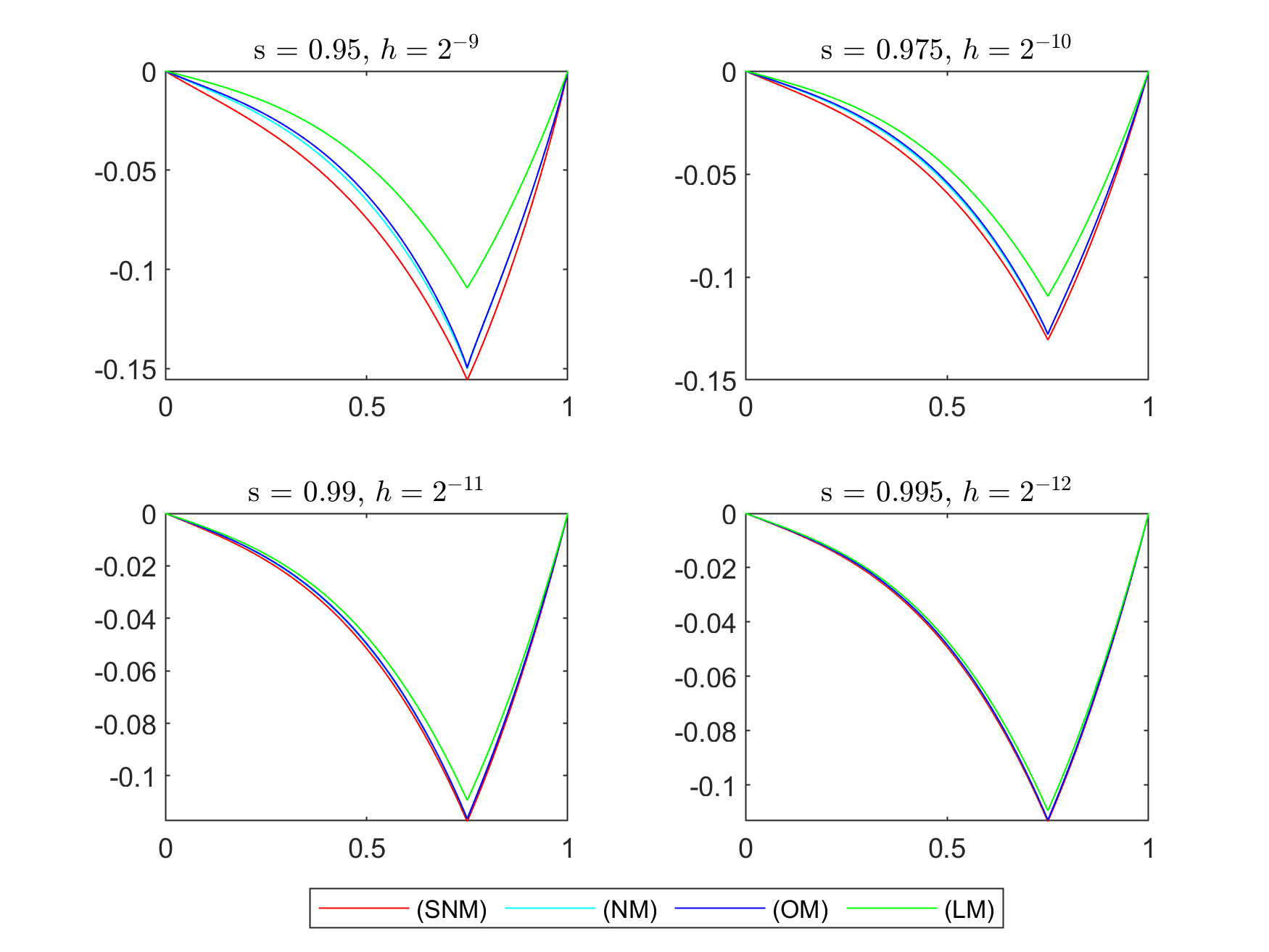}
		\caption{Convergence of nonlocal models towards the local one as $h\to0^+$ and $s\to1^-$.}
		\label{fig:TestB2}
	\end{figure}
\end{samepage}

\begin{figure}[H]
	\centering
	\includegraphics[width=0.47\textwidth]{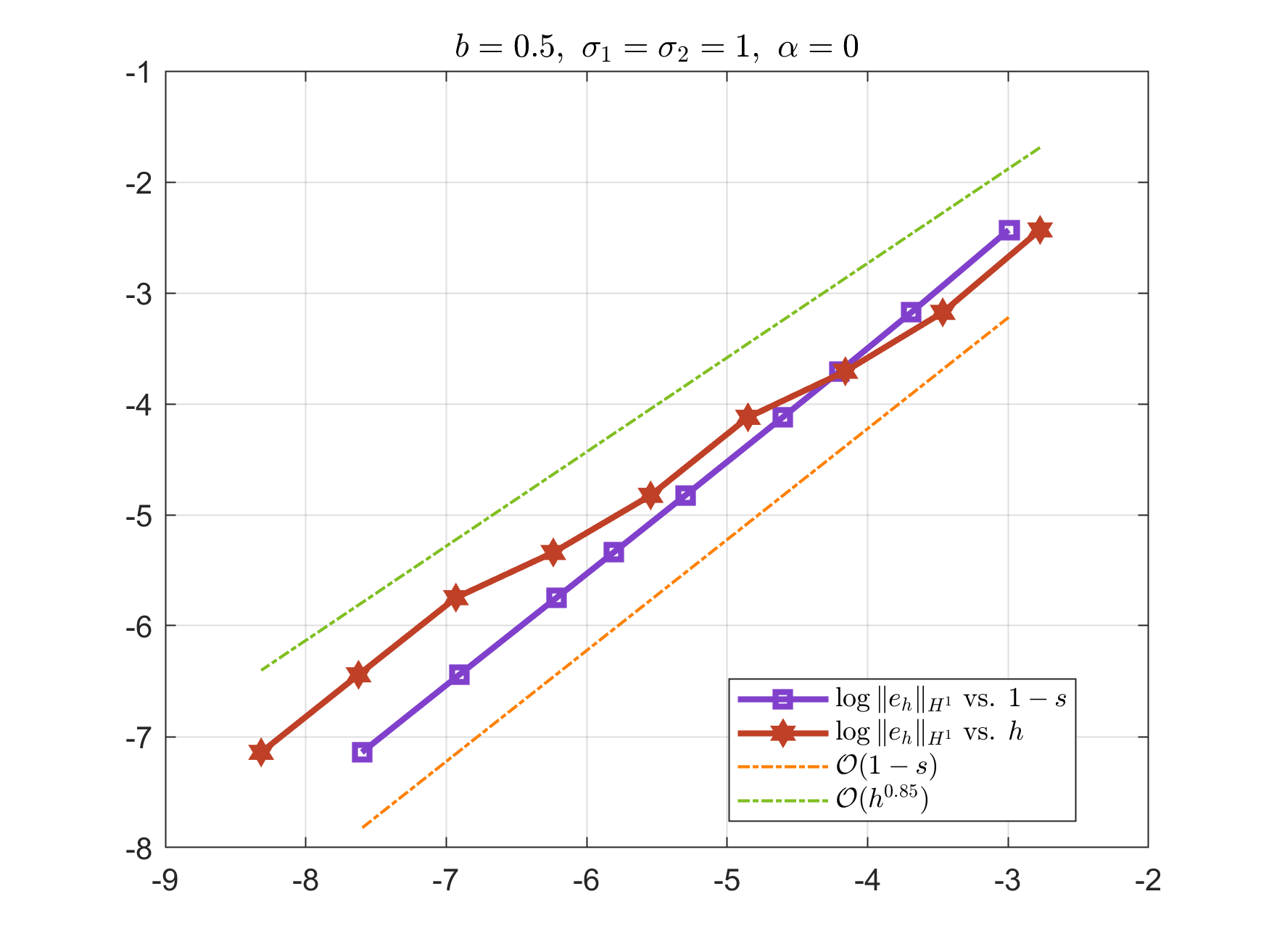}	\includegraphics[width=0.47\textwidth]{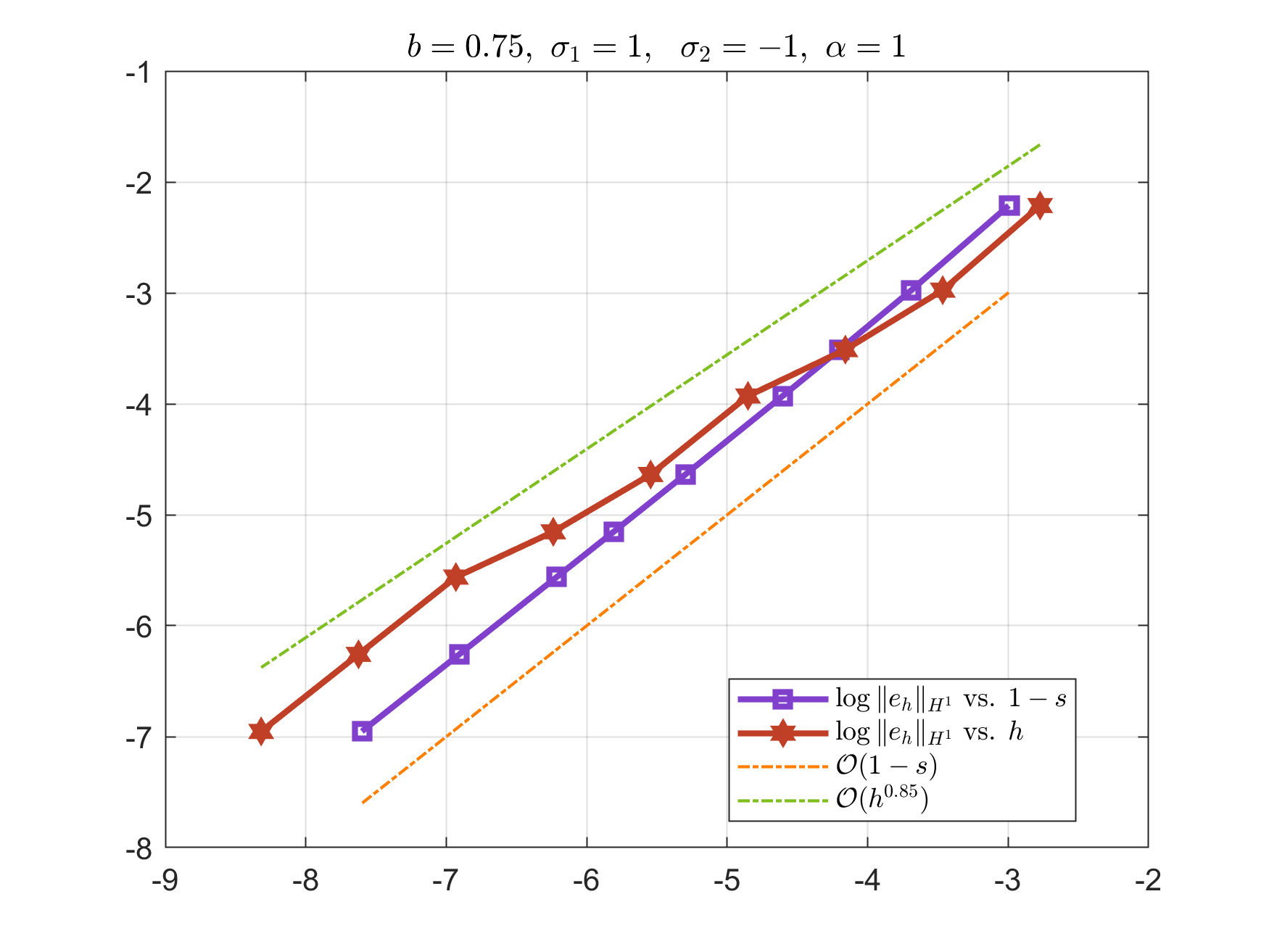}
	\caption{Convergence behavior of the $H^1$-error with respect to  $(1-s)$ and $h$ for two configurations of~(SNM).
	}
	\label{fig:SlopesB1}
\end{figure}

\hfill $\blacksquare$

\medbreak
Before concluding the 1D study, we point out an additional advantage of the
simplified nonlocal model. Besides its well-posedness and its consistency with
the local limit as $s\to 1^-$, the formulation naturally extends to a
multi-subdomain setting. Each subdomain can then be treated independently at
the discrete level, while the global interaction is recovered through only a
few interface unknowns. As a result, the global linear system has a natural
block structure: the diagonal blocks $\mathbb{A}_i$ represent the stiffness
matrices associated with the subdomain operators. A schematic illustration is given in
Figure~\ref{fig:block-structure}.

\begin{figure}[h]
	\centering
	\includegraphics[width=.5\textwidth]{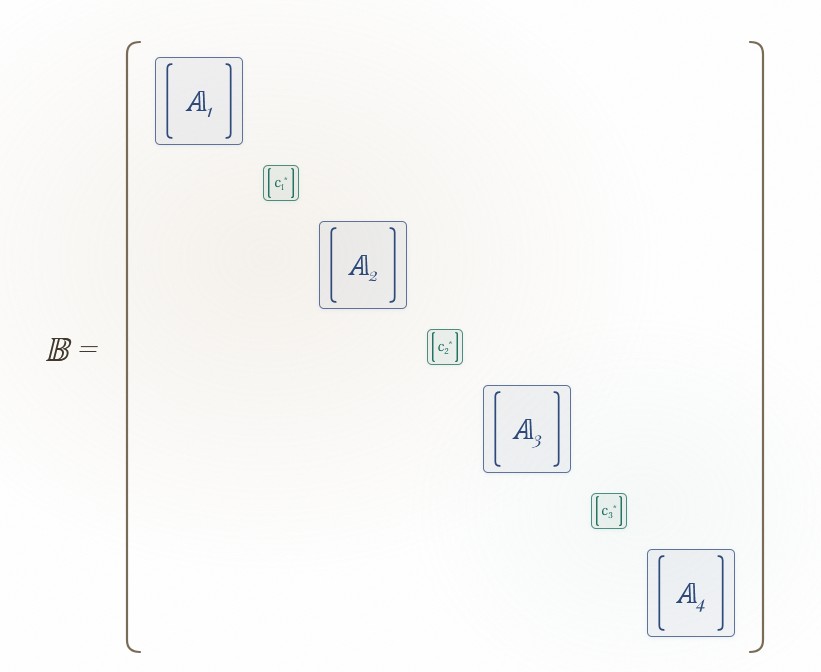}
	\caption{Schematic block structure of the global discrete system in a multi-subdomain setting.}
	\label{fig:block-structure}
\end{figure}

\section{Preliminary 2D numerical illustration}\label{Section6}
This section is exploratory in nature. We do not claim here any two-dimensional
well-posedness or convergence result. Its purpose is only to illustrate that
the interface-reconstruction strategy introduced in the one-dimensional setting
can also be implemented numerically in a simple 2D configuration. The implementation relies in part on the finite element approach introduced in
\cite{AcostaBersetcheBorthagaray}, as well as on a code shared by
J.~P.~Borthagaray and adapted by the author to the present setting.
\paragraph{First numerical tests.} 
Let us consider the domain
$
\Omega = (0,1)\times(0,1),
$
with an interface located at
$
\Gamma := \{b\}\times(0,1),
$
which splits the domain into the two subdomains
\[
\Omega_1 := (0,b)\times(0,1), \qquad
\Omega_2 := (b,1)\times(0,1).
\]

Let us assume that $f(x,y)=1$ and $b=0.5$.

We compare the reconstructed solution $u_h^s$ with a reference solution obtained by solving the fractional problem directly on $\Omega$, using the choice
\[
\sigma_3 := \frac{\sigma_1+\sigma_2}{2}.
\]

Figure~\ref{fig:s1} shows the two solutions for $s=0.999$ in the case $\sigma_1=1$ and $\sigma_2=-0.5$. Figure~\ref{fig:s2} displays the same comparison for the more contrasted configuration $\sigma_1=1$ and $\sigma_2=-2$.

\begin{figure}[H]
	\centering
	\includegraphics[width=0.47\textwidth]{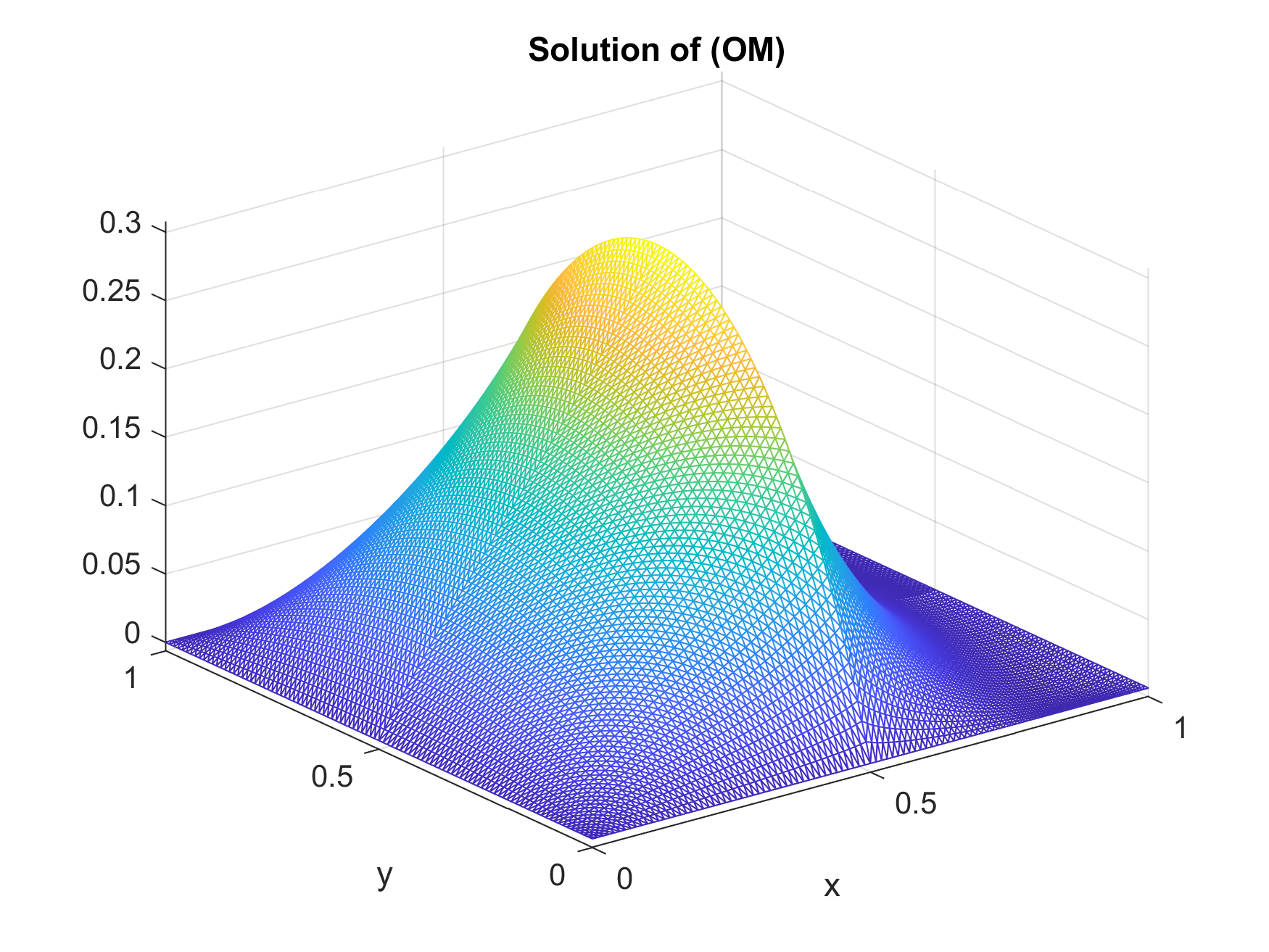}	\includegraphics[width=0.47\textwidth]{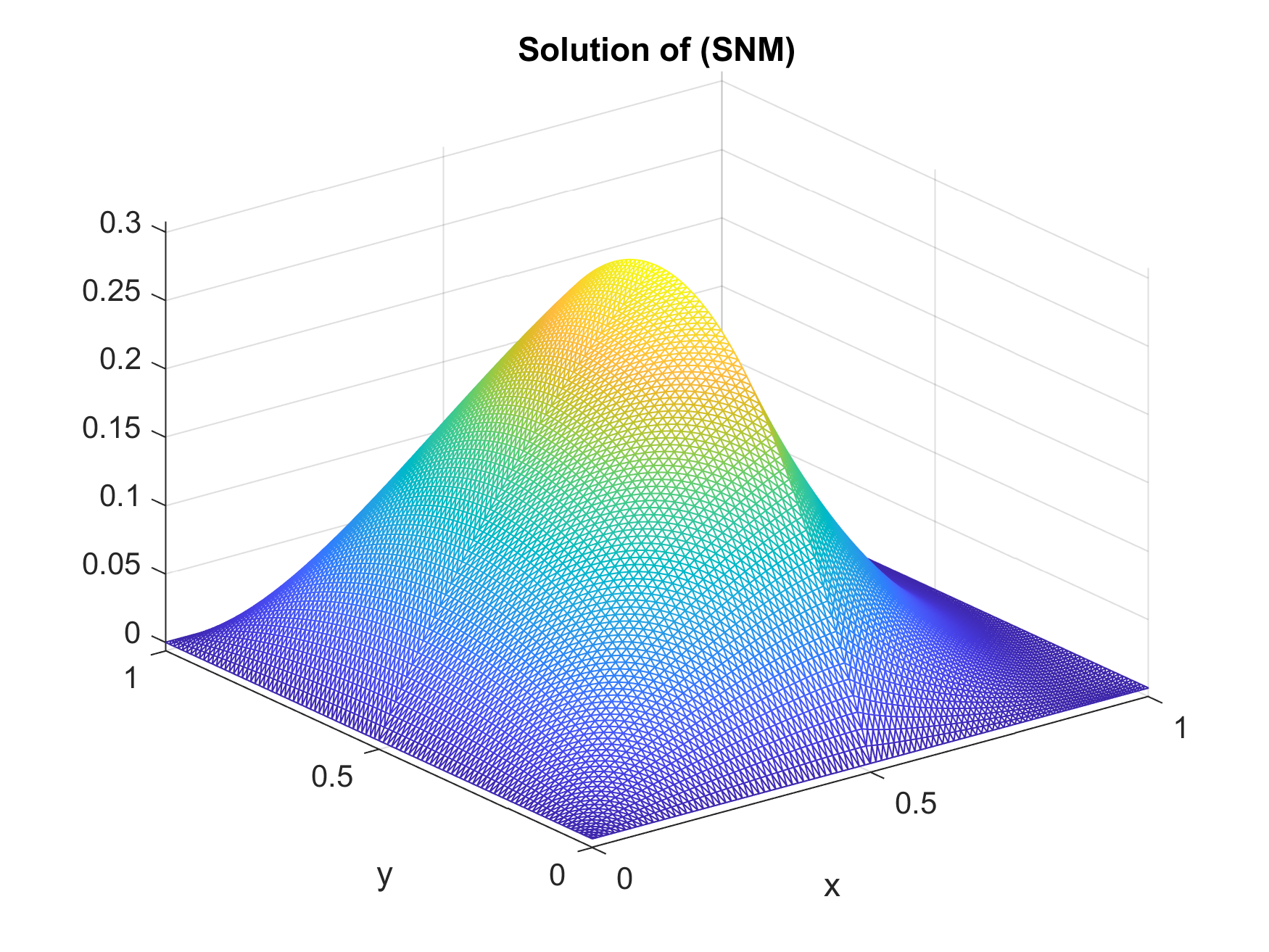}
	\caption{Comparison of the solutions of old and simplified new models for $\sigma_1=1$ and $\sigma_2=-0.5$.
	}
	\label{fig:s1}
\end{figure}

%Despite the strong contrast and the sign change across the interface, the two
%solutions exhibit comparable qualitative features, especially away from the
%interface, indicating that the interface correction successfully compensates for
%the missing nonlocal interactions in the bulk decomposition.

%
%We repeat the same comparison for a more contrasted configuration,
%namely $\sigma_1=1$ and $\sigma_2=-2$, still with $s=0.999$.
%The corresponding solutions are shown in
%Figures~.

\begin{figure}[H]
	\centering
	\includegraphics[width=0.47\textwidth]{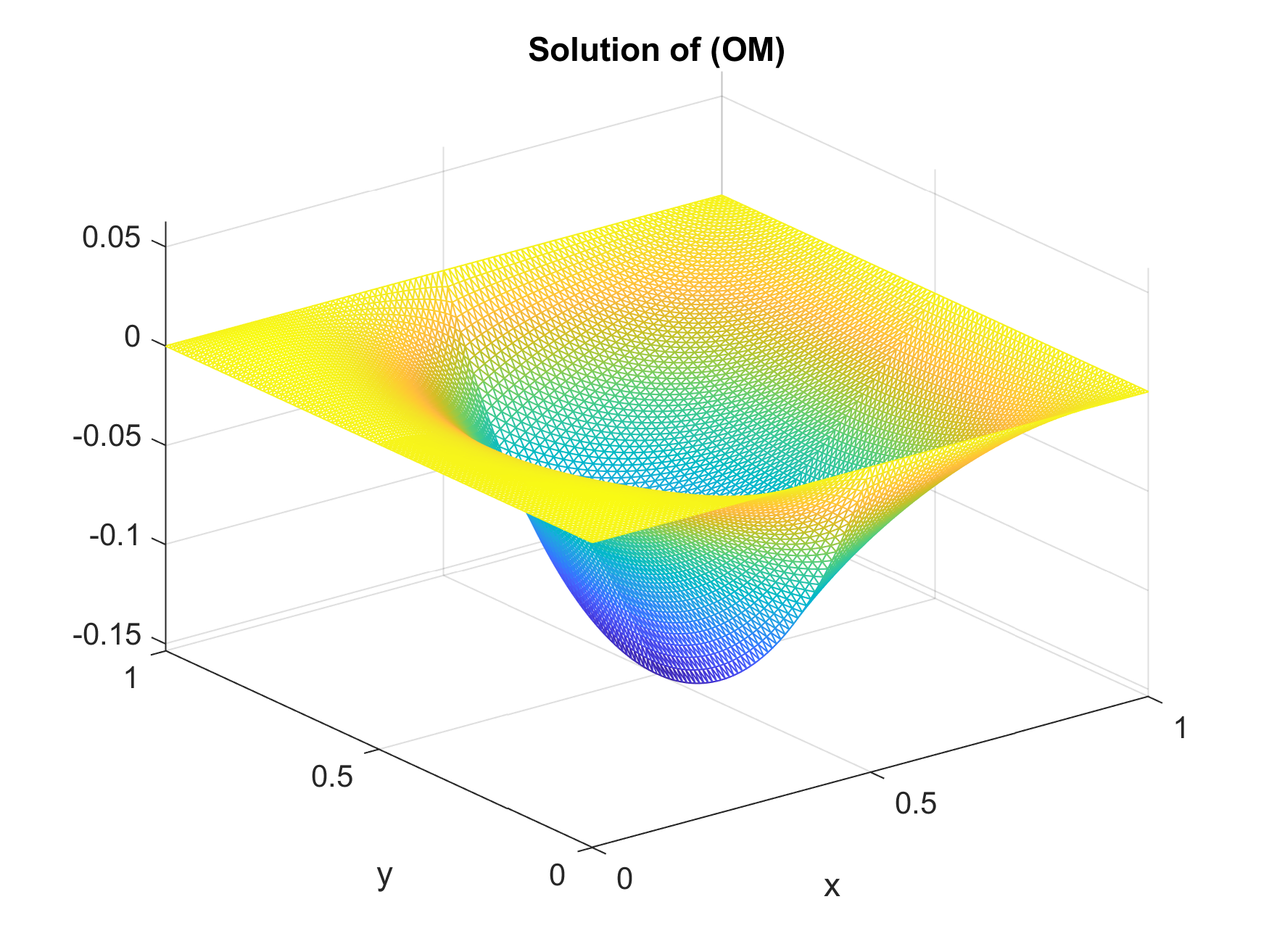}	\includegraphics[width=0.47\textwidth]{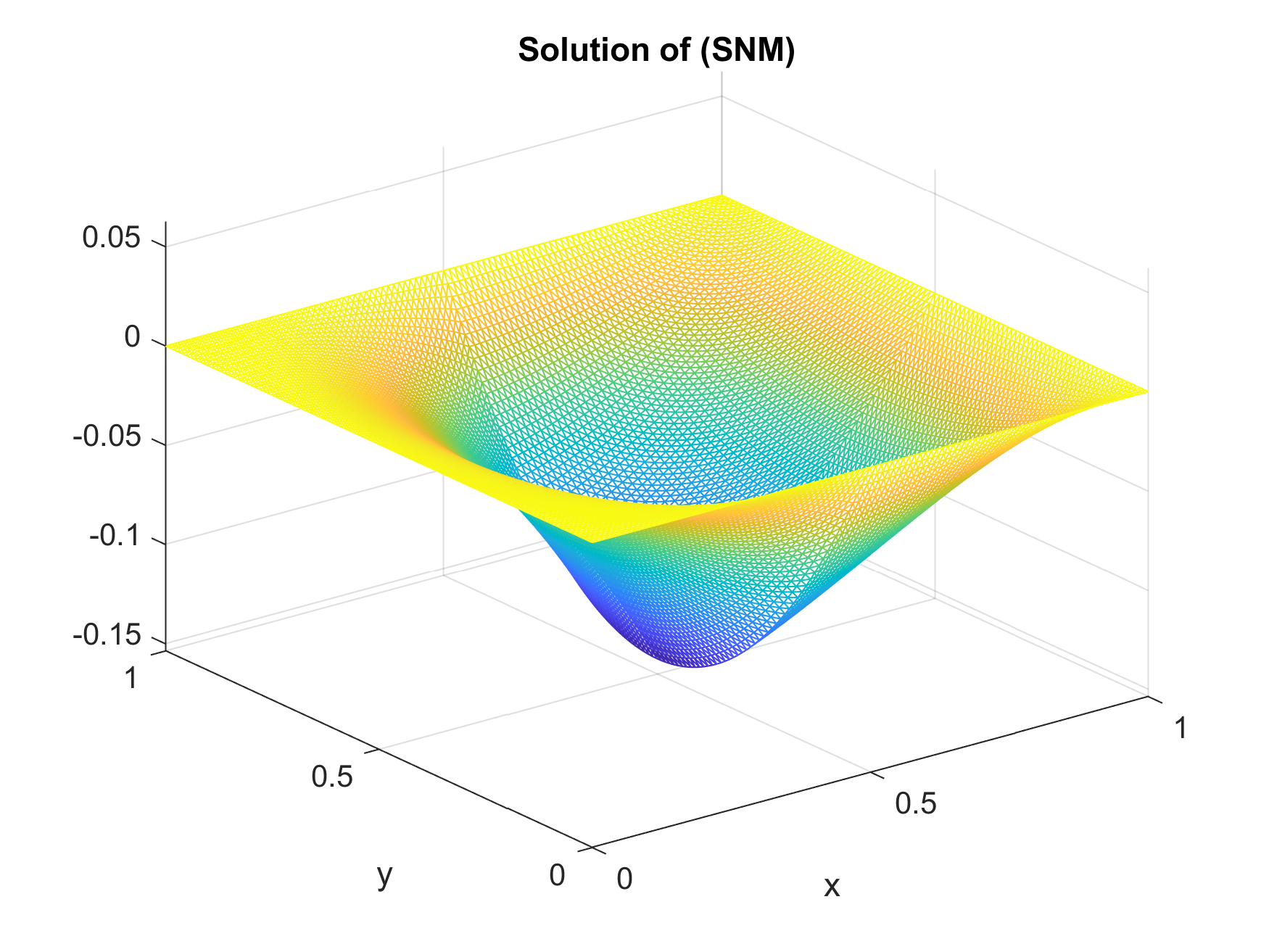}
	\caption{Comparison of the solutions of old and simplified new models for $\sigma_1=1$ and $\sigma_2=-2$.
	}
	\label{fig:s2}
\end{figure}

%A systematic analysis of the influence of the neglected interface terms $D_i$ and
%	of the convergence with respect to $s\to1^{-}$ will be addressed later.
	
	\section{Conclusion}
	We proposed and analyzed a reconstructed nonlocal formulation for a one-dimensional transmission problem with sign-changing coefficients, in the simplified regime $\sigma_3 = 0$. Indeed, we established weak T-coercivity for the global nonlocal problem and convergence of a simplified discrete reconstructed model toward the classical local solution as $s \to 1^-$ and $h \to 0^+$. The numerical experiments support the relevance of the approach and show that the simplified formulation captures the expected local limit. Extending the analysis beyond the case $\sigma_3 = 0$, and justifying the multidimensional setting, remain open questions.

\section*{Funding} This work was supported by the Agence de l’Innovation de Défense (AID) through the Centre Interdisciplinaire d’Etudes pour la Défense et la Sécurité (CIEDS) (project 2022 ElectroMath).

\section*{Acknowledgements}
This work was initiated, and a substantial part of it was carried out, during the author's postdoctoral position at ENSTA Paris, under the supervision of Patrick Ciarlet. The author would like to express sincere gratitude to him for his guidance and support. The author also warmly thanks Juan Pablo Borthagaray for many fruitful
discussions and valuable remarks.
%and for kindly sharing numerical code that was
%used in the preliminary two-dimensional experiments.

\clearpage
 \section{Appendix : Computations of the elements of $\mathbb{A}_{\text{old}}$} \label{Appendix8}
 This appendix is devoted to the explicit computation of the entries of
 $\mathbb{A}_{\mathrm{old}}$ (see Subsection \ref{Subsection3.1}). We begin by recalling the notation used in the
 derivations below.
 
 \smallbreak
 \noindent\textbf{\em --- Some notations :} 
 \begin{itemize}
 	\item[$\bullet$] $\mathbb{B}:=\dfrac{2}{C(s)}\mathbb{A}_{\text{old}}$ ;
 	\item [$\bullet$] $	 {H(h,s)}:=\dfrac{h^{1-2s}}{2s(1-s)(1-2s)(3-2s)}$ $(s\neq\dfrac12)$~;
 	%		\item[$\bullet$] $ {H_1(h,s)}:= \dfrac{h^{1-2s}}{s}\dis\int_{-1}^0(1+z)^{2-2s}\mathrm{d}z=\dfrac{h^{1-2s}}{s(3-2s)}$ ;
 	\item[$\bullet$] $ {H_1(h,s)}:= \dfrac{h^{1-2s}}{s(3-2s)}$ ;
 	%	\item[$\bullet$] $ {H_2(h,s)}:= \dfrac{h^{1-2s}}{s}\dis\int_{0}^1\dfrac{(1-z)^2}{(1+z)^{2s}}\mathrm{d}z=h^{1-2s}\dfrac{-2s^2+7s-7+2^{3-2s}}{s(1-s)(1-2s)(3-2s)}$ $(s\neq\dfrac12)$ ;
 	\item[$\bullet$] $ {H_2(h,s)}:=\left\{\begin{array}{lll}
 		2 {H(h,s)}[-2s^2+7s-7+2^{3-2s}] &\;\; \text{if} &s\neq\frac12,
 		\\
 		-5+8\log(2)&\;\; \text{if} &s=\frac12\;;
 	\end{array} \right.$
 	
 	%		\item[$\bullet$] $ {H_3(h,s)}:=\dfrac{1}{2h^2(1-s)}\dis\int_{0}^h[z^{2-2s}-(z-h)^{2-2s}]\mathrm{d}z=\dfrac{h^{1-2s}}{(1-s)(3-2s)}$ ;
 	\item[$\bullet$] $ {H_3(h,s)}:=\dfrac{h^{1-2s}}{(1-s)(3-2s)}$ ;
 	%		\item[$\bullet$] $ {H_4(h,s)}:=h^{1-2s}\dis\int_0^1 \int_{-1}^0 \frac{(x+y)^2}{(x-y)^{1+2s}}\mathrm{d}y\mathrm{d}x=h^{1-2s}\dfrac{2s^2-5s+4-2^{2-2s}}{s(1-s)(1-2s)(3-2s)}$ $(s\neq\dfrac12)$ ;
 	%		
 	\item[$\bullet$] $ {H_4(h,s)}:=\left\{\begin{array}{lll}
 		2 {H(h,s)}[2s^2-5s+4-2^{2-2s}] &\;\; \text{if} &s\neq\frac12,
 		\\
 		3-4\log(2)&\;\; \text{if} &s=\frac12\;;
 	\end{array} \right.$	
 	
 	%		\item[$\bullet$] $ {H_5(h,s)}:= \dfrac{h^{1-2s}}{s}\dis\int_{-1}^0\dfrac{|z|}{(1+z)^{-1+2s}}\mathrm{d}z=\dfrac{h^{1-2s}}{2s(1-s)(3-2s)}$ ;
 	
 	\item[$\bullet$] $ {H_5(h,s)}:= {H(h,s)}[1-2s] $ ;
 	%		\item[$\bullet$] $ {H_6(h,s)}:= -2 h^{1-2s} \dis\int_{-1}^0 \int_{-1}^0 \frac{(1+x)(1+y)}{(1+x-y)^{1+2s}}\mathrm{d}y\mathrm{d}x=h^{1-2s}\dfrac{s-2+2^{1-2s}}{s(1-s)(3-2s)}$ ;
 	\item[$\bullet$] $ {H_6(h,s)}:= 2 {H(h,s)}[1-2s][s-2+2^{1-2s}]
 	$ ;
 	%		\item[$\bullet$] $
 	%		 {H_7(h,s)}:= -2 h^{1-2s} \dis\int_{0}^1 \int_{-1}^0 \frac{(1-x)(1+y)}{(1+x-y)^{1+2s}}\mathrm{d}y\mathrm{d}x=h^{1-2s}\dfrac{4s^2-s(14-2^{4-2s})+13+3^{3-2s}-5\cdot2^{3-2s}}{2s(1-s)(1-2s)(3-2s)}$ $(s\neq\dfrac12)$~;
 	%		
 	\item[$\bullet$] $ {H_7(h,s)}:=\left\{\begin{array}{lll}
 		{H(h,s)}[4s^2-s(14-2^{4-2s})+13+3^{3-2s}-5\cdot2^{3-2s}] &\;\; \text{if} &s\neq\frac12,
 		\\
 		1-16\log(2)+9\log(3)&\;\; \text{if} &s=\frac12\;;
 	\end{array} \right.$
 	
 	%		\item[$\bullet$] for any $r>h$, $
 	%		\begin{array}{llll}
 		%		 {S_1(h,s,r)}	:=\dfrac{h}{s}\dis\int_{-1}^1\frac{(1-|z|)^2}{(hz+r)^{2s}}\mathrm{d}z
 		%		=\dfrac{(h+r)^{3-2s}+2hr^{2-2s}[2s-3]-(r-h)^{3-2s}}{h^2s(1-s)(1-2s)(3-2s)} \quad (s\neq\dfrac12)~;
 		%		\end{array}
 	%		$
 	
 	\item[$\bullet$]  $\forall r>h$, $ {S_1(h,s,r)}:=\left\{\begin{array}{lll}
 		\dfrac{2 {H(h,s)}}{h^{3-2s}}[(h+r)^{3-2s}+2hr^{2-2s}[2s-3]-(r-h)^{3-2s}] &\;\; \text{if} &s\neq\frac12,
 		\\
 		\frac{1}{h^{2}}[-2 \log  \left(-h +r \right) \left(h -r \right)^{2}+2 \left(h +r \right)^{2} \log \left(h +r \right)-8 r \left(\log \left(r \right)+\frac{1}{2}\right) h]&\;\; \text{if} &s=\frac12\;;
 	\end{array} \right.$
 	
 	%			\item[$\bullet$] for any $r>h$, $
 	%		\begin{array}{llll}
 		%			 {S_2(h,s,r)}	:=-\dfrac{h}{s}\dis\int_{-1}^0\frac{z(1+z)}{(hz+r)^{2s}}\mathrm{d}z
 		%			=\dfrac{r^{2-2s}[2hs-3h+2r]+(r-h)^{2-2s}[2hs-h-2r]}{2h^2s(1-s)(1-2s)(3-2s)} \quad (s\neq\dfrac12)~;
 		%		\end{array}
 	%		$
 	
 	\item[$\bullet$]  $\forall r>h$, $ {S_2(h,s,r)}:=\left\{\begin{array}{lll}
 		\dfrac{ {H(h,s)}}{h^{3-2s}}[r^{2-2s}[2hs-3h+2r]+(r-h)^{2-2s}[2hs-h-2r]] &\;\; \text{if} &s\neq\frac12,
 		\\
 		\frac{1}{h^{2}}[2 r \left(h -r \right) \log  \left(-h +r \right)+2 r\log  \left(r \right)  \left(-h +r \right)+h \left(h -2 r \right)]&\;\; \text{if} &s=\frac12\;;
 	\end{array} \right.$
 	
 	%		\item[$\bullet$] $H_6(h,s):=H_4(h,s)+H_5(h,s)=h^{1-2s}\dfrac{-4s+4-2^{2-2s}}{s(1-s)(1-2s)(3-2s)}$ $(s\neq\dfrac12)$ ;
 	%	\item[$\bullet$]  $\forall k\geq2$, $\begin{aligned}[t]  {L_1(h,s,k)}:=  {H(h,s)}[3k^{3-2s}-2(k+1)^{3-2s}+(k+2)^{3-2s}-(k-2s+2)(k-1)^{2-2s}\\  +(-4s+6)k^{2-2s}-(k-2s+4)(k+1)^{2-2s}]\; \text{if} \;s\neq\frac12 ,\end{aligned}$
 	%	\\
 	%	if $s=\frac12 $, $ {L_1(h,s,k)}:=\left(-k^{2}+1\right) \log  \left(-1+k \right)+\left(2+k \right)^{2} \log \left(2+k \right)+\left(-3 k^{2}-8 k -5\right) \log \left(k +1\right)+3 \log  \left(k \right) \left(k +\frac{4}{3}\right) k$
 	
 	\item[$\bullet$] $\forall k \geq 2$,
 	\[
 	{L_1(h,s,k)} :=
 	\begin{cases}
 		{H(h,s)}\Big[
 		3k^{3-2s} - 2(k+1)^{3-2s} + (k+2)^{3-2s}
 		- (k - 2s + 2)(k-1)^{2-2s} \\
 		\hspace{4.5cm}  + (-4s + 6)k^{2 - 2s}
 		- (k - 2s + 4)(k+1)^{2 - 2s}
 		\Big] & \text{if } s \neq \frac{1}{2}, \\[1.5ex]
 		(-k^2 + 1)\log(k - 1)
 		+ (k + 2)^2 \log(k + 2) +[-3k^2 - 8k - 5]\log(k + 1)
 		\\
 		\hspace{9cm}+ 3k[k + \tfrac{4}{3}] \log(k)
 		& \text{if } s = \frac{1}{2}\;;
 	\end{cases}
 	\]
 	
 	%		\item[$\bullet$] for any $ \forall k\geq2$, $\begin{aligned}[t]  {L_2(h,s,k)}:=  {H(h,s)}[(k-2)^{3-2s}-2(k-1)^{3-2s}+3k^{3-2s}-(k+2s-4)(k-1)^{2-2s}\\ \hspace*{7.5cm}+(4s-6)k^{2-2s}-(k+2s-2)(k+1)^{2-2s}];\;	\end{aligned}$

 	\item[$\bullet$] $\forall k \geq 2$,
 	\[
 	{L_2(h,s,k)} :=
 	\begin{cases}
 		{H(h,s)}\Big[
 		(k-2)^{3-2s}-2(k-1)^{3-2s}+3k^{3-2s}-(k+2s-4)(k-1)^{2-2s} \\
 		\hspace{4.5cm} 	+(4s-6)k^{2-2s}-(k+2s-2)(k+1)^{2-2s}
 		\Big] & \text{if } s \neq \frac{1}{2}, \\[1.5ex]
 		\left(k -2\right)^{2} \log \left(k -2\right)+[-3 k^{2}+8 k -5] \log \left(k-1 \right)+\left(-k^{2}+1\right) \log \left(k +1\right)
 		\\
 		\hspace{9cm}+3  k [k -\frac{4}{3}]\log \left(k \right)
 		& \text{if } s = \frac{1}{2} \text{ and } k\neq2,
 		\\[1.5ex]
 		4 \log (2)-3 \log (3) & \text{if } s = \frac{1}{2} \text{ and } k=2.
 	\end{cases}
 	\]

 \end{itemize}

 In the subsequent, we present some details on the calculations of the elements of $\mathbb{B}$.
 
 \bigbreak
 %\noindent\textbf{\em --- First case :} diagonal elements $b_{ii}$. 
 \subsection{First case : diagonal elements $\mathbb{B}_{ii}$}
 We have
 $$
 \mathbb{B}_{ii}= \int_{-\infty}^{+\infty}	 \int_{-\infty}^{+\infty} \underline{\sigma}(x,y) \frac{(\phi_{i}(x)-\phi_{i}(y))^2}{|x-y|^{1+2s}}\mathrm{d}y\mathrm{d}x=\int_{-\infty}^{+\infty}	 \int_{-\infty}^{+\infty} g_{i}(x,y)\mathrm{d}y\mathrm{d}x, \quad i=1,...,N_h.
 $$
 To simplify the computations, let us subdivide the domain of integration:
 $$
 \begin{array}{llll}
 	\dis
 	\int_{-\infty}^{+\infty}\int_{-\infty}^{+\infty}g_{i}(x,y)\mathrm{d}y\mathrm{d}x&\dis=\overset{\hspace*{3.75cm}=0}{\cancel{\int_{-\infty}^{x_{i-1}}\int_{-\infty}^{x_{i-1}}g_{i}(x,y)\mathrm{d}y\mathrm{d}x}}+\overbrace{2\int_{-\infty}^{x_{i-1}}\int_{x_{i-1}}^{x_{i+1}}g_{i}(x,y)\mathrm{d}y\mathrm{d}x}^{J_1(h,s)}
 	\\\\
 	&\dis +2\overset{\hspace*{3.75cm}=0}{\cancel{\int_{-\infty}^{x_{i-1}}\int_{x_{i+1}}^{+\infty}g_{i}(x,y)\mathrm{d}y\mathrm{d}x}} +\overbrace{2\int_{x_{i+1}}^{+\infty}\int_{x_{i-1}}^{x_{i+1}}g_{i}(x,y)\mathrm{d}y\mathrm{d}x}^{J_2(h,s)}
 	\\\\
 	&\dis+\overset{\hspace*{3.75cm}=0}{\cancel{\int_{x_{i+1}}^{+\infty}\int_{x_{i+1}}^{+\infty}g_{i}(x,y)\mathrm{d}y\mathrm{d}x}}+\overbrace{\int_{x_{i-1}}^{x_{i+1}}\int_{x_{i-1}}^{x_{i+1}}g_{i}(x,y)\mathrm{d}y\mathrm{d}x}^{J_3(h,s)}.
 \end{array}
 $$
 In order, to compute $J_1$, $J_2$ and $J_3$, wee need to distinguish five cases ; $x_i=b$, $x_{i+1}=b$, $x_{i-1}=b$,  $x_{i-1}>b$ and   $x_{i+1}<b$.
 \\
 %{\bf\em 1) when $x_i=b$.} 
 \subsubsection{When $x_i=b$}
 In this case, we split $\dis\int_{x_{i-1}}^{x_{i+1}}\ldots=\int_{x_{i-1}}^b\ldots+\int_b^{x_{i+1}}\ldots$. So, we have the following.
 \begin{itemize}
 	\item[$\bullet$] \underline{Calculation of $J_1(h,s)$} :
 	$$
 	\begin{array}{llllll}
 		J_1(h,s)&=2\dis\int_{-\infty}^{x_{i-1}}\int_{x_{i-1}}^{x_{i+1}}\underline{\sigma}(x,y) \frac{(\overset{\hspace*{.5cm}=0}{\cancel{\phi_{i}(x)}}-\phi_{i}(y))^2}{|x-y|^{1+2s}}\mathrm{d}y\mathrm{d}x
 		\\\\
 		&=2\sigma_1 \dis\int_{-\infty}^{x_{i-1}}\int_{x_{i-1}}^b\frac{(\phi_{i}(y))^2}{(y-x)^{1+2s}}\mathrm{d}y\mathrm{d}x+2\sigma_3 \dis\int_{-\infty}^{x_{i-1}}\int_b^{x_{i+1}}\frac{(\phi_{i}(y))^2}{(y-x)^{1+2s}}\mathrm{d}y\mathrm{d}x
 		\\\\
 		&=\dis \frac{\sigma_1}{s}\int_{x_{i-1}}^b\frac{(\phi_{i}(y))^2}{(y-x_{i-1})^{2s}}\mathrm{d}y+\frac{\sigma_3}{s}\int_b^{x_{i+1}}\frac{(\phi_{i}(y))^2}{(y-x_{i-1})^{2s}}\mathrm{d}y.
 	\end{array}
 	$$
 	By performing the change of variables $\widetilde{y}=\dfrac{|x_i-y|}{h}$, we get $ {J_1(h,s)=\sigma_1H_1(h,s)+\sigma_3 H_2(h,s)}$.
 	\item[$\bullet$] \underline{Calculation of $J_2(h,s)$} : similarly, we obtain $ {J_2(h,s)=\sigma_2H_1(h,s)+\sigma_3 H_2(h,s)}$.
 	\item[$\bullet$] \underline{Calculation of $J_3(h,s)$} : here we need to subdivide both integrals, {\em that is},
 	$$
 	\int_{x_{i-1}}^{x_{i+1}}\int_{x_{i-1}}^{x_{i+1}} \underline{\sigma}\ldots=\sigma_1	\int_{x_{i-1}}^b\int_{x_{i-1}}^b\ldots+	\sigma_2\int_b^{x_{i+1}}\int_b^{x_{i+1}}\ldots +2\sigma_3	\int_b^{x_{i+1}}\int_{x_{i-1}}^b\ldots.
 	$$
 	--- We have 
 	$$
 	\begin{array}{llll}
 		\dis\int_{x_{i-1}}^b\int_{x_{i-1}}^b\frac{(\phi_{i}(x)-\phi_{i}(y))^2}{|x-y|^{1+2s}}\mathrm{d}y\mathrm{d}x&=\dis h^{-2}   \int_{x_{i-1}}^b\int_{x_{i-1}}^b\frac{(|y-b|-|x-b|)^2}{|x-y|^{1+2s}}\mathrm{d}y\mathrm{d}x
 		\\\\
 		&=\dis  h^{-2}   \int_{x_{i-1}}^b\int_{x_{i-1}}^b|x-y|^{1-2s}\mathrm{d}y\mathrm{d}x
 		\\\\
 		&=\dis\dfrac{h^{-2}}{2(1-s)} \int_{x_{i-1}}^b [(b-x)^{2-2s}-(x_{i-1}-x)^{2-2s}]\mathrm{d}x.
 	\end{array}
 	$$
 	With the change of variables $\widehat{x}=b-x$, we get 
 	$$
 	\dis\int_{x_{i-1}}^b\int_{x_{i-1}}^b\frac{(\phi_{i}(x)-\phi_{i}(y))^2}{|x-y|^{1+2s}}\mathrm{d}y\mathrm{d}x=H_3(h,s).
 	$$
 	--- In the same way, we obtain
 	$$
 	\dis\int_b^{x_{i+1}}\int_b^{x_{i+1}}\frac{(\phi_{i}(x)-\phi_{i}(y))^2}{|x-y|^{1+2s}}\mathrm{d}y\mathrm{d}x=H_3(h,s).
 	$$
 	--- Furthermore, we have
 	$$
 	\begin{array}{llll}
 		\dis\int_b^{x_{i+1}}\int_{x_{i-1}}^b\frac{(\phi_{i}(x)-\phi_{i}(y))^2}{|x-y|^{1+2s}}\mathrm{d}y\mathrm{d}x&=	\dis\int_b^{x_{i+1}}\int_{x_{i-1}}^b\frac{\left|\frac{|y-b|}{h}-\frac{|x-b|}{h}\right|^2}{|x-y|^{1+2s}}\mathrm{d}y\mathrm{d}x
 		\\\\
 		&= H_4(h,s),
 	\end{array}
 	$$
 	where we have performed the change of variables  
 	$\widetilde{x}=\frac{x-b}{h}$ and $\widetilde{y}=\frac{y-b}{h}$.
 	\\
 	Therefore, $ {J_3(h,s)=(\sigma_1+\sigma_2)H_3(h,s)+2\sigma_3H_4(h,s)}$.
 \end{itemize}
 Hence, in the case $ {x_i=b}$, we get
 %\begin{center}
 %	\boxed{$$ {b_{ii}=(\sigma_1+\sigma_2)[H_1(h,s)+H_3(h,s)]+2\sigma_3[H_2(h,s)+H_4(h,s)]}.$$}
 %\end{center}
% \begin{tcolorbox}[colback=blue!10, colframe=blue!70!black, boxrule=0.5pt, arc=50pt]
 	%, title=Final Result
 	\[
 	{\mathbb{B}_{ii} = (\sigma_1 + \sigma_2)\big[H_1(h, s) + H_3(h, s)\big] + 2\sigma_3\big[H_2(h, s) + H_4(h, s)\big]}.
 	\]
% \end{tcolorbox}
 \subsubsection{When $x_{i+1}=b$}
 %In this case
 \begin{itemize}
 	\item[$\bullet$] \underline{Calculation of $J_1(h,s)$} :
 	$$
 	\begin{array}{llllll}
 		J_1(h,s)&=2\sigma_1\dis\int_{-\infty}^{x_{i-1}}\int_{x_{i-1}}^{b} \frac{(\phi_{i}(y))^2}{|x-y|^{1+2s}}\mathrm{d}y\mathrm{d}x
 		\\\\
 		&=\dfrac{\sigma_1}{s}\dis \int_{x_{i-1}}^{b}\frac{(\phi_{i}(y))^2}{(y-x_{i-1})^{2s}}\mathrm{d}y
 		\\\\
 		&=\sigma_1\dfrac{h^{1-2s}}{s}\dis\int_{-1}^1\frac{(1-|y|)^2}{(1+y)^{2s}}\mathrm{d}y,
 	\end{array}
 	$$
 	where we have performed the change of variables $\widetilde{y}=\dfrac{y-x_i}{h}$. Thus,
 	$ {J_1(h,s)=\sigma_1\big[H_1(h, s) + H_2(h, s)\big] }$.
 	
 	\item[$\bullet$] \underline{Calculation of $J_2(h,s)$} :  similarly, we obtain $ {J_2(h,s)=\sigma_3\big[H_1(h, s) + H_2(h, s)\big] }$.

 	\item[$\bullet$] \underline{Calculation of $J_3(h,s)$} : repeating the subdivision done in the case $x_i=b$, we get $$ {J_3(h,s)=2\sigma_1\big[H_3(h, s) + H_4(h, s)\big] }.$$
 \end{itemize}
 Therefore, in the case $ {x_{i+1}=b}$, we have
% \begin{tcolorbox}[colback=blue!10, colframe=blue!70!black, boxrule=0.5pt, arc=50pt]
 	\[
 	{\mathbb{B}_{ii} = (\sigma_1 + \sigma_3)\big[H_1(h, s) + H_2(h, s)\big] + 2\sigma_1\big[H_3(h, s) + H_4(h, s)\big]}.  
 	\]
% \end{tcolorbox}
 \subsubsection{When $x_{i-1}=b$} 
 In a manner similar to the case $x_{i+1}=b$, we get
% \begin{tcolorbox}[colback=blue!10, colframe=blue!70!black, boxrule=0.5pt, arc=50pt]
 	\[
 	{\mathbb{B}_{ii} = (\sigma_2 + \sigma_3)\big[H_1(h, s) + H_2(h, s)\big] + 2\sigma_2\big[H_3(h, s) + H_4(h, s)\big]}.
 	\]
% \end{tcolorbox}
 
 \subsubsection{When  $x_{i-1}>b$}
 \begin{itemize}
 	\item[$\bullet$] \underline{Calculation of $J_1(h,s)$} : in this case, we split $\dis\int_{-\infty}^{x_{i-1}}\ldots=\dis\int_{-\infty}^b\ldots+\dis\int_b^{x_{i-1}}\ldots$. Consequently,
 	$$
 	\begin{array}{lll}
 		J_1(h,s)&=\dis 2 \sigma_3\dis\int_{-\infty}^b\int_{x_{i-1}}^{x_{i+1}} \frac{(\phi_{i}(y))^2}{(y-x)^{1+2s}}\mathrm{d}y\mathrm{d}x+2 \sigma_2\dis\int_b^{x_{i-1}}\int_{x_{i-1}}^{x_{i+1}} \frac{(\phi_{i}(y))^2}{(y-x)^{1+2s}}\mathrm{d}y\mathrm{d}x
 		\\\\
 		&=\dis \dfrac{\sigma_3}{s} \int_{x_{i-1}}^{x_{i+1}} \frac{(\phi_{i}(y))^2}{(y-b)^{2s}}\mathrm{d}y+\dfrac{\sigma_2}{s} \int_{x_{i-1}}^{x_{i+1}}\left[\frac{1}{(y-x_{i-1})^{2s}}-\frac{1}{(y-b)^{2s}}\right](\phi_{i}(y))^2\mathrm{d}y
 		\\\\
 		&=\dis\dfrac{\sigma_2}{s} \int_{x_{i-1}}^{x_{i+1}}\frac{(\phi_{i}(y))^2}{(y-x_{i-1})^{2s}}\mathrm{d}y+ \dfrac{\sigma_3-\sigma_2}{s} \int_{x_{i-1}}^{x_{i+1}} \frac{(\phi_{i}(y))^2}{(y-b)^{2s}}\mathrm{d}y.
 	\end{array}
 	$$
 	After performing the change of variables $\widetilde{y}=\dfrac{y-x_i}{h}$, we get $$ {J_1(h,s)=\sigma_2\big[H_1(h, s) + H_2(h, s)\big]+(\sigma_3-\sigma_2)S_1(h,s,r)},$$
 	with $ {r=x_i-b}>h$.
 	\item[$\bullet$] \underline{Calculation of $J_2(h,s)$} : initially, the case 
 	$x_{i-1}=b$
 	was treated as $x_{i}>b$. However, due to the need to handle various cases separately, we were required to split these situations further. Thus, we obtain $ {J_2(h,s)=\sigma_2\big[H_1(h, s) + H_2(h, s)\big]}$.
 	\item[$\bullet$] \underline{Calculation of $J_3(h,s)$} :  following the same calculations done in the previous cases, we obtain
 	$$ {J_3(h,s)=2\sigma_2\big[H_3(h, s) + H_4(h, s)\big] }.$$
 \end{itemize}
 Thence, in the case $ {x_{i-1}>b}$, we find
% \begin{tcolorbox}[colback=blue!10, colframe=blue!70!black, boxrule=0.5pt, arc=50pt]
 	\[	 {\mathbb{B}_{ii} = 2\sigma_2 \big[H_1(h, s) + H_2(h, s)+H_3(h, s) + H_4(h, s)\big] +(\sigma_3-\sigma_2)S_1(h,s,r)},
 	\] \ \ with $ {r=x_i-b}>h$.
% \end{tcolorbox}
 \subsubsection{When  $x_{i+1}<b$}
 Following the same reasoning as for the case  $x_{i-1}>b$, we obtain
% \begin{tcolorbox}[colback=blue!10, colframe=blue!70!black, boxrule=0.5pt, arc=50pt]
 	\[	 {\mathbb{B}_{ii} = 2\sigma_1 \big[H_1(h, s) + H_2(h, s)+H_3(h, s) + H_4(h, s)\big] +(\sigma_3-\sigma_1)S_1(h,s,t)},
 	\]   \ \ with $ {t=b-x_i}>h$.
% \end{tcolorbox}
 \subsection{Second case : upper diagonal elements $\mathbb{B}_{i,i+1}$}
 We have
 $$
 \begin{array}{llll}
 	\mathbb{B}_{i,i+1}&\dis= \int_{-\infty}^{+\infty}	 \int_{-\infty}^{+\infty} \underline{\sigma}(x,y) \frac{(\phi_{i}(x)-\phi_{i}(y))(\phi_{i+1}(x)-\phi_{i+1}(y))}{|x-y|^{1+2s}}\mathrm{d}y\mathrm{d}x
 	\\\\ &\dis
 	=\int_{-\infty}^{+\infty}	 \int_{-\infty}^{+\infty} g_{i,i+1}(x,y)\mathrm{d}y\mathrm{d}x, \quad i=1,...,N_h.
 \end{array}
 $$
 To simplify the computations, let us subdivide the domain of integration:
 $$
 \begin{array}{llll}
 	\dis
 	\int_{-\infty}^{+\infty}\int_{-\infty}^{+\infty}g_{i,i+1}(x,y)\mathrm{d}y\mathrm{d}x&\dis=\overset{\hspace*{3.75cm}=0}{\cancel{\int_{-\infty}^{x_{i}}\int_{-\infty}^{x_{i}}g_{i,i+1}(x,y)\mathrm{d}y\mathrm{d}x}}+\overbrace{2\int_{x_i}^{x_{i+1}}\int_{-\infty}^{x_{i}}g_{i,i+1}(x,y)\mathrm{d}y\mathrm{d}x}^{K_1(h,s)}
 	\\\\
 	&\dis +\overbrace{2\int_{x_{i+1}}^{+\infty}\int_{-\infty}^{x_i}g_{i,i+1}(x,y)\mathrm{d}y\mathrm{d}x}^{K_2(h,s)} +\overbrace{\int_{x_{i}}^{x_{i+1}}\int_{x_{i}}^{x_{i+1}}g_{i,i+1}(x,y)\mathrm{d}y\mathrm{d}x}^{K_3(h,s)}
 	\\\\
 	&\dis+\overbrace{2\int_{x_{i+1}}^{+\infty}\int_{x_{i}}^{x_{i+1}}g_{i,i+1}(x,y)\mathrm{d}y\mathrm{d}x}^{K_4(h,s)}+\overset{\hspace*{3.75cm}=0}{\cancel{\int_{x_{i+1}}^{+\infty}\int_{x_{i+1}}^{+\infty}g_{i,i+1}(x,y)\mathrm{d}y\mathrm{d}x}}.
 \end{array}
 $$
 In order, to compute $K_1$, $K_2$, $K_3$ and $K_4$, we need to distinguish four cases ; $x_i=b$, $x_{i+1}=b$,  $x_{i}>b$ and   $x_{i+1}<b$. To do so, let us point out some simplifications :
 \begin{enumerate}
 	\item[(i)]  $ \begin{aligned}[t]
 		K_1(h,s)&=2\int_{x_i}^{x_{i+1}}	 \int_{-\infty}^{x_i} \underline{\sigma}(x,y) \frac{(\phi_{i}(x)-\phi_{i}(y))(\phi_{i+1}(x)-\overset{\hspace*{.75cm}=0}{\cancel{\phi_{i+1}(y)}})}{|x-y|^{1+2s}}\mathrm{d}y\mathrm{d}x
 		\\\\
 		&= \underbrace{2\int_{x_i}^{x_{i+1}}\phi_{i}(x)	\phi_{i+1}(x) \int_{-\infty}^{x_i}  \underline{\sigma}(x,y)\frac{\mathrm{d}y\mathrm{d}x}{(x-y)^{1+2s}}}_{K_{11}(h,s)}\underbrace{-2\int_{x_i}^{x_{i+1}}	\int_{x_{i-1}}^{x_i}  \underline{\sigma}(x,y)\frac{\phi_{i+1}(x) \phi_{i}(y)}{(x-y)^{1+2s}}\mathrm{d}y\mathrm{d}x}_{K_{12}(h,s)}.
 	\end{aligned} $
 	\item[(ii)] $ \begin{aligned}[t]
 		K_2(h,s)&=2\int_{x_{i+1}}^{+\infty}	 \int_{-\infty}^{x_i} \underline{\sigma}(x,y) \frac{(\overset{\hspace*{.55cm}=0}{\cancel{\phi_{i}(x)}}-\phi_{i}(y))(\phi_{i+1}(x)-\overset{\hspace*{.75cm}=0}{\cancel{\phi_{i+1}(y)}})}{|x-y|^{1+2s}}\mathrm{d}y\mathrm{d}x
 		\\\\
 		&= -2\int_{x_{i+1}}^{x_{i+2}}	\int_{x_{i-1}}^{x_i}  \underline{\sigma}(x,y)\frac{\phi_{i+1}(x) \phi_{i}(y)}{(x-y)^{1+2s}}\mathrm{d}y\mathrm{d}x.
 	\end{aligned} $
 	\item[(iii)] $
 	\begin{aligned}[t]
 		\dis	K_3(h,s)&=\int_{x_i}^{x_{i+1}}	 \int_{x_i}^{x_{i+1}}	 \underline{\sigma}(x,y) \frac{(\phi_{i}(x)-\phi_{i}(y))(\phi_{i+1}(x)-\phi_{i+1}(y))}{|x-y|^{1+2s}}\mathrm{d}y\mathrm{d}x
 		\\\\
 		&= -\dfrac{1}{h^2}\int_{x_i}^{x_{i+1}}	 \int_{x_i}^{x_{i+1}}  \underline{\sigma}(x,y) |x-y|^{1-2s}\mathrm{d}y\mathrm{d}x,
 	\end{aligned} 
 	$ 
 	\\
 	as $\phi_i(x)=1-\dfrac{x-x_i}{h}$  and $\phi_{i+1}(x)=1-\dfrac{x_{i+1}-x}{h}$ for any $x\in(x_i,x_{i+1})$.
 	\item[(iv)]  $ \begin{aligned}[t]
 		K_4(h,s)&=\dis2\int_{x_i}^{x_{i+1}}	 \int_{x_{i+1}}^{+\infty} \underline{\sigma}(x,y) \frac{(\phi_{i}(x)-\overset{\hspace*{.55cm}=0}{\cancel{\phi_{i}(y)}})(\phi_{i+1}(x)-\phi_{i+1}(y))}{|x-y|^{1+2s}}\mathrm{d}y\mathrm{d}x
 		\\\\
 		&= \underbrace{2\int_{x_i}^{x_{i+1}}\phi_{i}(x)	\phi_{i+1}(x)  \int_{x_{i+1}}^{+\infty}  \underline{\sigma}(x,y)\frac{\mathrm{d}y\mathrm{d}x}{(y-x)^{1+2s}}}_{K_{41}(h,s)}\underbrace{-2\int_{x_i}^{x_{i+1}}	\int_{x_{i+1}}^{x_{i+2}}  \underline{\sigma}(x,y)\frac{\phi_{i}(x) \phi_{i+1}(y)}{(y-x)^{1+2s}}\mathrm{d}y\mathrm{d}x}_{K_{42}(h,s)}.
 	\end{aligned} $
 \end{enumerate}
 Then,we have 
 $$
 \dis
 \int_{-\infty}^{+\infty}\int_{-\infty}^{+\infty}g_{i,i+1}(x,y)\mathrm{d}y\mathrm{d}x=K_{11}(h,s)+K_{12}(h,s)+K_{2}(h,s)+K_{3}(h,s)+K_{41}(h,s)+K_{42}(h,s).
 $$
 %{\bf\em 1) when $x_i=b$.} 
 \subsubsection{When $x_i=b$}
 \begin{itemize}
 	\item[$\bullet$] \underline{Calculation of $K_{11}(h,s)$} :
 	$$
 	\begin{array}{llllll}
 		K_{11}(h,s)&=	\dis2\sigma_3\int_{x_i}^{x_{i+1}}\phi_{i}(x)	\phi_{i+1}(x) \int_{-\infty}^{x_i}  \frac{\mathrm{d}y\mathrm{d}x}{(x-y)^{1+2s}}
 		\\\\
 		&=\dfrac{\sigma_3}{s}\dis \int_{x_i}^{x_{i+1}}\frac{\phi_{i}(x)	\phi_{i+1}(x)}{(x-x_{i})^{2s}}\mathrm{d}x
 		\\\\
 		&=\sigma_3\dfrac{h^{1-2s}}{s}\dis\int_{-1}^0\frac{(1-|1+\widetilde{x}|)(1-|\widetilde{x}|)}{(1+\widetilde{x})^{2s}}\mathrm{d}\widetilde{x}	\\\\
 		&=-\sigma_3\dfrac{h^{1-2s}}{s}\dis\int_{-1}^0\frac{\widetilde{x}}{(1+\widetilde{x})^{-1+2s}}\mathrm{d}\widetilde{x},
 	\end{array}
 	$$
 	where we have performed the change of variables $\widetilde{x}=\dfrac{x-x_{i+1}}{h}$. Thus,
 	$ {K_{11}(h,s)=\sigma_3H_5(h, s)}$.
 	
 	\item[$\bullet$] \underline{Calculation of $K_{12}(h,s)$} : 
 	$$
 	\begin{array}{llllll}
 		K_{12}(h,s)&=	\dis-2\sigma_3\int_{x_i}^{x_{i+1}} \int_{x_{i-1}}^{x_i}  \frac{	\phi_{i+1}(x)\phi_{i}(y)}{(x-y)^{1+2s}}\mathrm{d}y\mathrm{d}x
 		\\\\
 		&=\dis-2\sigma_3  h^{1-2s} \dis\int_{-1}^0 \int_{-1}^0 \frac{(1+\widetilde{x})(1+\widetilde{y})}{(1+\widetilde{x}-\widetilde{y})^{1+2s}}\mathrm{d}\widetilde{y}\mathrm{d}\widetilde{x},
 	\end{array}
 	$$
 	with the change of variables $\widetilde{x}=\dfrac{x-x_{i+1}}{h}$ and $\widetilde{y}=\dfrac{y-x_{i}}{h}$. Therefore, $ {K_{12}(h,s)=\sigma_3H_6(h, s)}$.
 	\item[$\bullet$] \underline{Calculation of $K_{2}(h,s)$} : 
 	$$
 	\begin{array}{llllll}
 		K_{2}(h,s)&=	\dis-2\sigma_3\int_{x_{i+1}}^{x_{i+2}}	\int_{x_{i-1}}^{x_i}  \frac{\phi_{i+1}(x) \phi_{i}(y)}{(x-y)^{1+2s}}\mathrm{d}y\mathrm{d}x
 		\\\\
 		&=\dis-2\sigma_3h^{1-2s} \dis\int_{0}^1 \int_{-1}^0 \frac{(1-\widetilde{x})(1+\widetilde{y})}{(1+\widetilde{x}-\widetilde{y})^{1+2s}}\mathrm{d}\widetilde{y}\mathrm{d}\widetilde{x},
 	\end{array}
 	$$
 	by using the same change of variables as earlier. Thus, we get  $ {K_{2}(h,s)=\sigma_3H_7(h, s)}$.
 	\item[$\bullet$] \underline{Calculation of $K_{3}(h,s)$} : 
 	$$
 	\begin{array}{llllll}
 		K_{3}(h,s)&=	\dis-\dfrac{ \sigma_2}{h^2} \int_{x_{i}}^{x_{i+1}}	\int_{x_{i}}^{x_{i+1}}  |x-y|^{1-2s}\mathrm{d}y\mathrm{d}x
 		\\\\
 		&= \dis-\dfrac{\sigma_2}{2h^2(1-s)}\dis\int_{x_{i}}^{x_{i+1}}[|x-x_{i+1}|^{2-2s}-|x-x_i|^{2-2s}]\mathrm{d}x
 		\\\\
 		&=\dis-\dfrac{\sigma_2}{2h^2(1-s)}\dis\int_{0}^h[\widehat{x}^{2-2s}-(\widehat{x}-h)^{2-2s}]\mathrm{d}\widehat{x} ,
 	\end{array}
 	$$ 
 	where we have performed the change of variables $\widehat{x}=x_{i+1}-x$. Then, $ {	K_{3}(h,s)=-\sigma_2H_3(h,s)}$.
 	\item[$\bullet$] \underline{Calculation of $K_{41}(h,s)$} : 
 	$$
 	\begin{array}{llllll}
 		K_{41}(h,s)&\dis=2\sigma_2\int_{x_i}^{x_{i+1}}\phi_{i}(x)	\phi_{i+1}(x)  \int_{x_{i+1}}^{+\infty}  \frac{\mathrm{d}y\mathrm{d}x}{(y-x)^{1+2s}}
 		\\\\
 		&\dis= 2\sigma_2\int_{x_i}^{x_{i+1}}\phi_{i}(x)	\phi_{i+1}(x) (x_{i+1}-x)^{-2s} {\mathrm{d}x}
 		\\\\
 		&\dis=\sigma_2\dfrac{h^{1-2s}}{s}\dis\int_{0}^1\frac{(1-\widetilde{x})(1-(1-\widetilde{x}))}{|1-\widetilde{x}|^{2s}}\mathrm{d}\widetilde{x}
 		\\\\
 		&\dis=\sigma_2\dfrac{h^{1-2s}}{s}\dis\int_{0}^1\frac{\widetilde{x}}{|1-\widetilde{x}|^{-1+2s}}\mathrm{d}\widetilde{x},
 	\end{array}
 	$$
 	with the change of variables $\widetilde{x}=\dfrac{x-x_i}{h}$. Then, $ {K_{41}(h,s)=\sigma_2H_{5}(h,s)}$.
 	\item[$\bullet$] \underline{Calculation of $K_{42}(h,s)$} : considering the following change of variables : $\widetilde{x}=\dfrac{x-x_i}{h}$ and $\widehat{y}=\dfrac{y-x_{i+1}}{h}$, we get
 	$$
 	\begin{array}{llllll}
 		K_{42}(h,s)&=\dis -2\sigma_2\int_{x_i}^{x_{i+1}}	\int_{x_{i+1}}^{x_{i+2}}  \frac{\phi_{i}(x) \phi_{i+1}(y)}{(y-x)^{1+2s}}\mathrm{d}y\mathrm{d}x
 		\\\\&=
 		-2\sigma_2  h^{1-2s} \dis\int_{0}^1 \int_{0}^1 \frac{(1-\widetilde{x})(1-\widetilde{y})}{(1-\widetilde{x}+\widehat{y})^{1+2s}}\mathrm{d}\widehat{y}\mathrm{d}\widetilde{x},
 	\end{array}
 	$$
 	which implies that $ {K_{42}(h,s)=\sigma_2H_{6}(h,s)}$.
 \end{itemize}
 Therefore, in the case $ {x_i=b}$, we obtain
 %\begin{center}
 %	\boxed{$$ {b_{ii}=(\sigma_1+\sigma_2)[H_1(h,s)+H_3(h,s)]+2\sigma_3[H_2(h,s)+H_4(h,s)]}.$$}
 %\end{center}
% \begin{tcolorbox}[colback=blue!10, colframe=blue!70!black, boxrule=0.5pt, arc=50pt]
 	%, title=Final Result
 	\[
 	{\mathbb{B}_{i,i+1} =-\sigma_2H_3(h, s) +(\sigma_2 + \sigma_3)\big[H_5(h, s) + H_6(h, s)\big] + \sigma_3H_7(h, s)}.  
 	\]
% \end{tcolorbox}
 \subsubsection{When $x_{i+1}=b$}
 In a way similar to the case $x_{i}=b$, we get
% \begin{tcolorbox}[colback=blue!10, colframe=blue!70!black, boxrule=0.5pt, arc=50pt]
 	\[
 	{\mathbb{B}_{i,i+1} =-\sigma_1H_3(h, s) +(\sigma_1 + \sigma_3)\big[H_5(h, s) + H_6(h, s)\big] + \sigma_3H_7(h, s)}.
 	\]
% \end{tcolorbox}
 
 \subsubsection{When $x_i>b$}
 \begin{itemize}
 	\item[$\bullet$] \underline{Calculation of $K_{11}(h,s)$} :
 	$$
 	\begin{array}{llllll}
 		K_{11}(h,s)&=	\dis2\int_{x_i}^{x_{i+1}}\phi_{i}(x)	\phi_{i+1}(x) \int_{-\infty}^{x_i}  \frac{\underline{\sigma}(x,y)}{(x-y)^{1+2s}}\mathrm{d}y\mathrm{d}x
 		\\\\
 		&=	\dis2\sigma_3\int_{x_i}^{x_{i+1}}\phi_{i}(x)	\phi_{i+1}(x) \int_{-\infty}^{b}  \frac{\mathrm{d}y\mathrm{d}x}{(x-y)^{1+2s}}+2\sigma_2\int_{x_i}^{x_{i+1}}\phi_{i}(x)	\phi_{i+1}(x) \int_{b}^{x_i}  \frac{\mathrm{d}y\mathrm{d}x}{(x-y)^{1+2s}}
 		\\\\
 		&=\dfrac{\sigma_2}{s}\dis \int_{x_i}^{x_{i+1}}\frac{\phi_{i}(x)	\phi_{i+1}(x)}{(x-x_{i})^{2s}}\mathrm{d}x+\dfrac{\sigma_3-\sigma_2}{s}\dis \int_{x_i}^{x_{i+1}}\frac{\phi_{i}(x)	\phi_{i+1}(x)}{(x-b)^{2s}}\mathrm{d}x
 		\\\\
 		&=-\sigma_2\dfrac{h^{1-2s}}{s}\dis\int_{-1}^0\frac{\widetilde{x}}{(1+\widetilde{x})^{-1+2s}}\mathrm{d}\widetilde{x}-\dfrac{\sigma_3-\sigma_2}{s} h \dis\int_{-1}^0\frac{\widetilde{x}(1+\widetilde{x})}{(h\widetilde{x}+\rho)^{2s}}\mathrm{d}\widetilde{x},
 	\end{array}
 	$$
 	where $\widetilde{x}=\dfrac{x-x_{i+1}}{h}$ and $\rho=x_{i+1}-b$. Thus,
 	$ {K_{11}(h,s)=\sigma_2H_5(h, s)-(\sigma_3-\sigma_2) S_2(h,s,\rho)}$.
 	\item[$\bullet$] \underline{Calculation of $K_{12}(h,s)$, $K_{2}(h,s)$, $K_{3}(h,s)$, $K_{41}(h,s)$ and $K_{42}(h,s)$} : Following the case $x_i=b$, we easily get
 	\begin{itemize}
 		\item[---]  $ {K_{12}(h,s)=\sigma_2H_6(h, s)}$ ; 
 		\item[---]  $ {K_{2}(h,s)=\sigma_2H_7(h, s)}$ ;
 		\item[---] $ {	K_{3}(h,s)=-\sigma_2H_3(h,s)}$ ;
 		\item[---] $ {K_{41}(h,s)=\sigma_2H_{5}(h,s)}$ ;
 		\item[---] $ {K_{42}(h,s)=\sigma_2H_{6}(h,s)}$.
 	\end{itemize}
 \end{itemize}	
 So, in the case $ {x_i>b}$, we have
 %\begin{center}
 %	\boxed{$$ {b_{ii}=(\sigma_1+\sigma_2)[H_1(h,s)+H_3(h,s)]+2\sigma_3[H_2(h,s)+H_4(h,s)]}.$$}
 %\end{center}
% \begin{tcolorbox}[colback=blue!10, colframe=blue!70!black, boxrule=0.5pt, arc=50pt]
 	%, title=Final Result
 	\[
 	{\mathbb{B}_{i,i+1} =\sigma_2[-H_3(h, s) +2H_5(h, s) + 2H_6(h, s)+H_7(h, s)]-(\sigma_3-\sigma_2)S_2(h,s,\rho)},  
 	\] \ \ with $ {\rho=x_{i+1}-b}>h$.
% \end{tcolorbox}
 \subsubsection{When $x_{i+1}<b$} Following the previous steps, we obtain
% \begin{tcolorbox}[colback=blue!10, colframe=blue!70!black, boxrule=0.5pt, arc=50pt]
 	%, title=Final Result
 	\[
 	{\mathbb{B}_{i,i+1} =\sigma_1[-H_3(h, s) +2H_5(h, s) + 2H_6(h, s)+H_7(h, s)]-(\sigma_3-\sigma_1)S_2(h,s,t)},  
 	\] \ \ with $ {t=b-x_{i}}>h$.
% \end{tcolorbox}
 \subsection{Third case : superdiagonal elements $\mathbb{B}_{ij}$ for $j\geq i+2$}
 We have
 $$
 \begin{array}{llll}
 	\mathbb{B}_{ij}&\dis= \int_{-\infty}^{+\infty}	 \int_{-\infty}^{+\infty} \underline{\sigma}(x,y) \frac{(\phi_{i}(x)-\phi_{i}(y))(\phi_{j}(x)-\phi_{j}(y))}{|x-y|^{1+2s}}\mathrm{d}y\mathrm{d}x
 	\\\\ &\dis
 	=	 \int_{x_{j-1}}^{x_{j+1}}	 \int_{x_{i-1}}^{x_{i+1}} \underline{\sigma}(x,y) \frac{(\phi_{i}(x)-\phi_{i}(y))(\phi_{j}(x)-\phi_{j}(y))}{|x-y|^{1+2s}}\mathrm{d}y\mathrm{d}x
 	\\\\ &\dis
 	=	 \int_{x_{j-1}}^{x_{j+1}}	 \int_{x_{i-1}}^{x_{i+1}} \underline{\sigma}(x,y) \frac{\phi_{i}(y)\phi_{j}(x)}{|x-y|^{1+2s}}\mathrm{d}y\mathrm{d}x,
 \end{array}
 $$
 as 
 $\operatorname{supp} \varphi_i\cap \operatorname{supp} \varphi_j=\emptyset$ ($x$ and $y$ should exist in different intervals). 
 \subsubsection{When $x_i=b$}
 With the change of variables  $\widetilde{x}=\dfrac{x-x_i}{h}$ and  $\widetilde{y}=\dfrac{y-x_i}{h}$, we get
 $$
 \begin{array}{llll}
 	\mathbb{B}_{ij}&\dis	=	\sigma_3 \int_{x_{j-1}}^{x_{j+1}}	 \int_{x_{i-1}}^{x_{i}} \frac{\phi_{i}(y)\phi_{j}(x)}{|x-y|^{1+2s}}\mathrm{d}y\mathrm{d}x+\sigma_2 \int_{x_{j-1}}^{x_{j+1}}	 \int_{x_{i}}^{x_{i+1}} \frac{\phi_{i}(y)\phi_{j}(x)}{|x-y|^{1+2s}}\mathrm{d}y\mathrm{d}x
 	\\\\ &\dis
 	=-2	\sigma_3 h^{1-2s}\int_{-1}^1\int_{-1}^0\frac{(1-|\widetilde{x}|)(1+\widetilde{y})}{(\widetilde{x}-\widetilde{y}+j-i)^{1+2s}}\mathrm{d}y\mathrm{d}x-2	\sigma_2 h^{1-2s}\int_{-1}^1\int_{0}^1\frac{(1-|\widetilde{x}|)(1-\widetilde{y})}{(\widetilde{x}-\widetilde{y}+j-i)^{1+2s}}\mathrm{d}y\mathrm{d}x
 	\\\\ &\dis
 	=\sigma_3 {L_1(h,s,k)}+\sigma_2 {L_2(h,s,k)},
 \end{array}
 $$
 where $ {k=j-i}\geq 2$. So, in this case
% \begin{tcolorbox}[colback=blue!10, colframe=blue!70!black, boxrule=0.5pt, arc=50pt]
 	\[
 	{\mathbb{B}_{ij}} =\sigma_3 {L_1(h,s,k)}+\sigma_2 {L_2(h,s,k)},
 	\]
 	\ \	with $	 {k=j-i}\geq 2$.
% \end{tcolorbox}
 \subsubsection{When $x_j=b$} Similarly, we obtain
% \begin{tcolorbox}[colback=blue!10, colframe=blue!70!black, boxrule=0.5pt, arc=50pt]
 	\[
 	{\mathbb{B}_{ij}} =\sigma_3 {L_1(h,s,k)}+\sigma_1 {L_2(h,s,k)},
 	\]
 	\ \	with $	 {k=j-i}\geq 2$.
% \end{tcolorbox}
 \subsubsection{When $x_i,x_j\neq b$} Let us denote $$\widehat{\sigma}_{ij}:=\left\{\begin{array}{llll}
 	\sigma_1&\text{if}& x_i,x_j<b,
 	\\
 	\sigma_2&\text{if}&x_i,x_j>b,
 	\\
 	\sigma_3&\text{if}& \text{otherwise}.
 \end{array}   \right.$$
 In this case,
 $$
 \begin{array}{llll}
 	\mathbb{B}_{ij} &\dis
 	=-2	\widehat{\sigma}_{ij} h^{1-2s}\int_{-1}^1\int_{-1}^1\frac{(1-|\widetilde{x}|)(1-|\widetilde{y}|)}{(\widetilde{x}-\widetilde{y}+j-i)^{1+2s}}\mathrm{d}y\mathrm{d}x
 	\\\\ &\dis
 	=\widehat{\sigma}_{ij}[ {L_1(h,s,k)}+ {L_2(h,s,k)}],
 \end{array}
 $$
 where $ {k=j-i}\geq 2$. Therefore,
% \begin{tcolorbox}[colback=blue!10, colframe=blue!70!black, boxrule=0.5pt, arc=50pt]
 	\[
 	{\mathbb{B}_{ij}} =\widehat{\sigma}_{ij}[ {L_1(h,s,k)}+ {L_2(h,s,k)}],
 	\]
 	\ \ with $	 {k=j-i}\geq 2$.
% \end{tcolorbox}

	\end{document}